\documentclass[11pt,letterpaper,reqno]{amsart}
\usepackage{amssymb}
\usepackage{amsmath}
\usepackage{amsthm}
\usepackage{amsfonts}
\usepackage{bbm}
\usepackage{enumitem} 
\usepackage{booktabs}
\usepackage{graphicx}
\usepackage{xcolor}
\usepackage[T1]{fontenc}
\usepackage{doi}
\usepackage{comment} 
\usepackage{hyperref}
\hypersetup{pdfstartview={FitH}}
\usepackage{bookmark}
\usepackage[capitalize,noabbrev]{cleveref}

\newtheorem{theorem}{Theorem}[section]
\newtheorem{lemma}[theorem]{Lemma}
\newtheorem{proposition}[theorem]{Proposition}
\newtheorem{corollary}[theorem]{Corollary}

\newtheorem{question}[theorem]{Question}

\theoremstyle{definition}
\newtheorem{example}[theorem]{Example}
\newtheorem{remark}[theorem]{Remark}
\newtheorem{definition}[theorem]{Definition}

\newtheorem*{proposition*}{Proposition}
\newtheorem*{theorem*}{Theorem}
\newtheorem*{corollary*}{Corollary}

\DeclareMathOperator{\tr}{tr}
\DeclareMathOperator{\spec}{spec}

\newcommand{\dd}{\,d}
\newcommand{\one}{\mathbf 1}
\newcommand{\norm}[1]{\left\lVert #1\right\rVert}
\newcommand{\QQ}{\mathbb{Q}}

\newcommand{\ppar}[1]{\par\addvspace{\medskipamount}\noindent{\bfseries #1.}\enspace\ignorespaces}
\newcommand{\cF}{\mathcal{F}}

\begin{document}

\title{Path-Minimality of $p$-Energy for Connected Graphs}
\author[Yinchen Liu]{Yinchen Liu}
\author[Quanyu Tang]{Quanyu Tang}

\address{Institute for Interdisciplinary Information Sciences, Tsinghua University, Beijing 100084, P. R. China}
\email{liuyinch23@mails.tsinghua.edu.cn}

\address{School of Mathematics and Statistics, Xi'an Jiaotong University, Xi'an 710049, P. R. China}
\email{tangquanyu827@gmail.com}



\begin{abstract}
Let $G$ be a simple connected graph on $n$ vertices, and let $\lambda_1(G),\lambda_2(G),\ldots,\lambda_n(G)$ be the eigenvalues of its adjacency matrix $A(G)$. For $p>0$, define the $p$-energy of $G$ by $\mathcal E_p(G)=\sum_{i=1}^n |\lambda_i(G)|^p$. We prove that, for every real number $p\ge 2$ and every simple connected graph $G$ on $n$ vertices,
$$
\mathcal E_p(G)\ge \mathcal E_p(P_n),
$$
where $P_n$ denotes the path on $n$ vertices. Moreover, for each fixed $p>2$, equality holds if and only if $G\cong P_n$. Together with the previously known star-minimality results, this completes the solution of two questions of Nikiforov. 

The proof combines two different comparison principles. For $2<p<4$, we use a bipartite reduction, a Mellin representation of fractional powers, and a determinant comparison involving matching generating polynomials and tree shifts. For $p\ge4$, we prove a second-order stop-loss comparison for the squared singular values of bipartite graphs. This comparison is established by rank-one spectral-shift estimates, deletion-minimal counterexamples, and a finite certified analysis of the terminal sparse-sun configurations.
\end{abstract}

\maketitle

\tableofcontents

\section{Introduction}

\subsection{Background and main result}\label{sec:bamr}

Throughout this paper, all graphs are finite, simple, undirected, and unweighted.  Let \(G\) be a graph with vertex set \(V(G)=\{v_1,\ldots,v_n\}\), and let \(A(G)\) be its \(n\times n\) adjacency matrix, whose \((i,j)\)-entry is \(1\) if \(v_iv_j\in E(G)\) and is \(0\) otherwise.  We write
\[
        \lambda_1(G),\lambda_2(G),\ldots,\lambda_n(G)
\]
for the eigenvalues of \(A(G)\), counted with multiplicity. For a real number \(x\), we write \(x_+:=\max\{x,0\}\) for its positive part. For \(p>0\) we define the \emph{\(p\)-energy} of \(G\) by
\[
        \mathcal E_p(G):=\sum_{i=1}^n |\lambda_i(G)|^p .
\]
For \(p\ge1\), the \emph{Schatten \(p\)-norm} of a matrix \(M\) is
\[
        \norm{M}_{S_p}:=
        \left(\sum_j s_j(M)^p\right)^{1/p},
\]
where \(s_j(M)\) are the singular values of \(M\).  Since \(A(G)\) is real symmetric, its singular values are \(|\lambda_1(G)|,\ldots,|\lambda_n(G)|\). Thus, for \(p\ge1\), we have $\mathcal E_p(G)=\norm{A(G)}_{S_p}^p$. Following Nikiforov's notation, one may write \(\norm{G}_p\) for \(\norm{A(G)}_{S_p}\); after taking \(p\)th roots, inequalities for \(\mathcal E_p(G)\) with \(p\ge1\) are exactly inequalities for these graph norms~\cite{Nikiforov2016}.

Let $P_n$ denote the path graph and $S_n$ the star graph on $n$ vertices. Nikiforov asked whether paths and stars remain extremal for Schatten norms in the following sense~\cite{Nikiforov2016}:

\begin{question}[{\cite[Question~4.51]{Nikiforov2016}}]\label{q1}
Let \( G \) be a connected graph of order \( n \).
\begin{itemize}
    \item[(a)] Is it true that \( \|G\|_{p} \geq \|P_n\|_{p} \) for every \( p > 2 \)?
    \item[(b)] Is it true that \( \|G\|_{p} \geq \|S_n\|_{p} \) for every \( 1 < p < 2 \)?
\end{itemize}
\end{question}

\begin{question}[{\cite[Question~4.52]{Nikiforov2016}}]\label{q2}
Let \( T \) be a tree of order \( n \).
\begin{itemize}
    \item[(a)] Is it true that \( \|S_n\|_{p} \geq \|T\|_{p} \geq \|P_n\|_{p} \) for every \( p > 2 \)?
    \item[(b)] Is it true that \( \|P_n\|_{p} \geq \|T\|_{p} \geq \|S_n\|_{p} \) for every \( 1 < p < 2 \)?
\end{itemize}
\end{question}
These questions extend the classical extremal behavior of paths and stars for
ordinary graph energy and for the spectral radius.

Several parts of this picture were known. If $p=2k$ is a positive even integer, then
\[
\mathcal E_{2k}(G)=\tr A(G)^{2k},
\]
which counts closed walks of length $2k$. Csikv\'ari proved that among connected graphs of fixed order the path minimizes the number of closed walks of every length; in Schatten-norm language this gives $\norm{G}_{2k}\ge \norm{P_n}_{2k}$ for every connected $G$~\cite{Csikvari2010,Nikiforov2016}.  On the tree side, Arizmendi and Arizmendi proved the star upper bound $\mathcal E_p(T)\le \mathcal E_p(S_n)$ for $p>2$ by a bipartite Schatten-energy estimate \cite{ArizmendiArizmendi2023}. For the opposite range, Arizmendi and Guerrero proved that for every tree $T$ on $n$ vertices and every $0<p<2$,
\[
\mathcal E_p(S_n)\le \mathcal E_p(T)\le \mathcal E_p(P_n),
\]
using Coulson-type integral formulas and the tree quasi-order~\cite{ArizmendiGuerrero2023}. Tang, Liu and Wang later settled the connected-graph star-minimality problem in the stronger form $\mathcal E_p(G)\ge \mathcal E_p(S_n)$ for every connected $G$ and every $0<p<2$, with equality only for the star \cite{TangLiuWang2025Nikiforov}.

Thus the remaining issue was path-minimality for real exponents $p>2$. In the notation above, the connected-graph version asks whether $\mathcal E_p(G)\ge \mathcal E_p(P_n)$ for every connected $n$-vertex graph $G$, while the tree version asks for the same inequality when $G$ is restricted to be a tree.  The connected-graph version clearly implies the tree version.  We shall refer to these two questions as Remaining Problem A and Remaining Problem B, respectively:

\medskip

\noindent
\textbf{Remaining Problem A.}
Let \(G\) be a connected graph of order \(n\).  Is it true that, for every real number \(p>2\), \(\mathcal E_p(G)\ge \mathcal E_p(P_n)\)?

\medskip

\noindent
\textbf{Remaining Problem B.}
Let \(T\) be a tree of order \(n\).  Is it true that, for every real number \(p>2\), \(\mathcal E_p(T)\ge \mathcal E_p(P_n)\)?

\medskip

The limiting case $p=\infty$ is consistent with this expectation: Lov\'asz and Pelik\'an proved that the path minimizes the spectral radius among trees of fixed order, and the same conclusion among connected graphs follows by taking a spanning tree and using monotonicity of the spectral radius under edge addition~\cite{LovaszPelikan1973}.

The purpose of this paper is to resolve both remaining problems. Our main result is the following theorem.

\begin{theorem}\label{thm:pgeq2}
Let $G$ be a connected graph on $n$ vertices.  Then for every real number
$p\geq 2$,
\[
   \mathcal E_p(G)\ge \mathcal E_p(P_n).
\]
Moreover, for each fixed \(p>2\), equality holds if and only if $G\cong P_n$.
\end{theorem}

Theorem~\ref{thm:pgeq2} proves Remaining Problem A, and therefore also proves Remaining Problem B. We emphasize that the proof is not a direct interpolation from even exponents. Instead, it is based on two comparison principles: an integrated resolvent comparison for the range \(2<p<4\), and an admissible-function comparison for the range \(p\ge4\), the latter being proved via a second-order stop-loss comparison for the squared singular values. These two principles are introduced in the proof outline below.

\subsection{Proof outline}
We first provide a standard reduction to the bipartite case, and then give an outline of the two comparison principles and their proofs.
\ppar{Reduction to bipartite graphs}
We reduce the problem from arbitrary connected graphs to connected
bipartite spanning subgraphs.  Given a spanning tree of \(G\), its bipartition
defines a cut.  Keeping only the edges crossing this cut gives a connected
bipartite spanning subgraph \(H\).  At the matrix level, \(A(H)\) is the
arithmetic mean of \(A(G)\) and \(-D A(G)D\) for a signature matrix \(D\).
Hence, for every even convex function \(F\), the trace spectral functional
\(\sum_i F(\lambda_i)\) decreases under this operation.  In particular,
\(\mathcal E_p(H)\le \mathcal E_p(G)\) for \(p\ge1\).  It is therefore enough
to prove the required lower bound for connected bipartite graphs.

For a connected bipartite graph \(G\), write \(\mu_i(G)\) for the squared
singular values of a biadjacency matrix, so that
\(\mathcal E_p(G)=2\sum_i \mu_i(G)^{p/2}\).  The two ranges \(2<p<4\) and
\(p\ge4\) are handled by different comparison principles.

\subsubsection{The case \texorpdfstring{$2<p<4$}{2<p<4}}
We set \(\alpha=p/2\in(1,2)\) and introduce the integrated resolvent
\[
        r_1(y):=\int_0^y\frac{s}{1+s}\,ds=y-\log(1+y),
        \qquad
        R_1(G;x):=\sum_i r_1(x\mu_i(G)).
\]
The Mellin identity in Lemma~\ref{lem:mellin} expresses \(u^\alpha\) as a
positive integral of \(r_1(xu)\) over \(x>0\), so a pointwise comparison
\[
        R_1(G;x)\ge R_1(P_n;x)\qquad(x>0)
\]
implies \(\mathcal{E}_p(G)\ge\mathcal{E}_p(P_n)\) for every \(p\in(2,4)\).
The advantage of passing to \(R_1\) is that it admits a clean determinant
expression.  Given the bipartition \(V(G)=X\sqcup Y\), the
\emph{biadjacency matrix} \(B_G\) is the \(|X|\times|Y|\) 0-1 matrix
whose \((i,j)\)-entry equals~1 if \(x_iy_j\in E(G)\) and 0 otherwise.
With this notation,
\[
        R_1(G;x)=x|E(G)|-\log\det(I+xB_G^{\top}B_G).
\]
For a tree \(T\), it can be shown that the matrix \(I+xB_T^{\top}B_T\) has determinant equal to the
matching generating polynomial \(M_T(x):=\sum_{k\ge0}m_k(T)\,x^k\), where
\(m_k(T)\) counts \(k\)-matchings of \(T\).  Thus the \(R_1\)-comparison for
trees is equivalent to a pointwise comparison of matching generating polynomials:
\[
        M_T(x)\le M_{P_n}(x)\qquad(x\ge0).
\]

\ppar{The one-branch tree shift}
The matching-polynomial comparison for trees is proved by using the inverse
direction of a one-branch special case of Csikv\'ari's generalized tree
shift~\cite{Csikvari2010}.  Csikv\'ari's generalized tree shift induces a
poset on trees of fixed order; in that poset the path is the unique minimal
element and the star is the unique maximal element, and a proper generalized
tree shift increases the number of closed walks of every length.  Here we use
the same local geometry in the opposite, path-directed direction.

For the present argument, the required point is the following coefficientwise matching-polynomial monotonicity.  Suppose that a tree \(T\) contains a vertex \(v\) with two pendant path arms
\[
        v u_1\cdots u_a
        \qquad\text{and}\qquad
        v w_1\cdots w_b,
        \qquad a,b\ge1.
\]
Let \(T'\) be obtained from \(T\) by deleting the edge \(vw_1\) and adding the edge \(u_aw_1\).  Thus the entire arm \(w_1\cdots w_b\) is reattached to the far end \(u_a\) of the first arm, and the two pendant arms are replaced by the single pendant path $v u_1\cdots u_a w_1\cdots w_b$; see Figure~\ref{fig:tree-transform}. Then
\[
        M_{T'}(x)=M_T(x)+x^2Q(x)\ge M_T(x)\qquad(x\ge0),
\]
where \(Q\) is a polynomial with non-negative coefficients.

Every non-path tree admits such a path-directed shift: choose a branching
vertex of maximal distance from a fixed root leaf, and choose two pendant path
components emanating from it away from the root.  This shift reduces the number
of leaves by one.  Iterating the operation therefore terminates at the unique
tree with no branching vertex, namely \(P_n\). This yields \(M_T(x)\le M_{P_n}(x)\), and hence \(R_1(T;x)\ge R_1(P_n;x)\) for every tree \(T\).

\ppar{Extension to general bipartite graphs}
The tree inequality extends to all connected bipartite graphs by induction on
\(|V(G)|\), with trees as the base.  Suppose \(G\) is not a tree; choose a
vertex \(v\) on a cycle and let \(G-v=G_1\sqcup\cdots\sqcup G_q\) with
\(n_i=|V(G_i)|\) and \(n=1+\sum_in_i\).  Since \(v\) lies on a cycle, two of
its neighbors remain connected after its removal, so \(q\le d-1\) where
\(d=d_G(v)\).  Applying the inductive hypothesis to each component and using
additivity of \(R_1\) over disjoint unions gives
\[
        R_1(G-v;x)=\sum_iR_1(G_i;x)\ge\sum_iR_1(P_{n_i};x).
\]
It therefore suffices to show that the \emph{gain},
\[
        \mathrm{gain}(v,G;x):=R_1(G;x)-R_1(G-v;x),
\]
is at least the \emph{path deficit} \(R_1(P_n;x)-\sum_iR_1(P_{n_i};x)\).
Let \(b\in\{0,1\}^{|Y|}\) be the indicator vector of \(N_G(v)\) on the right
side of the bipartition and set \(M=B_{G-v}^{\top}B_{G-v}\).  The matrix
determinant lemma applied to the rank-one update \(B_G^{\top}B_G=M+bb^{\top}\)
yields the gain formula
\[
        \mathrm{gain}(v,G;x)=xd-\log\bigl(1+xb^{\top}(I+xM)^{-1}b\bigr).
\]
Since \(M\succeq0\) implies \((I+xM)^{-1}\preceq I\), we have
\(b^{\top}(I+xM)^{-1}b\le b^{\top}b=d\), so
\(\mathrm{gain}(v,G;x)\ge xd-\log(1+xd)=r_1(dx)\).  Since \(r_1\) is
increasing and \(d\ge q+1\), this yields \(\mathrm{gain}(v,G;x)\ge r_1((q+1)x)\).

For the path deficit: a similar rank-one update argument together with the star-of-paths tree---the
tree formed by connecting \(v\) to one endpoint of each \(P_{n_i}\) as a
hub---shows \(R_1(P_n;x)\le\sum_iR_1(P_{n_i};x)+r_1((q+1)x)\).  Combining
the two estimates gives \(R_1(G;x)\ge R_1(P_n;x)\), completing the induction.

\subsubsection{The case \texorpdfstring{$p\ge4$}{p>=4}}

For \(p\ge4\), we prove a stronger comparison for a whole cone of test
functions, with the power functions as special cases.  More precisely, we work
with functions \(F\colon[0,\infty)\to\mathbb R\) satisfying $F(0)=0$, $F'(0)\ge0$ and $F''(0)\ge0$ in the sense made precise later: \(F\in C^2([0,\infty))\), \(F''\) is locally
absolutely continuous, and \(F'''\ge0\) almost everywhere on \((0,\infty)\).
We call such functions \emph{admissible third-order convex}.  The comparison
proved in this range is
\[
        \sum_{i=1}^n F\left(\lambda_i(G)^2\right)
        \ge
        \sum_{i=1}^n F\left(\lambda_i(P_n)^2\right)
\]
for every connected graph \(G\) on \(n\) vertices and every admissible
third-order convex function \(F\).  Taking \(F(u)=u^{p/2}\) gives $\mathcal E_p(G)\ge \mathcal E_p(P_n)$ for every $p\ge4$.

\ppar{Reduction to second-order stop-loss}
Fix an admissible third-order convex function \(F\).  For bipartite graphs, the comparison for \(F\) is reduced to a second-order stop-loss comparison for squared singular values. If \(\mu_i(G)\) denotes the squared singular values of a biadjacency matrix of \(G\), set
\[
        \mathsf S_t(G):=\sum_i(\mu_i(G)-t)_+^2 .
\]
Taylor's formula with integral remainder gives, for \(u\ge0\),
\[
        F(u)
        =
        F'(0)u+\frac{F''(0)}2u^2
        +\frac12\int_0^\infty F'''(t)(u-t)_+^2\,dt .
\]
Here \(F(0)=0\) has been used, and \(F'''\) is understood in the almost-everywhere sense specified in the definition of admissibility. Applying this identity to the squared singular values of \(G\) and \(P_n\), and then subtracting, gives
\[
\begin{aligned}
        \sum_i F(\mu_i(G))-\sum_i F(\mu_i(P_n))
        &=
        F'(0)\bigl(|E(G)|-(n-1)\bigr)  \\
        &\quad
        +\frac{F''(0)}2\bigl(\mathsf S_0(G)-\mathsf S_0(P_n)\bigr) \\
        &\quad
        +\frac12\int_0^\infty F'''(t)
        \bigl(\mathsf S_t(G)-\mathsf S_t(P_n)\bigr)\,dt .
\end{aligned}
\]
In the first term we used $\sum_i\mu_i(G)=|E(G)|$ and $\sum_i\mu_i(P_n)=n-1$. Since \(G\) is connected, \(|E(G)|\ge n-1\).  Thus the first term is
non-negative because \(F'(0)\ge0\).  If the stop-loss domination
\[
        \mathsf S_t(G)\ge\mathsf S_t(P_n)\qquad(t\ge0)
\]
holds, then the second term is non-negative because \(F''(0)\ge0\), and the
integral term is non-negative because \(F'''\ge0\) a.e.  Therefore the
admissible-function comparison follows from this stop-loss domination.

This domination is proved below first for trees and then, by a deletion
argument, for all connected bipartite graphs.

\ppar{The tree case}
The tree case of the stop-loss comparison again uses the one-branch tree
shift, but the argument is subtler than the corresponding \(R_1\) comparison.
For \(R_1\), coefficientwise monotonicity of the matching generating
polynomial is enough.  For the second-order stop-loss quantities, one has to
control how the roots of the matching polynomial move under the shift.

Let \(T'\) be obtained from \(T\) by deleting the old edge \(vw_1\) and adding
the new edge \(u_aw_1\).  For \(0\le\theta\le1\), let \(G_\theta\) be the
non-negatively weighted graph in which the old edge \(vw_1\) has weight
\(1-\theta\), the new edge \(u_aw_1\) has weight \(\theta\), and all other
edges have weight \(1\).  Let $M_\theta:=M_{G_\theta}$ be its weighted matching generating polynomial.  Thus $M_0=M_T$ and $M_1=M_{T'}$. A direct matching calculation gives $\partial_\theta M_\theta(x)=x^2Q(x)$, where \(Q\) is an explicit polynomial with non-negative coefficients.

The point of introducing \(M_\theta\) is that its roots can be followed
continuously.  By the Heilmann--Lieb theorem~\cite{HeilmannLieb1972}, \(M_\theta\) has only real
negative zeros for every \(0\le\theta\le1\).  Bronshtein's theorem on
hyperbolic polynomial families~\cite{Bronshtein1979} then allows us to label the corresponding
non-negative root parameters by locally Lipschitz functions
\(\mu_j(\theta)\), after adding zero parameters if necessary, so that
\[
        M_\theta(x)=\prod_j(1+x\mu_j(\theta)).
\]
For almost every \(\theta\), differentiating this factorization gives
\[
        \sum_j\frac{x\mu_j'(\theta)}{1+x\mu_j(\theta)}
        =
        \frac{x^2Q(x)}{M_\theta(x)}.
\]
Equivalently, if $\dot\nu_\theta:=\sum_j\mu_j'(\theta)\,\delta_{\mu_j(\theta)}$ is the signed speed measure of the roots, then its Stieltjes transform is $\mathcal S_{\dot\nu_\theta}(x)=xQ(x)/M_\theta(x)$.

It remains to convert the Stieltjes-transform identity into the sign of the
stop-loss derivative.  We use the Stieltjes cones introduced below: briefly,
\(\mathcal C_{\le k}\) is the cone generated by finite non-negative linear
combinations of products of at most \(k\) factors
\[
        (1+\alpha x)^{-1},\qquad \alpha\ge0,
\]
and \(\mathcal C_{\le k}^0\) denotes the subcone with \(O(x^{-1})\) decay at
infinity.  The deletion identities for the tree-shift interpolation give the
key membership
\[
        \frac{Q}{M_\theta}\in\mathcal C_{\le3}^0.
\]

For almost every \(\theta\), set $\dot\nu_\theta := \sum_j\mu_j'(\theta)\,\delta_{\mu_j(\theta)}$. Then $\mathcal S_{\dot\nu_\theta}(x)=xQ(x)/M_\theta(x)$. The analytic cone lemma proved below states that if a finite signed atomic
measure \(\sigma\) satisfies
\[
        \mathcal S_\sigma(x)=xg(x),
        \qquad g\in\mathcal C_{\le3}^0,
\]
then $\int_{[0,\infty)} (\lambda-t)_+\,d\sigma(\lambda)\le0$ for every $t\ge0$. Applying this lemma to \(\sigma=\dot\nu_\theta\) and
\(g=Q/M_\theta\), we obtain
\[
        \frac{d}{d\theta}\mathsf S_t(\theta)
        =
        \frac{d}{d\theta}\sum_j(\mu_j(\theta)-t)_+^2  
        =
        2\int_{[0,\infty)} (\lambda-t)_+\,d\dot\nu_\theta(\lambda)
        \le0 .
\]
Thus \(\mathsf S_t(\theta)\) is non-increasing along the interpolation.

At the endpoints, \(M_0=M_T\) and \(M_1=M_{T'}\) are matching polynomials of
trees, so the forest determinant identity identifies their root parameters
with the squared singular values of \(T\) and \(T'\), respectively.  Therefore $\mathsf S_t(T)\ge \mathsf S_t(T')$ for $t\ge0$. Iterating along the finite sequence of one-branch shifts from \(T\) to \(P_n\) yields $\mathsf S_t(T)\ge \mathsf S_t(P_n)$ for every tree \(T\) on \(n\) vertices.

\ppar{Deletion for general bipartite graphs}
The non-tree case uses a deletion-minimal counterexample argument.  When a
vertex \(v\) is reinserted, the rank-one trace formula and Jensen's inequality
turn the spectral update into a one-dimensional interval integral.  We call
\(v\) \emph{deletable} if the componentwise inductive hypothesis
\[
        \mathsf S_t(G-v)
        =
        \sum_i \mathsf S_t(G_i)
        \ge
        \sum_i \mathsf S_t(P_{n_i}),
        \qquad G-v=G_1\sqcup\cdots\sqcup G_q,
\]
together with the rank-one spectral-shift bound for reinserting \(v\), already
implies $\mathsf S_t(G)\ge\mathsf S_t(P_n)$. Comparing this local gain with the corresponding path-splicing cost shows that
\(v\) is deletable whenever $d_G(v)-\kappa(G-v)\ge2$, where \(\kappa(G-v)\) is the number of components of \(G-v\).  A sharper two-arm
estimate handles the borderline case.

\ppar{Reduction to terminal configurations}
These deletion criteria force any minimal counterexample into a rigid terminal
family.  Every \(2\)-connected block must be an even cycle, reducing the graph
to a \emph{sun graph}, that is, an even cycle with pendant leaves attached to
some cycle vertices.  The sun must moreover be sparse, meaning that at most one
pendant leaf is attached to each cycle vertex.  Its loaded cycle vertices,
namely the cycle vertices carrying pendant leaves, must be \(3\)-separated
along the cycle: any two of them have cyclic distance at least \(3\).  The pure
even-cycle case is excluded by a trigonometric pair comparison and a
weak-majorization argument.

\ppar{Excluding terminal sparse-sun configurations}
It remains to exclude \(3\)-separated sparse suns.  Choose a loaded cycle vertex
and delete one of its unloaded cycle neighbors, namely a neighboring cycle
vertex carrying no pendant leaf.  The rank-one spectral-shift formula reduces
the required deletability check to a comparison between the local interval
contribution of this deletion and the corresponding endpoint contribution for
the path extension \(P_{n-1}\rightsquigarrow P_n\).  This local comparison is
proved in three ranges.  For \(0\le t\le3\), it follows from a Jensen lower
envelope; for \(3\le t\le4\), from a fifth-moment lower envelope and exact
Bernstein certificates recorded in the appendix; and for \(t\ge4\), from the
fact that the path endpoint contribution vanishes.  Hence the chosen vertex is
deletable, contradicting minimality.  This eliminates the terminal sparse-sun
case and completes the stop-loss comparison for connected bipartite graphs.

Combining the \(R_1\)-comparison for \(2<p<4\), the admissible-function comparison for \(p\ge4\), and the identity $\mathcal E_2(G)=2|E(G)|$, we obtain the \(p\)-energy inequality for connected bipartite graphs for every \(p\ge2\).  The bipartite reduction then extends the result to all connected graphs.

\subsubsection{Equality}

The case \(p=4\) is the anchor:
the closed-walk formula for \(\operatorname{tr}A^4\) shows that equality at
\(p=4\) holds only for \(P_n\).

We first discuss equality for connected bipartite graphs. For \(2<p<4\), equality in the Mellin representation forces
\(R_1(G;x)=R_1(P_n;x)\) for all \(x>0\).  Expanding at \(x=0\) then gives
equality at \(p=4\), hence \(G\cong P_n\). For \(p>4\), take \(F(u)=u^{p/2}\) in the stop-loss decomposition above.  Since
\(F'(0)=F''(0)=0\) and \(F'''(t)>0\) for \(t>0\), equality forces $\mathsf S_t(G)=\mathsf S_t(P_n)$ for all \(t>0\), and
letting \(t\downarrow0\) gives equality at \(p=4\).  Thus \(G\cong P_n\) in the
bipartite case.

For a general connected graph, equality in the bipartite reduction forces no
edges to be lost when passing to the bipartite spanning subgraph.  Hence the
graph itself is the bipartite path.  Therefore, for every \(p>2\), equality
holds if and only if \(G\cong P_n\).

\subsection{Scope of the paper}\label{subsec:intro-scope}

The present paper is devoted to the proof of Theorem~\ref{thm:pgeq2}. Some of the comparison principles developed below also have consequences for positive adjacency \(p\)-energies, Laplacian-type spectral functionals, and line graphs. These applications are developed in the companion paper \cite{LiuTangApp}. We do not include them here, so as to keep the present paper focused on Nikiforov's path-minimality problem.

\subsection{Paper organization}

Section~\ref{sec:preliminaries} collects the preliminary material used
throughout the paper.  We first fix the spectral notation for general and
bipartite graphs, prove the reduction from connected graphs to connected
bipartite spanning subgraphs, and then record the matching-polynomial,
tree-shift, and Stieltjes-transform tools needed later.

Section~\ref{sec:R1-comparison} treats the range \(2<p<4\).  The main object
there is the integrated resolvent functional \(R_1\).  A Mellin representation
reduces fractional power sums to a pointwise \(R_1\)-comparison, which is first
proved for trees by matching-polynomial monotonicity and then extended to all
connected bipartite graphs by a vertex-deletion induction.

Sections~\ref{sec:sl2-framework}--\ref{sec:support-neighbor-deletion-clean}
develop the second-order stop-loss comparison needed for the range \(p\ge4\).
Section~\ref{sec:sl2-framework} introduces the stop-loss framework, proves the
tree-shift and path-splicing inputs, and establishes the pure even-cycle
comparison.  Section~\ref{sec:sl2-deletion} develops the deletion theory and
reduces any minimal counterexample to a sparse-sun configuration whose loaded
cycle vertices are \(3\)-separated.  Section~\ref{sec:support-neighbor-deletion-clean}
excludes the remaining sparse-sun case with pendant leaves by a
support-neighbor deletion argument and a finite certified comparison.

Section~\ref{sec:completion} combines the \(R_1\)-comparison for \(2<p<4\), the
second-order stop-loss comparison for \(p\ge4\), and the bipartite reduction to
prove Theorem~\ref{thm:pgeq2}, including the equality case.

The appendices contain the finite certified and elementary verifications used
in the proof.  Appendix~\ref{app:a2sl2-certificates} gives the rigorous interval
certificate for the low-range path-splicing estimate.  Appendix~\ref{app:interval-domination}
proves the interval domination inequalities used in the deletion step, and
Appendix~\ref{app:verification} records the exact polynomial certificates used
in the terminal sparse-sun argument.

\section{Preliminaries}\label{sec:preliminaries}

This section collects the auxiliary material used in the proof.  We first fix
the spectral notation and prove the bipartite reduction.  We then record the
matching-polynomial and tree-shift tools used in the \(R_1\)-comparison, and
finally recall the Stieltjes-transform and cone-closure facts needed for the
second-order stop-loss framework.

We shall use the following elementary notation.  For a real number \(x\), we write \(\lfloor x\rfloor\) for the floor of \(x\), and $\{x\}:=x-\lfloor x\rfloor\in[0,1)$ for its fractional part.

\subsection{Spectral notation and bipartite reduction}
\label{subsec:bipartite-spectral-prelim}

We first record the spectral notation used in the main proof.  For general
background on spectral graph theory, see, for example, \cite{Chung1997}.  We
use the adjacency-eigenvalue and \(p\)-energy notation introduced in
Section~\ref{sec:bamr}.

For real symmetric matrices \(A\) and \(B\) of the same size, we write
\(A\succeq B\) if \(A-B\) is positive semidefinite.  Let \(M\) be a real
symmetric matrix.  By the spectral theorem, we may write
\[
        M=U\operatorname{diag}(\lambda_1,\ldots,\lambda_n)U^{\top},
\]
where \(U\) is orthogonal and \(\lambda_1,\ldots,\lambda_n\) are the eigenvalues
of \(M\), counted with multiplicity.  We write \(\spec(M)\) for the spectrum of
\(M\).  If \(f\) is a real function defined on \(\spec(M)\), we define
\[
        f(M):=
        U\operatorname{diag}\bigl(f(\lambda_1),\ldots,f(\lambda_n)\bigr)U^{\top}.
\]
In particular, for \(t\in\mathbb R\), we write
\[
        (M-tI)_+
        :=
        U\operatorname{diag}\bigl((\lambda_1-t)_+,\ldots,(\lambda_n-t)_+\bigr)
        U^{\top}.
\]

Now assume that \(G\) is bipartite with a fixed bipartition \(V(G)=X\sqcup Y\),
and let \(B_G\) be the \(|X|\times |Y|\) biadjacency matrix.  Then
\[
        A(G)=
        \begin{pmatrix}
                0 & B_G\\
                B_G^{\top} & 0
        \end{pmatrix}.
\]
Thus the adjacency spectrum is symmetric about the origin.  Its nonzero
eigenvalues occur in opposite pairs
\[
        \pm\sigma_1(G),\ldots,\pm\sigma_m(G),
\]
where \(\sigma_1(G),\ldots,\sigma_m(G)\) are the positive singular values of
\(B_G\).  We write
\[
        \mu_i(G):=\sigma_i(G)^2.
\]
Equivalently, the \(\mu_i(G)\) are the nonzero eigenvalues of both
\(B_GB_G^{\top}\) and \(B_G^{\top}B_G\).  Hence, for every \(p>0\),
\begin{equation}
        \mathcal E_p(G)=2\sum_i \mu_i(G)^{p/2},
        \qquad
        \sum_i\mu_i(G)=\tr(B_G^{\top}B_G)=|E(G)|.
        \label{eq:bipartite-energy-mu}
\end{equation}
If \(G=\bigsqcup_{\ell=1}^r G_\ell\) is a disjoint union of bipartite graphs,
then the multiset of squared singular values of \(G\) is the union of the
corresponding multisets for the components.  Hence quantities defined as sums
over the \(\mu_i(G)\) are additive over disjoint unions.

Next, we state a structural reduction that allows us to work entirely in the
bipartite setting for even convex test functions.

\begin{lemma}\label{lem:bipartite-reduction}
Let $G$ be a connected graph on $n$ vertices.  Then there exists a connected bipartite spanning subgraph $H$ of $G$ such that
\[
        \sum_{i=1}^n F(\lambda_i(H))
        \le
        \sum_{i=1}^n F(\lambda_i(G))
\]
for every even convex function $F:\mathbb R\to\mathbb R$.  In particular, for every $p\ge1$, we have $\mathcal E_p(H)\le \mathcal E_p(G)$.
\end{lemma}

\begin{proof}
Let $T$ be a spanning tree of $G$, and let $V(G)=X\sqcup Y$ be the bipartition
of $T$.  Let $H$ be the spanning subgraph of $G$ consisting of all edges of
$G$ that cross the cut $(X,Y)$.  Since $T\subseteq H\subseteq G$, the graph
$H$ is connected and bipartite.

Order the vertices so that the vertices of $X$ come first.  Then
\[
A(G)=
\begin{pmatrix}
A_X & B\\
B^{\top} & A_Y
\end{pmatrix},
\qquad
A(H)=
\begin{pmatrix}
0 & B\\
B^{\top} & 0
\end{pmatrix}.
\]
With $D=\operatorname{diag}(I_{|X|},-I_{|Y|})$, we have
\[
        A(H)=\frac{A(G)+(-D A(G)D)}{2}.
\]

For the given even convex function \(F:\mathbb R\to\mathbb R\), and for any real symmetric matrix \(M\), write
\[
        \Phi_F(M):=\tr F(M)=\sum_{i=1}^n F(\lambda_i(M)).
\]
Since \(F\) is convex, the trace spectral functional $\Phi_F(M)$ is convex on the space of real symmetric matrices. This is a standard consequence of the Davis theorem on convex invariant functions of Hermitian matrices and its real-symmetric corollary; see \cite{Davis1957}. Therefore
\[
\Phi_F(A(H))=\Phi_F\left(\frac{A(G)+(-D A(G)D)}{2}\right) \le \frac{\Phi_F(A(G))+\Phi_F(-D A(G)D)}2.
\]
Since \(D\) is orthogonal and \(D^{-1}=D\), the matrix \(-D A(G)D=D(-A(G))D\) is orthogonally similar to \(-A(G)\). Thus
\[
        \Phi_F(-D A(G)D)=\Phi_F(-A(G)).
\]
Since $F$ is even, $F(-t)=F(t)$ for every $t\in\mathbb R$, and hence
\[
        \Phi_F(-A(G))=\Phi_F(A(G)).
\]
Thus $\Phi_F(A(H))\le\Phi_F(A(G))$, which is exactly the stated inequality. Taking $F(t)=|t|^p$ gives $\mathcal E_p(H)\le\mathcal E_p(G)$.
\end{proof}

Consequently, any lower bound for an even convex spectral functional over connected bipartite graphs of order $n$ immediately extends to all connected graphs of order $n$.  Thus the rest of the proof of Theorem~\ref{thm:pgeq2} may be carried out entirely in the connected bipartite setting.  

\subsection{Matching generating polynomials}
\label{subsec:matching-poly-prelim}

We shall later compare trees by translating a spectral determinant into a matching generating polynomial. This subsection records the matching-polynomial facts needed for that comparison.

\begin{definition}[Weighted matching generating polynomial]\label{def:matching-poly}
Let \(G\) be a finite graph whose edges carry nonnegative weights \(\omega_e\).  The \emph{weighted matching generating polynomial} of \(G\) is
\[
        M_G(x)=\sum_M \left(\prod_{e\in M}\omega_e\right)x^{|M|},
\]
where the sum is over all matchings \(M\) of \(G\). 
\end{definition}
\begin{remark}
If \(G\) is unweighted, we take \(\omega_e=1\) for every edge and recover the usual matching generating polynomial
\[
        M_G(x)=\sum_{j\geq 0} m_j(G)x^j,
\]
where \(m_j(G)\) is the number of \(j\)-matchings in \(G\).
\end{remark}

This is the standard monomer--dimer partition-function normalization; see Heilmann--Lieb~\cite{HeilmannLieb1972}. For general background on matching polynomials, see Godsil--Gutman~\cite{GodsilGutman1981}. We call a ratio of the form \(M_{G-S}(x)/M_G(x)\), where \(S\) is a deleted vertex set, a \emph{deletion ratio}.

The next theorem supplies the real-rootedness input for weighted graphs.

\begin{theorem}[Heilmann--Lieb zero theorem]
\label{thm:heilmann-lieb-zero}
Let \(G\) be a finite graph with nonnegative edge weights \(\omega_e\), and let \(M_G(x)\) be the corresponding weighted matching generating polynomial.  Then every zero of \(M_G(x)\) is real and negative.
\end{theorem}

\begin{proof}
Let
\[
        P(G;z):=\sum_{M}\left(\prod_{e\in M}\omega_e\right) z^{|V(G)|-2|M|},
\]
where the sum runs over all matchings \(M\) of \(G\). This is exactly the weighted monomer--dimer partition polynomial considered by Heilmann and Lieb~\cite[Eq.~(2.5a)]{HeilmannLieb1972}. For graphs with nonnegative edge weights, they proved that all zeros of \(P(G;z)\) lie on the imaginary axis~\cite[Remark~4.1]{HeilmannLieb1972}.

By the definition of \(M_G(x)\), for every \(z\neq0\) we have
\[
        P(G;z)
        =\sum_M \left(\prod_{e\in M}\omega_e\right) z^{|V(G)|-2|M|}
        =z^{|V(G)|}\sum_M \left(\prod_{e\in M}\omega_e\right)(z^{-2})^{|M|}
        =z^{|V(G)|}M_G(z^{-2}).
\]
Now let \(\xi\) be a zero of \(M_G(x)\). Since the empty matching contributes the constant term \(1\), we have \(M_G(0)=1\), and hence \(\xi\neq0\). Choose \(z\in\mathbb C\setminus\{0\}\) such that \(z^{-2}=\xi\). Then
\[
        P(G;z)=z^{|V(G)|}M_G(\xi)=0.
\]
Since every zero of \(P(G;z)\) is purely imaginary, we know that \(z\in i\mathbb R\). Therefore
\[
        \xi=z^{-2}\in(-\infty,0).
\]
Thus every zero of \(M_G(x)\) is real and negative.
\end{proof}

\begin{example}[Paths]
For the unweighted path \(P_m\) we write
\[
h_m(x):=M_{P_m}(x)=\sum_j m_j(P_m)x^j\qquad(m\ge0),
\]
and we use the convention $h_{-1}(x)=0$. Thus \(h_0(x)=h_1(x)=1\). The matching recurrence for the last edge gives
\begin{equation}\label{eq:path-matching-rec}
h_m(x)=h_{m-1}(x)+x h_{m-2}(x)\qquad(m\ge1).
\end{equation}
More generally, concatenating two paths at an endpoint and conditioning on whether the connecting edge is used gives the \emph{path-concatenation identity}
\begin{equation}\label{eq:path-concatenation}
h_{a+b}(x)=h_a(x)h_b(x)+x h_{a-1}(x)h_{b-1}(x)\qquad (a,b\ge0),
\end{equation}
again with the convention \(h_{-1}(x)=0\).
\end{example}

The next lemma is the bridge between weighted matching generating polynomials and the determinant identities used later.

\begin{lemma}
\label{lem:forest-det-matching}
Let \(\cF\) be a weighted forest, and fix a bipartition \(V(\cF)=L\sqcup R\).
Define the weighted biadjacency matrix \(B_{\cF}\) by
\[
        B_{\cF}=\bigl((B_{\cF})_{uv}\bigr)_{u\in L,\,v\in R},
        \qquad
        (B_{\cF})_{uv}
        =
        \begin{cases}
                \sqrt{\omega_{uv}}, & uv\in E(\cF),\\
                0, & uv\notin E(\cF).
        \end{cases}
\]
Then
\begin{equation}\label{eq:forest-det-matching-r1}
\det(I+xB_{\cF}B_{\cF}^{\top})=\det(I+xB_{\cF}^{\top}B_{\cF})=M_{\cF}(x).
\end{equation}
\end{lemma}

\begin{proof}
By expanding $\det(I+xB_{\cF}^{\top}B_{\cF})$ into principal minors and then applying Cauchy--Binet,
\[
\det(I+xB_{\cF}^{\top}B_{\cF})
=
\sum_{J\subseteq R}x^{|J|}\det\bigl((B_{\cF}^{\top}B_{\cF})_{J,J}\bigr)=
\sum_{I\subseteq L,\;J\subseteq R,\;|I|=|J|}
x^{|J|}\det(B_{\cF}[I,J])^2.
\]
Using \(\det(I+xB_{\cF}^{\top}B_{\cF})=\det(I+xB_{\cF}B_{\cF}^{\top})\), we also obtain
\[
\det(I+xB_{\cF}B_{\cF}^{\top})
=
\sum_{I\subseteq L,\;J\subseteq R,\;|I|=|J|}
x^{|J|}\det(B_{\cF}[I,J])^2.
\]

Now fix \(I\subseteq L\) and \(J\subseteq R\) with \(|I|=|J|\). Since \(\cF\) is a forest,
the induced subgraph \(\cF[I\cup J]\) is again a forest, and hence it has at most one
perfect matching. If \(\cF[I\cup J]\) has no perfect matching, then \(\det(B_{\cF}[I,J])=0\).
If \(\cF[I\cup J]\) has a perfect matching, then it is unique, and
\(\det(B_{\cF}[I,J])^2\) is exactly the product of the weights of the edges in that matching.
Therefore the contribution of all pairs \((I,J)\) to the above sum is precisely the
weighted matching generating polynomial \(M_{\cF}(x)\). This proves
\[
\det(I+xB_{\cF}B_{\cF}^{\top})=\det(I+xB_{\cF}^{\top}B_{\cF})=M_{\cF}(x).
\qedhere\]
\end{proof}

\subsection{Tree shifts}
\label{subsec:tree-transform-prelim}

With the matching-polynomial framework in place, we now introduce the tree shift that moves an arbitrary tree toward the path.

The following tree-shift operation is the key to obtaining the main spectral results for trees. It is a one-branch special case of Csikv\'ari's generalized tree shift; see \cite{Csikvari2010}. Figure~\ref{fig:tree-transform} illustrates this operation.

\begin{figure}[htbp]
        \centering
        \includegraphics[width=\textwidth]{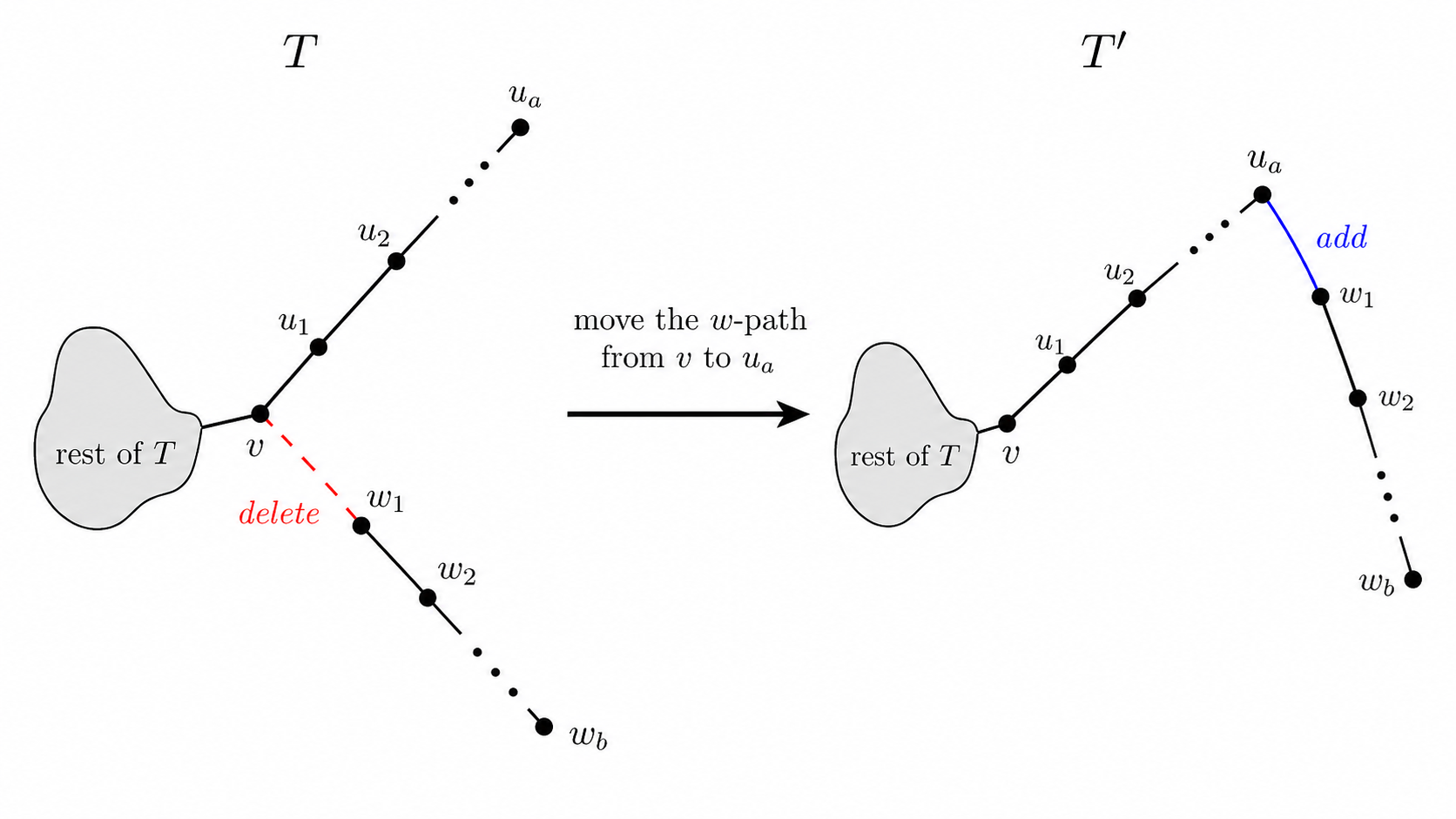}
        \caption{The one-branch tree shift used in Lemma~\ref{lem:one-shift}: delete the edge \(vw_1\) and reattach the pendant path \(w_1\cdots w_b\) at \(u_a\).}
        \label{fig:tree-transform}
\end{figure}

\begin{lemma}
\label{lem:one-shift}
Let \(T\) be a tree containing a vertex \(v\) with two pendant paths
\[
        vu_1u_2\cdots u_a,
        \qquad
        vw_1w_2\cdots w_b,
\]
where \(a,b\geq 1\). Let \(T'\) be the tree obtained from \(T\) by deleting the edge \(vw_1\) and adding the edge \(u_a w_1\).  Then there is a polynomial \(Q(x)\) with nonnegative coefficients such that
\[
        M_{T'}(x)=M_T(x)+x^2Q(x).
\]
\end{lemma}

\begin{proof}
Let \(C\) be the subtree obtained from \(T\) by deleting the vertices
\(u_1,\ldots,u_a,w_1,\ldots,w_b\); thus \(v\in V(C)\).  We shall use the path-concatenation~\eqref{eq:path-concatenation}.

Decomposing matchings according to whether the edge incident to \(v\) on each
displayed path is used, we obtain
\[
        M_T=h_ah_bM_C+x\bigl(h_{a-1}h_b+h_ah_{b-1}\bigr)M_{C-v}.
\]
After the shift, the two pendant paths form a single path on \(a+b\) vertices
attached to \(v\), so
\[
        M_{T'}=h_{a+b}M_C+xh_{a+b-1}M_{C-v}.
\]
Using the two identities
\[
        h_{a+b}=h_ah_b+xh_{a-1}h_{b-1}
\]
and
\[
        h_{a+b-1}=h_{a-1}h_b+h_ah_{b-1}-h_{a-1}h_{b-1},
\]
we get
\[
        M_{T'}-M_T=xh_{a-1}h_{b-1}\bigl(M_C-M_{C-v}\bigr).
\]
Applying the matching recurrence at \(v\) inside \(C\),
\begin{equation}\label{eq:matching-recurrence-v}
M_C-M_{C-v}=x\sum_{z\in C: v\sim z}M_{C-v-z}, 
\end{equation}
where the sum is interpreted as \(0\) if \(v\) is isolated in \(C\).  Hence
\[
        M_{T'}-M_T
        =x^2h_{a-1}h_{b-1}\sum_{z\in C: v\sim z}M_{C-v-z}.
\]
The right-hand side is a polynomial with nonnegative coefficients, which proves the claim.
\end{proof}

\begin{lemma}
\label{lem:reduce-to-path}
Every tree \(T\) on \(n\) vertices can be transformed into \(P_n\) by a
finite sequence of the tree shifts described in Lemma~\ref{lem:one-shift}.
\end{lemma}

\begin{proof}
We call a vertex \emph{branching} if it has degree at least \(3\).  If a
tree has no branching vertex, then its maximum degree is at most \(2\),
and hence, being connected, it is a path.

Assume that the current tree is not a path.  Root it at an arbitrary leaf,
and choose a branching vertex \(v\) of maximal distance from the root.  Let
\(\pi(v)\) be the parent of \(v\).  Every component of \(T-v\) not containing
\(\pi(v)\) contains no branching vertex; otherwise such a branching vertex
would be farther from the root than \(v\).  Hence each such component is a
path, attached to \(v\) at one of its endpoints.  Since \(v\) is branching
and has exactly one parent edge, there are at least two such pendant paths.
Choose two of them,
\[
   vu_1\cdots u_a
   \qquad\text{and}\qquad
   vw_1\cdots w_b .
\]

Now delete the edge \(vw_1\) and add the edge \(u_a w_1\).  This reconnects
the separated \(w\)-path to the rest of the tree, and therefore produces
again a tree on the same vertex set.  Under this operation the leaf
\(u_a\) ceases to be a leaf, while \(w_b\) remains a leaf; the vertex
\(v\) still has degree at least \(2\), since it remains adjacent to
\(\pi(v)\) and \(u_1\).  No other vertex changes its leaf status.  Thus the
number of leaves decreases by exactly \(1\).

Repeating the operation whenever the current tree is not a path must stop
after finitely many steps, because the number of leaves is a positive
integer and strictly decreases at each step.  The terminal tree has no
branching vertex, and hence is a path.  Since the number of vertices is
preserved throughout, the terminal tree is \(P_n\).
\end{proof}

\subsection{Stieltjes transforms and cones}
\label{subsec:stieltjes-prelim}

For the second-order stop-loss part we use Stieltjes transforms of finite signed measures and, in particular, finite atomic measures, the cones \(\mathcal C_{\le k}\), and the auxiliary closure lemma (Lemma~\ref{lem:stieltjes-closure}). The proof of that lemma uses one standard complete-Bernstein argument; see, for example, \cite[Chapter~7]{Schilling2012}.

\begin{definition}[Stieltjes transform]
If \(\tau\) is a finite signed Borel measure on \([0,\infty)\), its
\emph{Stieltjes transform} is
\[
        \mathcal S_\tau(x):=\int_{[0,\infty)}\frac{1}{1+x\lambda}\,d\tau(\lambda) \qquad(x>0).
\]
In the finite atomic case
\[
        \tau=m_0\delta_0+\sum_{\ell=1}^N m_\ell\delta_{s_\ell},
        \qquad m_0,m_\ell\in\mathbb R,\ s_\ell>0,
\]
where \(\delta_s\) denotes the Dirac measure at \(s\), this becomes
\[
        \mathcal S_\tau(x)=m_0+\sum_{\ell=1}^N\frac{m_\ell}{1+s_\ell x}.
\]
\end{definition}

\begin{definition}[Stieltjes cones]
For \(k\ge0\), let \(\mathcal C_{\le k}\) denote the convex cone generated by the functions
\[
        \prod_{j=1}^{m}(1+s_jx)^{-1},
        \qquad
        0\le m\le k,\qquad s_j>0,
\]
where the parameters \(s_j\) are allowed to repeat. In other words, \(\mathcal C_{\le k}\) consists of all finite nonnegative linear combinations of such functions. The case \(m=0\) is interpreted as the constant function \(1\). We also put
\[
        \mathcal C_{\le k}^{0}
        :=
        \{g\in\mathcal C_{\le k}: g(x)=O(x^{-1})\text{ as }x\to\infty\}.
\]
\end{definition}

It is immediate from the definition that \(\mathcal C_{\le k}\cdot\mathcal C_{\le \ell} \subseteq \mathcal C_{\le k+\ell}\). Moreover, \(\mathcal C_{\le k}^{0}\) is the subcone of \(\mathcal C_{\le k}\) consisting of functions that satisfy \(g(x)=O(x^{-1})\) as \(x\to\infty\). This distinction is important: deletion ratios may have a nonzero limit at infinity, and therefore naturally lie in \(\mathcal C_{\le k}\), not necessarily in \(\mathcal C_{\le k}^{0}\).

By definition, every element of \(\mathcal C_{\le1}\) is the Stieltjes
transform of a finite atomic nonnegative measure on \([0,\infty)\), and
\(\mathcal C_{\le1}^{0}\) consists exactly of those such functions with zero
constant term.

We shall also use the following standard terminology. A nonnegative function \(h:(0,\infty)\to[0,\infty)\) is called a \emph{Stieltjes function} if it admits a representation
\[
        h(x)=\frac{a}{x}+b+\int_{(0,\infty)}\frac{1}{x+t}\,d\mu(t),
        \qquad a,b\ge0,
\]
where \(\mu\) is a positive measure satisfying $\int_{(0,\infty)}(1+t)^{-1}\,d\mu(t)<\infty$; see \cite[Definition~2.1]{Schilling2012}. A function is called \emph{rational Stieltjes} if it has the form
\[
        f(x)=\frac{a}{x}+\mathcal S_\tau(x)
        =\frac{a}{x}+m_0+\sum_{\ell=1}^N\frac{m_\ell}{1+s_\ell x},
        \qquad
        a,m_\ell\ge0,
        \quad s_\ell>0,
\]
for some finite atomic nonnegative measure \(\tau=m_0\delta_0+\sum_{\ell=1}^N m_\ell\delta_{s_\ell}\) on \([0,\infty)\). In particular, every element of \(\mathcal C_{\le1}\) is a rational Stieltjes function with no pole at \(0\).

We now recall the standard notion of a \emph{complete Bernstein function}. We shall use the following consequence of \cite[Definition~6.1 and Theorem~6.2(vi)]{Schilling2012}: if \(\varphi\) is a complete Bernstein function, then its restriction to \((0,\infty)\) admits a representation
\[
        \varphi(x)=a+bx+\int_{(0,\infty)}\frac{x}{x+t}\,d\mu(t),
        \qquad a,b\ge0,
\]
where \(\mu\) is a positive measure satisfying $\int_{(0,\infty)}(1+t)^{-1}\,d\mu(t)<\infty$.

We shall use the following standard facts. If \(h\) is a nonzero Stieltjes function, then \(1/h\) is a complete Bernstein function, and conversely the reciprocal of a nonzero complete Bernstein function is a Stieltjes function \cite[Theorem~7.3]{Schilling2012}. If \(h\) is a Stieltjes function, then \(xh(x)\) is a complete Bernstein function \cite[Theorem~6.2(ii)]{Schilling2012}. Finally, complete Bernstein functions are closed under addition, which is immediate from the integral representation in \cite[Theorem~6.2(vi)]{Schilling2012}.

\begin{lemma}\label{lem:stieltjes-closure}
If \(c(x),d(x)\in\mathcal C_{\le1}\), then
\[
        \frac{c(x)}{1+xc(x)d(x)}\in\mathcal C_{\le1}.
\]
\end{lemma}

\begin{proof}
If \(c\equiv0\), then the assertion is immediate. Hence we may assume \(c\not\equiv0\). Since \(c,d\in\mathcal C_{\le1}\), both \(c\) and \(d\) are rational Stieltjes functions with no pole at \(0\). Moreover, \(c(x)>0\) for every \(x>0\).

In particular, \(1/c\) and \(xd\) are complete Bernstein functions, so $1/c(x)+xd(x)$ is again a complete Bernstein function. Its reciprocal is therefore a Stieltjes function; since \(c\) and \(d\) are rational, this reciprocal is rational as well. Thus $\frac{c(x)}{1+xc(x)d(x)}$ is a rational Stieltjes function.

It remains to check that this function lies in the cone \(\mathcal C_{\le1}\). Since \(c\) and \(d\) are regular at \(0\), so is \(\frac{c(x)}{1+xc(x)d(x)}\). Hence its \(a/x\)-term vanishes, and therefore it is a nonnegative linear combination of the constant function \(1\) and \((1+sx)^{-1}\). This completes the proof.
\end{proof}

\section{The \texorpdfstring{$R_1$}{R1}-comparison for \texorpdfstring{$2<p<4$}{2<p<4}}\label{sec:R1-comparison}

In this section we prove the path-minimality inequality throughout the interval \(2<p<4\).  The central device is the functional \(R_1\).  The Mellin representation below expresses the fractional power sums appearing in \(\mathcal E_p\) as positive integrals of \(R_1\).  We then use a determinant identity, together with the bipartite spectral notation from Section~\ref{subsec:bipartite-spectral-prelim}, the matching-polynomial preliminaries from Section~\ref{subsec:matching-poly-prelim}, and the tree-transform discussion in Section~\ref{subsec:tree-transform-prelim}, to prove the tree case.  Finally, a deletion induction gives the pointwise \(R_1\)-comparison for all connected bipartite graphs.

\subsection{From \texorpdfstring{$R_1$}{R1} to \texorpdfstring{$p$}{p}-energy}

Recall from Section~\ref{subsec:bipartite-spectral-prelim} that, for a bipartite graph \(G\), the squared singular values \(\mu_i(G)\) satisfy~\eqref{eq:bipartite-energy-mu}.  Thus a lower bound for $\mathcal E_p(G)$ against a path is equivalent to a lower bound for a fractional power sum of the squared singular values:
\[
        \sum_i \mu_i(G)^\alpha
        \ge
        \sum_i \mu_i(P_n)^\alpha,
        \qquad \alpha=\frac p2.
\]
The point of the resolvent functional below is to replace fractional powers by a positive integral of the single test function $r_1(xt)$.

Define
\[
        r_1(y):=\int_0^y\frac{s}{1+s}\,ds
        =y-\log(1+y),
        \qquad y\ge0.
\]
For a bipartite graph $G$ and $x>0$, define
\[
        R_1(G;x):=\sum_i r_1(x\mu_i(G)).
\]
Equivalently, $R_1(G;x)=\tr r_1\bigl(xB_G^{\top}B_G\bigr)$, where zero eigenvalues cause no contribution because $r_1(0)=0$.

By the disjoint-union observation in Section~\ref{subsec:bipartite-spectral-prelim}, the functional $R_1$ is additive over disjoint unions.

The relevance to fractional powers in the range $1<\alpha<2$ is the
following Mellin identity.

\begin{lemma}\label{lem:mellin}
Let $\alpha\in(1,2)$.  Then for every $t\ge0$,
\begin{equation}\label{eq:Mellin-representation_1}
t^\alpha = c_\alpha\int_0^\infty r_1(xt)x^{-\alpha-1}\,dx,
\end{equation}
where \(c_\alpha=\alpha\sin\left(\pi(\alpha-1)\right)/\pi>0\).
\end{lemma}

\begin{proof}
The identity is trivial at $t=0$, so assume $t>0$.  By the change of variables $u=xt$,
\[
        \int_0^\infty r_1(xt)x^{-\alpha-1}\,dx
        =t^\alpha\int_0^\infty r_1(u)u^{-\alpha-1}\,du.
\]
It remains to compute the constant.  The boundary terms in the following integration by parts vanish.  Indeed, as \(u\to0^+\), we have $r_1(u)=\frac{u^2}{2}+O(u^3)$, and hence
\[
        r_1(u)u^{-\alpha}=O(u^{2-\alpha})\to0
        \qquad(\alpha<2).
\]
As \(u\to\infty\), we have \(r_1(u)=O(u)\), and hence
\[
        r_1(u)u^{-\alpha}=O(u^{1-\alpha})\to0
        \qquad(\alpha>1).
\]
The same estimates also show that
\(r_1(u)u^{-\alpha-1}\) is integrable at both endpoints.  Therefore
integration by parts gives
\[
\int_0^\infty r_1(u)u^{-\alpha-1}\,du =\frac1\alpha\int_0^\infty u^{-\alpha}r_1'(u)\,du =\frac1\alpha\int_0^\infty \frac{u^{1-\alpha}}{1+u}\,du.
\]
Put $s=2-\alpha\in(0,1)$.  Then
\[
        \int_0^\infty \frac{u^{1-\alpha}}{1+u}\,du
        =\int_0^\infty \frac{u^{s-1}}{1+u}\,du
        =\frac{\pi}{\sin(\pi s)}
        =\frac{\pi}{\sin(\pi(\alpha-1))}.
\]
Thus
\[
        \int_0^\infty r_1(xt)x^{-\alpha-1}\,dx
        =t^\alpha\frac{\pi}{\alpha\sin(\pi(\alpha-1))},
\]
which proves the formula.
\end{proof}

\begin{theorem}\label{thm:R-implies-p}
Let $G$ and $H$ be bipartite graphs.  Suppose
\[
        R_1(G;x)\ge R_1(H;x)
        \qquad\text{for every }x>0.
\]
Then for every $p\in(2,4)$, we have $\mathcal E_p(G)\ge \mathcal E_p(H)$.
\end{theorem}

\begin{proof}
Set \(\alpha=p/2\), so \(\alpha\in(1,2)\).  Zero squared singular values may be included in the sums below, since both sides of \eqref{eq:Mellin-representation_1} vanish when \(t=0\).  Since each graph is finite, there are only finitely many
squared singular values, and \eqref{eq:Mellin-representation_1} may be summed over them. The resulting integrals are finite by the endpoint estimates in the proof of Lemma~\ref{lem:mellin}.
Hence
\[
\frac12\mathcal E_p(G) = \sum_i\mu_i(G)^\alpha = c_\alpha\int_0^\infty R_1(G;x)x^{-\alpha-1}\,dx.
\]
The same identity holds for \(H\).  The weight
\(c_\alpha x^{-\alpha-1}\) is strictly positive for \(x>0\).  Therefore the
pointwise comparison
\[
        R_1(G;x)\ge R_1(H;x)\qquad(x>0)
\]
implies $\mathcal E_p(G)\ge \mathcal E_p(H)$.
\end{proof}

Thus, by Theorem~\ref{thm:R-implies-p}, in the range $2<p<4$, it remains to prove the pointwise comparison
\[
        R_1(G;x)\ge R_1(P_n;x)
        \qquad(x>0)
\]
for every connected bipartite graph $G$ on $n$ vertices.

\subsection{Determinant identity and the tree case}

Recall that, for a bipartite graph $G$, we write $Q_G:=B_G^{\top}B_G$ and $R_1(G;x)=\tr r_1(xQ_G)$ for $x>0$. Since \(r_1(y)=y-\log(1+y)\) and \(\tr Q_G=|E(G)|\), functional calculus gives the determinant identity
\begin{equation}\label{eq:R1-determinant}
R_1(G;x) =\tr r_1(xQ_G) =x\tr Q_G-\tr\log(I+xQ_G) =x|E(G)|-\log\det(I+xQ_G).
\end{equation}
Here we used the standard identity \(\tr\log M=\log\det M\) for positive definite matrices \(M\), applied to \(M=I+xQ_G\). The same determinant is obtained from $B_GB_G^{\top}$, because $B_GB_G^{\top}$ and $B_G^{\top}B_G$ have the same nonzero eigenvalues.  The functional $R_1$ is additive over disjoint unions.

\begin{corollary}\label{cor:tree-r1}
Let $T$ be a tree on $n$ vertices.  Then, for every $x>0$,
\[
        R_1(T;x)\ge R_1(P_n;x).
\]
\end{corollary}

\begin{proof}
Since $|E(T)|=|E(P_n)|=n-1$, Lemma~\ref{lem:forest-det-matching} together with \eqref{eq:R1-determinant} gives
\[
        R_1(T;x)-R_1(P_n;x)
        =
        \log M_{P_n}(x)-\log M_T(x).
\]
By Lemma~\ref{lem:reduce-to-path}, there is a sequence of trees
\[
        T=T_0,T_1,\ldots,T_m=P_n
\]
such that each \(T_{j+1}\) is obtained from \(T_j\) by one operation of
Lemma~\ref{lem:one-shift}.  Since Lemma~\ref{lem:one-shift} gives
\(M_{T_{j+1}}(x)-M_{T_j}(x)=x^2Q_j(x)\) with nonnegative coefficients, we have
\(M_{T_j}(x)\le M_{T_{j+1}}(x)\) for every \(x\ge0\).  Iterating along the
sequence yields
\[
        M_T(x)\le M_{P_n}(x)
        \qquad(x\ge0).
\]
Therefore \(\log M_{P_n}(x)-\log M_T(x)\ge0\), which completes the proof.
\end{proof}

\subsection{Rank-one gain and path-deficit estimates}

We first record a rank-one vertex-addition estimate for \(R_1\). Equivalently, it gives a lower bound for the amount of \(R_1\) lost when a vertex of degree \(d\) is deleted.

\begin{lemma}\label{lem:vertex-gain-r1}
Let $H$ be a bipartite graph and let $G$ be obtained from $H$ by adding one
new vertex $v$ of degree $d$.  Then, for every $x>0$,
\[
        R_1(G;x)-R_1(H;x)\ge r_1(dx).
\]
\end{lemma}

\begin{proof}
Assume $v$ is added on the left side of the bipartition.  Let $B_H$ be the biadjacency matrix for $H$ and let $b$ be the $0$-$1$ column vector encoding the neighbors of $v$ on the right side. Put $M=B_H^{\top}B_H$. Then $Q_G=M+bb^{\top}$ and $b^{\top}b=d$. By the matrix determinant lemma,
\[
\begin{aligned}
R_1(G;x)-R_1(H;x)
&=xd-
\log\frac{\det(I+x(M+bb^{\top}))}{\det(I+xM)} \\
&=xd-\log\left(1+x b^{\top}(I+xM)^{-1}b\right).
\end{aligned}
\]
Since $M\succeq0$, we have $(I+xM)^{-1}\preceq I$.  Hence
\[
        b^{\top}(I+xM)^{-1}b\le b^{\top}b=d.
\]
Therefore
\[
        R_1(G;x)-R_1(H;x)
        \ge xd-\log(1+xd)
        =r_1(dx).
\]
The case in which $v$ is added on the right side is identical, using
$B_GB_G^{\top}$ instead of $B_G^{\top}B_G$.
\end{proof}

The second estimate controls the path comparator that appears when a vertex
is deleted and several components remain.  It is the matching upper bound for the preceding vertex-gain estimate.

\begin{lemma}\label{lem:path-deficit-r1}
Let $q\ge1$ and let $n_1,\ldots,n_q\ge1$. Put $N=1+n_1+\cdots+n_q$. Then for every $x>0$,
\begin{equation}
        R_1(P_N;x)-\sum_{i=1}^q R_1(P_{n_i};x)
        \le r_1((q+1)x).
        \label{eq:path-deficit-r1}
\end{equation}
\end{lemma}

\begin{proof}
Construct a tree $T$ as follows.  Take the disjoint union
$P_{n_1}\sqcup\cdots\sqcup P_{n_q}$, choose one endpoint in each path, and
add a new vertex adjacent to all chosen endpoints.  Then $T$ has $N$ vertices.
By Corollary~\ref{cor:tree-r1}, we know that $R_1(P_N;x)\le R_1(T;x)$. Thus it remains to estimate
\[
R_1(T;x)-\sum_{i=1}^qR_1(P_{n_i};x).
\]
Choose the bipartition of each path independently so that the selected endpoint lies on the right side, and place the new center on the left side.  With this choice, adding the new center to the disjoint union of the \(q\) paths is a rank-one update of degree \(q\). Let \(B_i\) be the biadjacency matrix of \(P_{n_i}\) with the chosen endpoint on the right side, and let \(e_i\) be the standard basis vector corresponding to this endpoint.  Put \(M_i=B_i^{\top}B_i\). For the disjoint union \(\cF=P_{n_1}\sqcup\cdots\sqcup P_{n_q}\), the matrix \(B_{\cF}^{\top}B_{\cF}\) is block diagonal with blocks \(M_1,\ldots,M_q\).  Adding the new center contributes the rank-one vector \(b=(e_1,\ldots,e_q)^{\top}\).
Thus
\[
        B_T^{\top}B_T=M+bb^{\top},
        \qquad
        M=\operatorname{diag}(M_1,\ldots,M_q).
\]
Using \eqref{eq:R1-determinant} and \(|E(T)|-|E(\cF)|=q\), we get
\[
R_1(T;x)-R_1(\cF;x)
=
qx-\log
\frac{\det(I+x(M+bb^{\top}))}{\det(I+xM)}.
\]
By the matrix determinant lemma,
\[
\frac{\det(I+x(M+bb^{\top}))}{\det(I+xM)}
=
1+x b^{\top}(I+xM)^{-1}b.
\]
Since \(M\) is block diagonal, \(b^{\top}(I+xM)^{-1}b = \sum_{i=1}^q e_i^{\top}(I+xM_i)^{-1}e_i\). Therefore
\begin{equation}
        R_1(T;x)-\sum_{i=1}^qR_1(P_{n_i};x)
        =
        qx-\log\left(1+x\sum_{i=1}^q e_i^{\top}(I+xM_i)^{-1}e_i\right).
        \label{eq:spider-increment}
\end{equation}
We claim that
\begin{equation}\label{eq:endpoint-resolvent-lower}
e_i^{\top}(I+xM_i)^{-1}e_i \ge \frac1{1+x}.
\end{equation}
Indeed, adding one new leaf to the chosen endpoint of $P_{n_i}$ produces
$P_{n_i+1}$.  Applying the determinant lemma to this operation gives
\[
h_{n_i+1}(x) = \det(I+x(M_i+e_i e_i^\top)) = h_{n_i}(x)(1+x e_i^{\top}(I+xM_i)^{-1}e_i).
\]
Using the path recurrence \eqref{eq:path-matching-rec}, we obtain
\[
        e_i^{\top}(I+xM_i)^{-1}e_i=\frac{h_{n_i-1}(x)}{h_{n_i}(x)}.
\]
If \(n_i=1\), this gives \(e_i^{\top}(I+xM_i)^{-1}e_i=h_0/h_1=1\ge(1+x)^{-1}\).  If
\(n_i\ge2\), then
\[
        h_{n_i}(x)=h_{n_i-1}(x)+xh_{n_i-2}(x)
        \le(1+x)h_{n_i-1}(x),
\]
because the path polynomials have non-negative coefficients and
\(h_{n_i-2}(x)\le h_{n_i-1}(x)\) for \(x>0\).  Hence
\eqref{eq:endpoint-resolvent-lower} follows in all cases.

Therefore
\[
        \sum_{i=1}^q e_i^{\top}(I+xM_i)^{-1}e_i \ge\frac{q}{1+x}.
\]
Substituting this into \eqref{eq:spider-increment},
\[
\begin{aligned}
R_1(T;x)-\sum_{i=1}^qR_1(P_{n_i};x)
&\le qx-\log\left(1+\frac{qx}{1+x}\right) \\
&=qx+\log(1+x)-\log(1+(q+1)x) \\
&\le(q+1)x-\log(1+(q+1)x) \\
&=r_1((q+1)x),
\end{aligned}
\]
where we used $\log(1+x)\le x$.  This proves
\eqref{eq:path-deficit-r1}.
\end{proof}

\subsection{The \texorpdfstring{$R_1$}{R1}-comparison for all connected bipartite graphs}

The preceding deletion estimates now yield the pointwise \(R_1\)-comparison needed for the range \(2<p<4\).

\begin{theorem}\label{thm:r1-main-comparison}
Let $G$ be a connected bipartite graph on $n$ vertices.  Then, for every
$x>0$,
\begin{equation}
        R_1(G;x)\ge R_1(P_n;x).
        \label{eq:r1-main-comparison}
\end{equation}
\end{theorem}

\begin{proof}
We argue by induction on $n$.  The cases $n=1,2$ are immediate.

If $G$ is a tree, then \eqref{eq:r1-main-comparison} follows from Corollary~\ref{cor:tree-r1}.

Assume now that $G$ is not a tree.  Then $G$ contains a cycle.  Choose a
vertex $v$ lying on a cycle.  Let
\[
        d=d_G(v),
        \qquad
        G-v=G_1\sqcup\cdots\sqcup G_q,
        \qquad
        n_i=|V(G_i)|.
\]
Since two neighbors of $v$ on the chosen cycle remain connected after $v$ is
removed, the number $q$ of components of $G-v$ is at most $d-1$:
\[
q\le d-1.
\]

By the induction hypothesis applied to each connected component $G_i$,
\[
        R_1(G_i;x)\ge R_1(P_{n_i};x).
\]
Using additivity over disjoint unions,
\begin{equation}\label{eq:induction-components}
        R_1(G-v;x)
        =\sum_{i=1}^qR_1(G_i;x)
        \ge\sum_{i=1}^qR_1(P_{n_i};x).
\end{equation}
By Lemma~\ref{lem:vertex-gain-r1}, adding $v$ back gives $R_1(G;x)-R_1(G-v;x) \ge r_1(dx)$. Since $r_1$ is increasing and $d\ge q+1$,
\begin{equation}\label{eq:gain-q-plus-one}
R_1(G;x)-R_1(G-v;x) \ge r_1((q+1)x).
\end{equation}
On the other hand, Lemma~\ref{lem:path-deficit-r1} gives
\begin{equation}\label{eq:path-comparator-q-plus-one}
        R_1(P_n;x)
        \le
        \sum_{i=1}^qR_1(P_{n_i};x)+r_1((q+1)x),
\end{equation}
where $n=1+\sum_i n_i$. Combining \eqref{eq:induction-components}, \eqref{eq:gain-q-plus-one}, and \eqref{eq:path-comparator-q-plus-one}, we obtain
\[
        R_1(G;x)\ge R_1(P_n;x).
\]
This completes the induction.
\end{proof}

\section{The second-order stop-loss framework for \texorpdfstring{$p\ge4$}{p>=4}}\label{sec:sl2-framework}

Section~\ref{sec:R1-comparison} handled the whole range \(2<p<4\) through the
kernel \(R_1\).  For the range \(p\ge4\), we prove a stronger statement:
not only power functions, but a whole cone of test functions can be compared.
This comparison is based on the second-order stop-loss quantities of the
squared singular values.

For a bipartite graph $G$ define the atomic measure of its squared singular values by
\[
        \nu_G:=\sum_i\delta_{\mu_i(G)}.
\]
The sum may be taken over the positive squared singular values only. Equivalently, one may include additional zero squared singular values.  This convention makes no difference for any of the quantities \(\mathsf S_t\), since
\((0-t)_+^2=0\) for \(t\ge0\).

\begin{definition}[Second-order stop-loss quantities]
For \(t\ge0\), define \emph{the second-order stop-loss quantity of \(G\) at level \(t\)} by
\[
        \mathsf S_t(G):=
        \int_{[0,\infty)}(\lambda-t)_+^2\,d\nu_G(\lambda)
        =
        \sum_i(\mu_i(G)-t)_+^2 .
\]
We say that \emph{\(G\) dominates \(H\) in the second-order stop-loss sense} if
\[
        \mathsf S_t(G)\ge \mathsf S_t(H)\qquad(t\ge0).
\]
\end{definition}

\begin{definition}[Admissible third-order convex functions]
A function \(F:[0,\infty)\to\mathbb R\) is called
\emph{admissible third-order convex} if
\[
        F\in C^2([0,\infty)),\qquad
        F''\in AC_{\mathrm{loc}}([0,\infty)),
\]
where \(AC_{\mathrm{loc}}([0,\infty))\) means absolute continuity on every compact subinterval of \([0,\infty)\), and
\[
        F(0)=0,\qquad F'(0)\ge0,\qquad F''(0)\ge0,
\]
while the a.e.\ derivative \(F'''\) satisfies
\[
        F'''(t)\ge0
        \qquad\text{for a.e. }t>0.
\]
\end{definition}

For every admissible third-order convex function \(F\), Taylor's formula with
integral remainder gives
\[
        F(\lambda)=F'(0)\lambda+\frac{F''(0)}2\lambda^2
        +\frac12\int_0^\infty(\lambda-t)_+^2F'''(t)\,dt
        \qquad(\lambda\ge0).
\]
Indeed, the integral over \([0,\infty)\) is the same as the integral over
\([0,\lambda]\), since \((\lambda-t)_+=0\) for \(t\ge\lambda\).

The next lemma is the basic comparison principle that converts second-order stop-loss domination, together with the first-moment comparison, into domination for admissible third-order convex test functions.

\begin{lemma}\label{lem:third-order-convex-from-stoploss}
Let \(G\) and \(H\) be bipartite graphs. Suppose that $\sum_i\mu_i(G)\ge \sum_i\mu_i(H)$ and that $\mathsf S_t(G)\ge \mathsf S_t(H)$ for every $t\ge0$. Then, for every admissible third-order convex \(F\),
\[
        \sum_i F(\mu_i(G))\ge \sum_iF(\mu_i(H)).
\]
\end{lemma}

\begin{proof}
Applying the Taylor representation above to each squared singular value of \(G\) and \(H\), and then subtracting, we obtain
\[
\begin{aligned}
\sum_iF(\mu_i(G))-\sum_iF(\mu_i(H))
&=
F'(0)\left(\sum_i\mu_i(G)-\sum_i\mu_i(H)\right)\\
&\quad+\frac{F''(0)}2\left(\mathsf S_0(G)-\mathsf S_0(H)\right)\\
&\quad+\frac12\int_0^\infty
        \left(\mathsf S_t(G)-\mathsf S_t(H)\right)F'''(t)\,dt .
\end{aligned}
\]
The first term is non-negative by the first-moment hypothesis and
\(F'(0)\ge0\).  The second term is non-negative because the stop-loss
comparison at \(t=0\) gives \(\mathsf S_0(G)\ge\mathsf S_0(H)\), and
\(F''(0)\ge0\).  The integral term is non-negative because
\(\mathsf S_t(G)\ge\mathsf S_t(H)\) for every \(t\ge0\) and
\(F'''(t)\ge0\) almost everywhere.  This completes the proof.
\end{proof}

\begin{remark}\label{rem:powers-admissible}
For \(p\ge4\), the function \(F(\lambda)=\lambda^{p/2}\) is admissible third-order convex. Indeed, with \(\alpha=p/2\ge2\), one has
\(F'(0)\ge0\), \(F''(0)\ge0\), and
\[
        F'''(\lambda)=\alpha(\alpha-1)(\alpha-2)\lambda^{\alpha-3}\ge0
\]
for \(\lambda>0\), with the case \(\alpha=2\) giving \(F'''\equiv0\).
\end{remark}

The main bipartite theorem in the range $p\ge4$ is the following.

\begin{theorem}\label{thm:bipartite-stoploss}
Let \(G\) be a connected bipartite graph on \(n\) vertices. Then
\[
        \mathsf S_t(G)\ge \mathsf S_t(P_n)
        \qquad(t\ge0).
\]
\end{theorem}

Combining this stop-loss comparison with the first-moment inequality and the bipartite reduction gives the following comparison, valid for all connected graphs.

\begin{corollary}\label{cor:connected-third-order-convex}
Let \(G\) be a connected graph on \(n\) vertices.  Then, for every admissible third-order convex function \(F\),
\[
        \sum_{i=1}^n F \left(\lambda_i(G)^2\right)
        \ge
        \sum_{i=1}^n F \left(\lambda_i(P_n)^2\right).
\]
\end{corollary}

\begin{proof}
First suppose that \(G\) is connected and bipartite. By Theorem~\ref{thm:bipartite-stoploss}, $\mathsf S_t(G)\ge \mathsf S_t(P_n)$ for every $t\ge0$. Moreover,
\[
        \sum_i\mu_i(G)=|E(G)|\ge n-1=\sum_i\mu_i(P_n).
\]
Lemma~\ref{lem:third-order-convex-from-stoploss}, applied with \(H=P_n\), gives
\[
        \sum_iF(\mu_i(G))\ge\sum_iF(\mu_i(P_n)).
\]
The adjacency eigenvalues of a bipartite graph are $\pm\sqrt{\mu_i}$, together with possible additional zeros. Since \(F(0)=0\), we have $\sum_{j=1}^n F\left(\lambda_j(G)^2\right) = 2\sum_iF(\mu_i(G))$, and the same identity holds for \(P_n\).  Therefore
\[
        \sum_{j=1}^n F\left(\lambda_j(G)^2\right)
        \ge
        \sum_{j=1}^n F\left(\lambda_j(P_n)^2\right).
\]

Now let \(G\) be an arbitrary connected graph. Since \(F\) is admissible, \(F'''\ge0\) a.e., \(F''(0)\ge0\), and \(F'(0)\ge0\).  Hence \(F''\ge0\) and \(F'\ge0\) on \([0,\infty)\). It follows that $\widetilde F(x):=F(x^2)$ is an even convex function on \(\mathbb R\). By Lemma~\ref{lem:bipartite-reduction}, there exists a connected bipartite spanning subgraph \(H\) of \(G\) such that
\[
        \sum_i F\left(\lambda_i(H)^2\right)
        =
        \sum_i \widetilde F(\lambda_i(H))
        \le
        \sum_i \widetilde F(\lambda_i(G))
        =
        \sum_i F\left(\lambda_i(G)^2\right).
\]
Applying the bipartite case to \(H\) gives
\[
        \sum_i F\left(\lambda_i(H)^2\right)
        \ge
        \sum_i F\left(\lambda_i(P_n)^2\right).
\]
Combining the last two inequalities proves the claim.
\end{proof}

The remainder of this section and the next two sections are devoted to the proof of Theorem~\ref{thm:bipartite-stoploss}.  The argument follows the same deletion-induction strategy as the \(R_1\) comparison, but with the analytic inputs replaced by second-order stop-loss estimates.

\subsection{Tree shifts at the second-order stop-loss level}\label{subsec:tree-shift-sl2}

Recall from Section~\ref{subsec:stieltjes-prelim} the cones \(\mathcal C_{\le k}\), their decaying subcones \(\mathcal C_{\le k}^0\), and the Stieltjes transform \(\mathcal S_\tau\). The deletion ratios used below naturally lie in the cones \(\mathcal C_{\le1}\) and \(\mathcal C_{\le2}\), while the final combined quantity will lie in the decaying cone \(\mathcal C_{\le3}^0\).  The following lemma is the analytic device that converts this cone membership into the needed stop-loss inequality.

\begin{lemma}\label{lem:C3-stoploss}
Let \(g\in\mathcal C_{\le3}^0\). Suppose that a finite signed atomic measure
\[
        \sigma=m_0\delta_0+\sum_{\ell=1}^N m_\ell\delta_{r_\ell},
        \qquad m_0,m_\ell\in\mathbb R,\quad r_\ell>0,
\]
satisfies $\mathcal{S}_\sigma(x)=xg(x)$ for every $x>0$. Then, for every \(t\ge0\),
\[
        \int_{[0,\infty)} (\lambda-t)_+\,d\sigma(\lambda)\le0.
\]
\end{lemma}
\begin{proof}
Fix \(t\ge0\).  For \(\varepsilon>0\), put
\[
        h_{\varepsilon,t}(\lambda)
        :=
        \varepsilon\log\bigl(1+\exp((\lambda-t)/\varepsilon)\bigr),
        \qquad \lambda\ge0 .
\]
Then \(h_{\varepsilon,t}\) is smooth, non-decreasing and convex on \([0,\infty)\), and $h_{\varepsilon,t}(\lambda)\longrightarrow(\lambda-t)_+$ locally uniformly as \(\varepsilon\downarrow0\).  We shall also use $\phi_{\varepsilon,t}(u):=u h_{\varepsilon,t}(u)$. Since
\[
        \phi_{\varepsilon,t}''(u)
        =
        2h_{\varepsilon,t}'(u)+u h_{\varepsilon,t}''(u)\ge0
        \qquad (u\ge0),
\]
the function \(\phi_{\varepsilon,t}\) is convex on \([0,\infty)\).

Since \(g\in\mathcal C_{\le3}^0\), it can be written as a finite nonnegative linear combination of nonconstant generators:
\[
        g(x)=\sum_{\nu=1}^M \alpha_\nu q_\nu(x),
        \qquad \alpha_\nu\ge0,
\]
where
\[
        q_\nu(x)=
        \prod_{j=1}^{m_\nu}\frac1{1+s_{\nu,j}x},
        \qquad
        s_{\nu,j}>0,\quad 1\le m_\nu\le3.
\]
For a smooth function \(F\), the first divided difference \([a,b]F\) is the
slope of the secant line through \((a,F(a))\) and \((b,F(b))\):
\[
        [a,b]F:=\frac{F(b)-F(a)}{b-a}.
\]
The second divided difference measures the change of these secant slopes:
\[
        [a,b,c]F:=\frac{[b,c]F-[a,b]F}{c-a}.
\]
When some nodes coincide, we use the standard confluent extension.  Thus
\[
        [a,a]F=F'(a),
        \qquad
        [a,a,c]F=\frac{[a,c]F-F'(a)}{c-a},
        \qquad
             [a,a,a]F=\frac12F''(a).
\]
For a generator
\[
        q(x)=\prod_{j=1}^{m}\frac1{1+s_jx},
        \qquad 1\le m\le3,
\]
with the available parameters ordered as $s_1\le s_2\le\cdots\le s_m$, define a linear functional on smooth functions \(h\) by
\[
        \Lambda_q(h)
        :=
        \begin{cases}
        \displaystyle
        \frac{h(0)-h(s_1)}{s_1},
        & m=1,\\[2ex]
        \displaystyle
        -[s_1,s_2]h,
        & m=2,\\[1.2ex]
        \displaystyle
        -[s_1,s_2,s_3]\bigl(u\mapsto u h(u)\bigr),
        & m=3.
        \end{cases}
\]
This is well-defined also for repeated parameters by the confluent
interpretation above.

Let
\[
        k_x(\lambda):=\frac1{1+x\lambda},
        \qquad x>0.
\]
A direct divided-difference calculation gives
\begin{equation}\label{eq:Lambda_q_kx}
\Lambda_q(k_x)=xq(x) \qquad (x>0).
\end{equation}
We next prove the sign of each generator on \(h_{\varepsilon,t}\).  If
\(m=1\), then
\[
        \Lambda_q(h_{\varepsilon,t})
        =
        \frac{h_{\varepsilon,t}(0)-h_{\varepsilon,t}(s_1)}{s_1}
        \le0
\]
because \(h_{\varepsilon,t}\) is non-decreasing.  If \(m=2\), then
\[
        \Lambda_q(h_{\varepsilon,t})
        =
        -[s_1,s_2]h_{\varepsilon,t}\le0,
\]
again because first divided differences of a non-decreasing function are
nonnegative, including the confluent case.  If \(m=3\), then
\[
        \Lambda_q(h_{\varepsilon,t})
        =
        -[s_1,s_2,s_3]\phi_{\varepsilon,t}.
\]
If \(s_1<s_2<s_3\), then convexity of \(\phi_{\varepsilon,t}\) means that its
secant slopes are non-decreasing, so
$
        [s_1,s_2]\phi_{\varepsilon,t}
        \le
        [s_2,s_3]\phi_{\varepsilon,t}.
$ Hence
\[
        [s_1,s_2,s_3]\phi_{\varepsilon,t}
        =
        \frac{[s_2,s_3]\phi_{\varepsilon,t}-[s_1,s_2]\phi_{\varepsilon,t}}{s_3-s_1}
        \ge0,
\]
and therefore \(\Lambda_q(h_{\varepsilon,t})\le0\).  For repeated nodes, the result follows from the same confluent interpretation of divided differences. 

Finally, we define
\[
        \Lambda_g(h):=\sum_{\nu=1}^M \alpha_\nu\Lambda_{q_\nu}(h).
\]
Then
\begin{equation}\label{eq:Lambda_g_sign}
\Lambda_g(h_{\varepsilon,t})\le0
\end{equation}
and, by \eqref{eq:Lambda_q_kx},
\begin{equation}\label{eq:Lambda_g_kx}
\Lambda_g(k_x)=xg(x) \qquad (x>0).
\end{equation}

It remains to identify \(\Lambda_g\) with integration against the given
atomic measure \(\sigma\).  This is the only point where one has to be careful
about repeated parameters and higher-order poles.

Let \(A\subset(0,\infty)\) be the finite set of all $s_{\nu,j}$ among the generators \(q_\nu\).  Since \(\Lambda_g\) is built from divided
differences of order at most two at these $s_{\nu,j}$'s, there exist real coefficients \(b_0\) and
\(b_{a,j}\), \(a\in A\), \(0\le j\le2\), such that for every smooth \(h\),
\begin{equation*}
\Lambda_g(h) = b_0h(0)+\sum_{a\in A}\sum_{j=0}^2 b_{a,j}h^{(j)}(a).
\end{equation*}
Applying this to \(k_x\) gives the rational function
\begin{equation}\label{eq:Lambda_g_rational}
\Lambda_g(k_x) = b_0 + \sum_{a\in A}\sum_{j=0}^2 b_{a,j}\, \frac{(-1)^j j! x^j}{(1+ax)^{j+1}}.
\end{equation}
After merging equal support points of \(\sigma\), write
\[
        \sigma=m_0\delta_0+\sum_{\ell = 1}^N m_\ell\delta_{r_\ell},
        \qquad
        r_\ell>0.
\]
Then
\begin{equation}\label{eq:atomic-stieltjes-expansion}
\mathcal S_\sigma(x) = m_0+\sum_{\ell=1}^N\frac{m_\ell}{1+r_\ell x}.
\end{equation}
By assumption and \eqref{eq:Lambda_g_kx}, the rational functions in
\eqref{eq:Lambda_g_rational} and \eqref{eq:atomic-stieltjes-expansion} agree
for all \(x>0\). Comparing their limits as \(x\to\infty\) gives \(b_0=m_0\).
Since the poles of \(\mathcal S_\sigma\) are simple, we also have
\(b_{a,j}=0\) for every \(j\ge1\).

Comparing the remaining simple pole coefficients gives $b_{r_\ell,0}=m_\ell$, and all other coefficients $b_{a,0}$ for $a\in A\setminus\{r_1,\ldots,r_N\}$ vanish.  Thus we obtain for every smooth \(h\),
\[
        \Lambda_g(h)
        =
        m_0h(0)+\sum_{\ell=1}^N m_\ell h(r_\ell)
        =
        \int_{[0,\infty)} h(\lambda)\,d\sigma(\lambda).
\]
Applying this to \(h=h_{\varepsilon,t}\) and using \eqref{eq:Lambda_g_sign}, we get
\[
        \int_{[0,\infty)} h_{\varepsilon,t}(\lambda)\,d\sigma(\lambda)
        =
        \Lambda_g(h_{\varepsilon,t})
        \le0.
\]
Finally let \(\varepsilon\downarrow0\).  Since \(\sigma\) has finite support and \(h_{\varepsilon,t}\to(\lambda-t)_+\) uniformly on finite sets,
\[
        \int_{[0,\infty)} (\lambda-t)_+\,d\sigma(\lambda)
        =
        \lim_{\varepsilon\downarrow0}
        \int_{[0,\infty)} h_{\varepsilon,t}(\lambda)\,d\sigma(\lambda)
        \le0.
\]
This proves the lemma.
\end{proof}

The following lemma gives the cone membership for deletion ratios of weighted forests.  In the applications below we shall only need the cases \(|U|=1\) and \(|U|=2\), but the same argument proves the general statement.

\begin{lemma}\label{lem:one-two-deletion-cones}
Let \(Y\) be a weighted forest with non-negative edge weights, and let
\(U\) be a set of vertices contained in one side of a bipartition of
\(Y\). Then
\[
        \frac{M_{Y-U}(x)}{M_Y(x)}\in\mathcal C_{\le|U|}.
\]
\end{lemma}

\begin{proof}
It suffices to consider the case where \(U\) is contained in the left side of the bipartition; the case where \(U\) is contained in the right side is obtained by replacing \(B\) with \(B^\top\).  Let \(L\) be the left vertex set, and let \(B=B_{Y}\) be the weighted biadjacency matrix from Lemma~\ref{lem:forest-det-matching}.  Set
$
        A:=BB^\top \succeq 0.
$
By \eqref{eq:forest-det-matching-r1},
\[
        M_Y(x)=\det(I+xA).
\]
Write \(U^c:=L\setminus U\).  Deleting the vertices in \(U\) corresponds to deleting the corresponding rows of \(B\).  Thus the weighted biadjacency matrix of \(Y-U\), with respect to the bipartition \(U^c\sqcup R\), is \(B[U^c,R]\). Hence
\[
        M_{Y-U}(x)
        =
        \det\bigl(I+xB[U^c,R]B[U^c,R]^\top\bigr)
        =
        \det\bigl((I+xA)_{U^c,U^c}\bigr).
\]
Applying the principal-minor case of Jacobi's identity for complementary minors \cite[Eq.~(12)]{BrualdiSchneider1983} to \(M=I+xA\), we obtain, for \(x>0\),
\[
        \frac{M_{Y-U}(x)}{M_Y(x)}
        =
        \frac{\det((I+xA)_{U^c,U^c})}{\det(I+xA)}
        =
        \det\bigl([(I+xA)^{-1}]_{U,U}\bigr).
\]
It remains to expand this determinant spectrally. Let
\[
        A=\sum_{\alpha=1}^{|L|} s_\alpha q_\alpha q_\alpha^\top,
        \qquad s_\alpha\ge0,
\]
be an orthonormal spectral decomposition of \(A\). Set $T_x:=(I+xA)^{-1}$, so $T_x q_\alpha = \frac{1}{1+s_\alpha x}q_\alpha$.

If \(U=\{u_1,\ldots,u_r\}\), write
\[
        e_U:=e_{u_1}\wedge\cdots\wedge e_{u_r}.
\]
Then the principal minor has the exterior-power interpretation
\[
        \det([T_x]_{U,U})
        =
        \langle e_U,\wedge^r T_x(e_U)\rangle.
\]
Since the vectors
\[
        q_{\alpha_1}\wedge\cdots\wedge q_{\alpha_r},
        \qquad \alpha_1<\cdots<\alpha_r,
\]
form an orthonormal eigenbasis for \(\wedge^r T_x\), with eigenvalues
$
        \prod_{\ell=1}^r \frac{1}{1+s_{\alpha_\ell}x},
$
we get
\[
\begin{aligned}
        \det([T_x]_{U,U})
        &=
        \sum_{\alpha_1<\cdots<\alpha_r}
        \frac{
        \bigl\langle e_U,\,
        q_{\alpha_1}\wedge\cdots\wedge q_{\alpha_r}
        \bigr\rangle^2
        }
        {\prod_{\ell=1}^r(1+s_{\alpha_\ell}x)}.
\end{aligned}
\]
All coefficients are squares, hence non-negative.  Therefore
$
        \frac{M_{Y-U}(x)}{M_Y(x)}
$
is a non-negative linear combination of products of at most \(|U|\)
first-order resolvent factors.  Thus it lies in \(\mathcal C_{\le|U|}\).
\end{proof}

We shall need to differentiate real-rooted polynomial families through their root parameters. The needed root regularity is supplied by Bronshtein's theorem on hyperbolic polynomials; we record the short consequence needed below. Here and below, \emph{locally Lipschitz} means Lipschitz on every compact subinterval.

\begin{lemma}\label{lem:root-velocity-sl2}
Let \(I\subset\mathbb R\) be an open interval and let
\[
P_\theta(x)=\sum_{k=0}^D a_k(\theta)x^k, \qquad a_0(\theta)\equiv1,
\]
be a family of polynomials of degree at most \(D\), where each coefficient function \(a_k:I\to\mathbb R\) is of class \(C^D\). Assume that, for every \(\theta\in I\), all zeros of \(P_\theta\) are real and non-positive. Then one can choose locally Lipschitz functions $\mu_1,\ldots,\mu_D:I\to[0,\infty)$ such that
\[
P_\theta(x)=\prod_{j=1}^D \left(1+x\mu_j(\theta)\right).
\]
Moreover, there is a set \(E\subset I\) such that \(I\setminus E\) has Lebesgue measure zero and, for every \(\theta\in E\), all derivatives \(\mu_j'(\theta)\) exist and, for every \(x>0\),
\[
\frac{\partial_\theta P_\theta(x)}{P_\theta(x)} = \sum_{j=1}^D \frac{x\mu_j'(\theta)}{1+x\mu_j(\theta)}.
\]
\end{lemma}

\begin{proof}
Define the reversed polynomial
\[
        \widetilde P_\theta(y):=y^D P_\theta(-1/y).
\]
Since \(a_0(\theta)\equiv1\), this is a monic polynomial of degree \(D\) whose
coefficients are \(C^D\) functions of \(\theta\).  Its roots are real and
non-negative; possible drops in the degree of \(P_\theta\) contribute additional
zero roots of \(\widetilde P_\theta\).

Recall that a univariate polynomial is called \emph{hyperbolic} if all its roots are real. Here \(C^{m,1}\) means \(C^m\) with Lipschitz \(m\)-th derivative. Since the coefficients of \(\widetilde P_\theta\) are \(C^D\), they are \(C^{D-1,1}\) on every compact subinterval of \(I\). By Bronshtein's theorem for hyperbolic polynomials~\cite{Bronshtein1979}, in the form stated in \cite[Theorem~2.1]{ParusinskiRainer2015}, the roots of \(\widetilde P_\theta\) can be labeled by locally Lipschitz functions \(\mu_1,\ldots,\mu_D:I\to[0,\infty)\). Thus $\widetilde P_\theta(y)=\prod_{j=1}^D(y-\mu_j(\theta))$. Substituting \(y=-1/x\) gives
\[
        P_\theta(x)=\prod_{j=1}^D(1+x\mu_j(\theta)).
\]

By Rademacher's theorem, each locally Lipschitz function \(\mu_j\) is differentiable almost everywhere on \(I\).  Since there are only finitely many functions \(\mu_1,\ldots,\mu_D\), there exists a set \(E\subset I\) such that \(I\setminus E\) has Lebesgue measure zero and all derivatives \(\mu_j'(\theta)\) exist for every \(\theta\in E\). For \(\theta\in E\), differentiating the displayed product identity with respect to \(\theta\) gives
\[
        \partial_\theta P_\theta(x)
        =
        \sum_{j=1}^D
        x\mu_j'(\theta)
        \prod_{\ell\ne j}(1+x\mu_\ell(\theta)).
\]
Since \(x>0\) and \(\mu_j(\theta)\ge0\), all factors are positive, and division by \(P_\theta(x)\) gives the claimed logarithmic derivative identity.
\end{proof}

\begin{corollary}\label{cor:root-stoploss-derivative}
With the notation of Lemma~\ref{lem:root-velocity-sl2}, fix \(t\ge0\) and define
\[
        F_t(\theta):=\sum_{j=1}^D(\mu_j(\theta)-t)_+^2 .
\]
Then \(F_t\) is locally Lipschitz on \(I\), and for almost every
\(\theta\in I\),
\[
        F_t'(\theta)
        =
        2\sum_{j=1}^D(\mu_j(\theta)-t)_+\mu_j'(\theta).
\]
\end{corollary}

\begin{proof}
The function \(\phi_t(\lambda):=(\lambda-t)_+^2\) is \(C^1\), with $\phi_t'(\lambda)=2(\lambda-t)_+$. Since each \(\mu_j\) is locally Lipschitz, so is \(\phi_t\circ\mu_j\). The chain rule for Lipschitz functions gives
\[
        \frac{d}{d\theta}\phi_t(\mu_j(\theta))
        =
        2(\mu_j(\theta)-t)_+\mu_j'(\theta)
\]
for almost every \(\theta\).  Summing over \(j\) gives the result.
\end{proof}

We now prove that the one-branch tree shift is monotone for the
second-order stop-loss quantities.

\begin{theorem}\label{thm:tree-shift-sl2}
Let \(T\) contain two pendant paths
\[
        vu_1\cdots u_a,
        \qquad
        vw_1\cdots w_b,
        \qquad a,b\ge1,
\]
and let \(T'\) be obtained from \(T\) by deleting \(vw_1\) and adding
\(u_aw_1\), exactly as in Lemma~\ref{lem:one-shift}.  Then
\[
        \mathsf S_t(T)\ge \mathsf S_t(T')
        \qquad(t\ge0).
\]
\end{theorem}

\begin{proof}
By Lemma~\ref{lem:forest-det-matching}, for a forest $\cF$, we can identify the matching polynomial $M_{\cF}$ with the determinant $\det(I + xB_{\cF}B_{\cF}^\top)$, where $B_{\cF}$ is the weighted biadjacency matrix of $\cF$.  Hence for trees $T$ and $T'$, if we can factorize $M_T(x) = \prod_{j=1}^D(1+x\mu_j(0))$ and $M_{T'}(x) = \prod_{j=1}^D(1+x\mu_j(1))$, for some $D$, possibly padded with zero $\mu_j$'s, then $\mathsf S_t(T)\ge \mathsf S_t(T')$ is equivalent to $\sum_{j=1}^D(\mu_j(0)-t)_+^2 \ge \sum_{j=1}^D(\mu_j(1)-t)_+^2$ for every $t\ge 0$.  We will construct a family of polynomials $M_\theta$ interpolating between $M_T$ and $M_{T'}$, and apply Lemma~\ref{lem:root-velocity-sl2} to obtain the desired stop-loss inequality.

Use the notation from Lemma~\ref{lem:one-shift}. Let \(C\) be the subtree obtained
from \(T\) by deleting the vertices \(u_1,\ldots,u_a,w_1,\ldots,w_b\), and put
\[
        R:=M_{C-v},
        \qquad
        W:=\frac{M_C-M_{C-v}}{x}.
\]
If \(C_1,\ldots,C_s\) are the connected components of \(C-v\) and
\(r_i\in C_i\) is adjacent to \(v\), then by \eqref{eq:matching-recurrence-v} in the proof of Lemma~\ref{lem:one-shift}, together with multiplicativity of matching polynomials over disjoint unions,
\[
W = \frac{M_C-M_{C-v}}{x} = \sum_{i=1}^s M_{C-v-r_i} = \sum_{i=1}^s M_{C_i-r_i}\prod_{\ell\ne i}M_{C_\ell}.
\]
Hence, we know that
\[
        c:=\frac WR
        =
        \sum_{i=1}^s
        \frac{M_{C_i-r_i}}{M_{C_i}},
\]
with the right-hand side interpreted as \(0\) when \(s=0\).  Each summand is a one-vertex deletion ratio. Hence, by Lemma~\ref{lem:one-two-deletion-cones} and the convexity of the cone \(\mathcal C_{\le1}\), we have $c\in\mathcal C_{\le1}$.

Let \(h_m:=M_{P_m}\), with \(h_0=1\), and set
\[
        H:=h_{a+b+1},
        \qquad
        D_0:=h_ah_b,
        \qquad
        D_1:=h_{a+b},
        \qquad
        E:=h_{a-1}h_{b-1}.
\]
The path-concatenation identity \eqref{eq:path-concatenation} gives $D_1=D_0+xE$. For \(0\le\theta\le1\), put
\[
D_\theta:=(1-\theta)D_0+\theta D_1 = D_0+\theta xE.
\]
Label the vertices of \(P_{a+b+1}\) as $z_0,z_1,\ldots,z_{a+b}$. Deleting \(z_a\) leaves \(P_a\sqcup P_b\), while deleting the endpoint \(z_0\) leaves \(P_{a+b}\).  Hence
\[
        \frac{D_0}{H}=\frac{h_ah_b}{h_{a+b+1}},
        \qquad
        \frac{D_1}{H}=\frac{h_{a+b}}{h_{a+b+1}}
\]
are one-vertex deletion ratios in \(P_{a+b+1}\).  Therefore again by Lemma~\ref{lem:one-two-deletion-cones},
\[
        d_\theta:=\frac{D_\theta}{H}
        =
        (1-\theta)\frac{D_0}{H}
        +
        \theta\frac{D_1}{H}
        \in \mathcal C_{\le1}.
\]
Moreover, deleting \(z_{a-1}\) and \(z_{a+1}\) leaves $P_{a-1}\sqcup P_1\sqcup P_{b-1}$. The vertices \(z_{a-1}\) and \(z_{a+1}\) have the same parity, hence lie on the same side of the bipartition.  Since \(h_1=1\), by Lemma~\ref{lem:one-two-deletion-cones}, the two-vertex same-side deletion ratio is
\[
        e:=\frac EH
        =
        \frac{h_{a-1}h_{b-1}}{h_{a+b+1}} \in\mathcal C_{\le2}.
\]
We now define the interpolation, as illustrated in Figure~\ref{fig:tree-shift-interpolation}.

\begin{figure}[htbp]
        \centering
        \includegraphics[width=\textwidth]{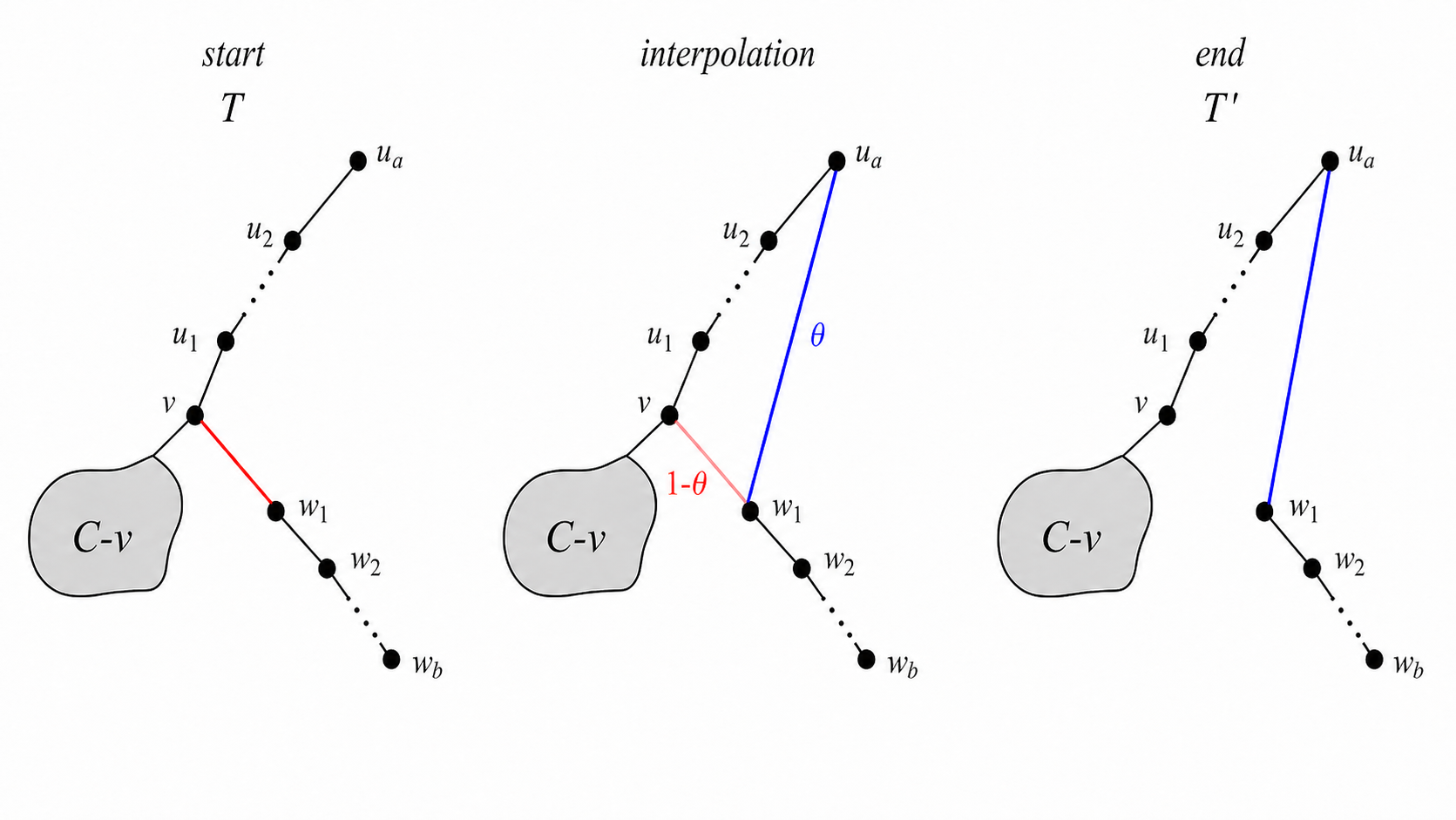}
        \caption{The interpolation from \(T\) to \(T'\): the old edge \(vw_1\) carries weight \(1-\theta\), the new edge \(u_aw_1\) carries weight \(\theta\), and all other edges have weight \(1\).}
        \label{fig:tree-shift-interpolation}
\end{figure}

Let \(G_\theta\) be the non-negatively weighted graph obtained by assigning weight \(1-\theta\) to the old edge \(vw_1\), weight \(\theta\) to the new edge \(u_aw_1\), and weight \(1\) to all other edges.  Let
\[
        M_\theta:=M_{G_\theta}.
\]
Thus \(M_0=M_T\) and \(M_1=M_{T'}\).

Let \(H_\theta\) denote the weighted branch on the vertices
\[
        v,u_1,\ldots,u_a,w_1,\ldots,w_b
\]
with the same edge weights as above.  The two special edges \(vw_1\) and
\(u_aw_1\) share the vertex \(w_1\), so no matching can use both.  Matchings using \(vw_1\) contribute
$
        (1-\theta)x h_a h_{b-1},
$
whereas matchings using \(u_aw_1\) contribute
$
        \theta x h_a h_{b-1}.
$
All matchings using neither special edge are independent of \(\theta\). Therefore
\[
        M_{H_\theta}=M_{H_0}=h_{a+b+1}=H.
\]
On the other hand, after deleting \(v\), the old edge \(vw_1\) disappears, while the new edge \(u_aw_1\) remains with weight \(\theta\). Hence,
\[
        M_{H_\theta-v}
        =
        h_ah_b+\theta x h_{a-1}h_{b-1}
        =
        D_\theta.
\]
Conditioning on whether \(v\) is matched to an edge lying in \(C\), we get
\[
        M_\theta
        =
        R\,M_{H_\theta}
        +
        xW\,M_{H_\theta-v}
        =
        RH+xWD_\theta.
\]
Consequently, $\partial_\theta M_\theta=x^2WE$. 

Set \(Q:=WE\). We next prove that
\[
        \frac{Q}{M_\theta}\in\mathcal C_{\le3}^0.
\]
Indeed, we have
\[
        \frac{Q}{M_\theta}
        =
        \frac{WE}{RH+xWD_\theta}
        =
        \frac{(W/R)(E/H)}
             {1+x(W/R)(D_\theta/H)}
        =
        e\,\frac{c}{1+xcd_\theta}.
\]
We already proved
\[
        c\in\mathcal C_{\le1},
        \qquad
        d_\theta\in\mathcal C_{\le1},
        \qquad
        e\in\mathcal C_{\le2}.
\]
By Lemma~\ref{lem:stieltjes-closure},
\[
        \frac{c}{1+xcd_\theta}\in\mathcal C_{\le1}.
\]
Therefore, using
$
        \mathcal C_{\le2}\cdot\mathcal C_{\le1}
        \subseteq
        \mathcal C_{\le3},
$
we obtain $Q/M_\theta \in\mathcal C_{\le3}$. It remains to check the decay at infinity.  If \(Q=0\), then \(Q/M_\theta=0\in\mathcal C_{\le3}^0\), and there is nothing to prove. Assume \(Q\ne0\).  Then \(W\ne0\), and
\[
        M_\theta=RH+xW(D_0+\theta xE),
        \qquad
        Q=WE.
\]
Since
$
        \deg D_0\ge \deg E,
$
the summand \(xWD_0\) has degree at least
$
        1+\deg W+\deg E
        =
        1+\deg Q.
$
All coefficients are non-negative, so no leading-term cancellation is
possible in \(M_\theta\).  Hence
$
        \deg M_\theta\ge \deg Q+1,
$
and therefore
\[
        \frac{Q}{M_\theta}=O(x^{-1})
        \qquad(x\to\infty).
\]
Thus $Q/M_\theta \in\mathcal C_{\le3}^0$.

Now Theorem~\ref{thm:heilmann-lieb-zero}, applied to \(G_\theta\), shows that \(M_\theta\) has only real negative zeros for
every \(0\le\theta\le1\). Apply Lemma~\ref{lem:root-velocity-sl2} on the open interval \((0,1)\). We write
\[
        M_\theta(x)=\prod_{j=1}^D(1+x\mu_j(\theta)),
\]
where the root parameters are non-negative and zero parameters are allowed in
order to keep the number of factors independent of \(\theta\).  For almost
every \(\theta\in(0,1)\), all derivatives \(\mu_j'(\theta)\) exist and, for
every \(x>0\),
\[
        \sum_{j=1}^D \frac{x\mu_j'(\theta)}{1+x\mu_j(\theta)}
        =
        \frac{\partial_\theta M_\theta(x)}{M_\theta(x)}
        =
        \frac{x^2Q(x)}{M_\theta(x)}.
\]
Thus the finite signed atomic measure $\dot\nu_\theta := \sum_{j=1}^D\mu_j'(\theta)\delta_{\mu_j(\theta)}$ has Stieltjes transform
\[
        \mathcal S_{\dot\nu_\theta}(x)
        =
        \int_{[0,\infty)}\frac1{1+x\lambda}\,d\dot\nu_\theta(\lambda)
        =
        \sum_{j=1}^D \frac{\mu_j'(\theta)}{1+x\mu_j(\theta)}
        =
        x\frac{Q(x)}{M_\theta(x)}.
\]
Since \(Q/M_\theta\in\mathcal C_{\le3}^0\), Lemma~\ref{lem:C3-stoploss} yields
\[
        \int_{[0,\infty)}(\lambda-t)_+\,d\dot\nu_\theta(\lambda)
        \le0
        \qquad(t\ge0)
\]
for almost every \(\theta\in(0,1)\).  By Corollary~\ref{cor:root-stoploss-derivative}, for almost every \(\theta\in(0,1)\),
\[
        \frac{d}{d\theta}
        \sum_{j=1}^D(\mu_j(\theta)-t)_+^2
        =
        2\int_{[0,\infty)}(\lambda-t)_+\,d\dot\nu_\theta(\lambda)
        \le0.
\]
Integrating over \([\varepsilon,1-\varepsilon]\) and then letting
\(\varepsilon\downarrow0\), using continuity of the root parameters at the
endpoints, gives
\[
        \sum_{j=1}^D(\mu_j(0)-t)_+^2
        \ge
        \sum_{j=1}^D(\mu_j(1)-t)_+^2.
\]
Since \(M_0=M_T\), Lemma~\ref{lem:forest-det-matching} gives
\[
        M_0(x)=M_T(x)=\det(I+xB_TB_T^\top).
\]
By the spectral theorem, this determinant factors as $\det(I+xB_TB_T^\top)=\prod_i(1+x\lambda_i)$, where the \(\lambda_i\)'s are the eigenvalues of \(B_TB_T^\top\), equivalently the squared singular values of \(B_T\).  Comparing this factorization with $M_0(x)=\prod_{j=1}^D(1+x\mu_j(0))$, we see that the nonzero parameters \(\mu_j(0)\) are precisely the squared singular values of \(T\), with harmless zero padding.  The same argument applies to \(M_1=M_{T'}\).  Hence the endpoint inequality is exactly
\[
        \mathsf S_t(T)\ge \mathsf S_t(T')
        \qquad(t\ge0).
\qedhere\]
\end{proof}

\begin{corollary}\label{cor:tree-stoploss-path}
For every tree \(T\) on \(n\) vertices,
\[
        \mathsf S_t(T)\ge \mathsf S_t(P_n)
        \qquad(t\ge0).
\]
\end{corollary}

\begin{proof}
By Lemma~\ref{lem:reduce-to-path}, every tree \(T\) can be transformed into \(P_n\) by a finite sequence of the one-branch tree shifts. Applying Theorem~\ref{thm:tree-shift-sl2} at each step and telescoping gives the result.
\end{proof}

\subsection{Path splicing at the second-order stop-loss level}
\label{subsec:a2-sl2}

We first prove a path-specific splicing estimate at the level of the
second-order stop-loss quantities.  The proof uses the explicit spectrum of
paths and a uniform bound for a sawtooth error term.

The derivative of the threshold kernel will produce the following weight. For \(0\le \rho\le 1/2\), put
\[
        w_\rho(u)
        =
        16\pi\sin(2\pi u)
        \bigl(\cos(2\pi u)-\cos(2\pi\rho)\bigr),
        \qquad 0\le u\le \rho.
\]
We also write $\psi(x):=\{x\}-\frac12$, where \(\{x\}=x-\lfloor x\rfloor\) is the fractional part, and define
\[
        J_n(\rho)=\int_0^\rho \psi(nu)w_\rho(u)\,du .
\]

The following estimate is a standard periodic-Bernoulli integration-by-parts
bound.  We include the proof in order to keep the constant explicit.

\begin{lemma}\label{lem:sawtooth-bound}
For every \(\rho\in[1/6,1/2]\) and every integer \(n\ge1\),
\[
        |J_n(\rho)|\le \frac{17\sqrt3\pi^2}{27n^2}.
\]
\end{lemma}

\begin{proof}
Let
\[
        B_2(x)=x^2-x+\frac16,
        \qquad
        B_3(x)=x^3-\frac32x^2+\frac12x
        \qquad(0\le x\le1),
\]
and let \(\widetilde B_2,\widetilde B_3\) be their \(1\)-periodic extensions to \(\mathbb R\). Thus $\widetilde B_2(x)=B_2(\{x\})$ and $\widetilde B_3(x)=B_3(\{x\})$. Away from the points where \(nu\in\mathbb Z\), we have
\[
        \frac{d}{du}\widetilde B_2(nu)=2n\psi(nu),
        \qquad
        \frac{d}{du}\widetilde B_3(nu)=3n\widetilde B_2(nu).
\]

We partition \([0,\rho]\) at the finitely many points \(u\) for which
\(nu\in\mathbb Z\). On each resulting open subinterval, the functions
\(\widetilde B_2(nu)\) and \(\widetilde B_3(nu)\) are ordinary smooth
functions of \(u\), since there \(\{nu\}=nu-k\) for a fixed integer \(k\). Since \(\widetilde B_2\) and \(\widetilde B_3\) are
continuous, the boundary terms at the interior partition points cancel after
summing over all subintervals. Therefore
\[
\begin{aligned}
J_n(\rho)
&=\frac1{2n}\int_0^\rho w_\rho(u)(\widetilde B_2(nu))'\,du  \\
&=\frac1{2n}[w_\rho(u)\widetilde B_2(nu)]_0^\rho
  -\frac1{2n}\int_0^\rho w_\rho'(u)\widetilde B_2(nu)\,du \\
&=-\frac1{2n}\int_0^\rho w_\rho'(u)\widetilde B_2(nu)\,du \\
&=-\frac1{6n^2}\int_0^\rho w_\rho'(u)(\widetilde B_3(nu))'\,du \\
&=-\frac1{6n^2}[w_\rho'(u)\widetilde B_3(nu)]_0^\rho
  +\frac1{6n^2}\int_0^\rho \widetilde B_3(nu)w_\rho''(u)\,du .
\end{aligned}
\]
The first boundary term vanishes because \(w_\rho(0)=w_\rho(\rho)=0\). Also \(\widetilde B_3(0)=0\), and a direct calculation gives $\|\widetilde B_3\|_\infty = \frac{\sqrt3}{36}$. Hence
\[
        |J_n(\rho)|\le
        \frac{\sqrt3}{216n^2}
        \left(|w_\rho'(\rho)|+\int_0^\rho |w_\rho''(u)|\,du\right).
\]

Since \(\rho\in[1/6,1/2]\), we have
\(\cos(2\pi\rho)\in[-1,1/2]\). Writing \(z=\cos(2\pi u)\), direct differentiation gives
\[
        w_\rho'(u)=32\pi^2\left(2z^2-\cos(2\pi\rho)z-1\right),
\]
and therefore
\[
        |w_\rho'(\rho)|=32\pi^2(1-\cos^2(2\pi\rho)).
\]
Moreover,
\[
        w_\rho''(u)
        =
        64\pi^3\sin(2\pi u)\left(\cos(2\pi\rho)-4\cos(2\pi u)\right).
\]
Since \(0\le u\le\rho\le1/2\), the substitution \(z=\cos(2\pi u)\) gives
\[
        \int_0^\rho |w_\rho''(u)|\,du
        =
        32\pi^2\int_{\cos(2\pi\rho)}^1 |\cos(2\pi\rho)-4z|\,dz.
\]
Thus
\[
        |w_\rho'(\rho)|+\int_0^\rho |w_\rho''(u)|\,du
        =
        32\pi^2
        \left(1-\cos^2(2\pi\rho)+\int_{\cos(2\pi\rho)}^1 |\cos(2\pi\rho)-4z|\,dz\right).
\]
If \(\cos(2\pi\rho)\in[0,1/2]\), then the expression in parentheses is
\[
        3-\cos(2\pi\rho)-2\cos^2(2\pi\rho)\le3.
\]
If \(\cos(2\pi\rho)\in[-1,0]\), then it is
\[
        3-\cos(2\pi\rho)+\frac{\cos^2(2\pi\rho)}{4}\le\frac{17}{4}.
\]
Consequently,
\[
        |w_\rho'(\rho)|+\int_0^\rho |w_\rho''(u)|\,du
        \le 136\pi^2.
\]
Therefore
\[
        |J_n(\rho)|
        \le
        \frac{\sqrt3}{216n^2}\cdot136\pi^2
        =
        \frac{17\sqrt3\pi^2}{27n^2}
        <
        \frac{11}{n^2}.
\qedhere\]
\end{proof}

We now prove the path-splicing estimate needed in the deletion step of the second-order stop-loss induction. If \(X\) is a set, we write \(\one_X\) for its indicator function.

\begin{theorem}\label{thm:a2-sl2}
For all integers \(a,b\ge1\) and all \(t\ge0\),
\begin{equation}\label{eq:ststst}
        \mathsf S_t(P_{a+b})-\mathsf S_t(P_a)-\mathsf S_t(P_b)
        \le (4-t)_+^2-(3-t)_+^2.
\end{equation}
\end{theorem}

\begin{proof}
We use the standard formula for the nonzero squared singular values of a path:
\[
        \mu_i(P_m)=4\cos^2\frac{\pi i}{m+1},
        \qquad
        i=1,\ldots,\left\lfloor\frac m2\right\rfloor .
\]
Put
\[
        p=a+1,\qquad q=b+1,\qquad L=p+q-1=a+b+1.
\]
If \(t>4\), then all squared singular values of paths are at most \(4\), and
both sides of the desired inequality are zero. Thus we may assume
\(0\le t\le4\), and write
\[
        t=4\cos^2(\pi\rho),
        \qquad 0\le\rho\le\frac12.
\]
Let \(K(t)\) be the right-hand side minus the left-hand side of the inequality~\eqref{eq:ststst}, and set $K_{p,q}(\rho):=K(4\cos^2(\pi\rho))$. Define $g_\rho(u) := \left(4\cos^2(\pi u)-4\cos^2(\pi\rho)\right)_+^2$. Then the path spectrum formula gives
\[
        K_{p,q}(\rho)
        =
        g_\rho(0)-g_\rho(1/6)
        +\sum_{i=1}^{\lfloor p/2\rfloor}g_\rho(i/p)
        +\sum_{j=1}^{\lfloor q/2\rfloor}g_\rho(j/q)        
        -\sum_{k=1}^{\lfloor L/2\rfloor}g_\rho(k/L).
\]

We use the following elementary summation-by-parts identity. Let \(\{(\alpha,\eta_\alpha)\}_{\alpha\in A}\) be a finite set of pairs with
\(\alpha\in[0,\rho]\), and let \(g\) be \(C^1\) on \([0,\rho]\). If
\(g(\rho)=0\) and \(w(u)=-g'(u)\), then we have
\[
        g(\alpha)=\int_\alpha^\rho w(u)\,du
        =
        \int_0^\rho \mathbf 1_{\alpha\le u}w(u)\,du .
\]
Therefore, using that the sum over \(\alpha\) is finite,
\[
\sum_\alpha \eta_\alpha g(\alpha)
=
\int_0^\rho
\left(\sum_\alpha \eta_\alpha \mathbf 1_{\alpha\le u}\right)w(u)\,du =
\int_0^\rho
\left(\sum_{\alpha\le u}\eta_\alpha\right)w(u)\,du .
\]

Discarding the terms with \(\alpha>\rho\) and applying this identity to the remaining signed atoms
\[
        (0,+1),\qquad (1/6,-1),\qquad
        (i/p,+1),\qquad (j/q,+1),\qquad (k/L,-1),
\]
we obtain
\[
        K_{p,q}(\rho)
        =
        \int_0^\rho A_{p,q}(u)w_\rho(u)\,du,
\]
where the values of \(A_{p,q}\) at the finitely many discontinuity points are irrelevant for the integral, and
\[
        A_{p,q}(u)=
        1-\one_{u\geq 1/6}
        +\lfloor pu\rfloor+\lfloor qu\rfloor-\lfloor Lu\rfloor
        =
        \one_{u<1/6}
        +\lfloor pu\rfloor+\lfloor qu\rfloor-\lfloor Lu\rfloor.
\]
Here
\[
        w_\rho(u)=-g_\rho'(u)
        =
        16\pi\sin(2\pi u)
        \bigl(\cos(2\pi u)-\cos(2\pi\rho)\bigr),
        \qquad 0\le u\le\rho.
\]
Since \(0\le u\le\rho\le1/2\), we have $\sin(2\pi u)\ge0$ and $\cos(2\pi u)-\cos(2\pi\rho)\ge0$. Thus \(w_\rho(u)\ge0\) on \([0,\rho]\).

Because \(L=p+q-1\), we have $Lu=pu+qu-u$. Hence
\[
        \lfloor Lu\rfloor-\lfloor pu\rfloor-\lfloor qu\rfloor
        =
        \bigl\lfloor \{pu\}+\{qu\}-u\bigr\rfloor
        \le1.
\]
Therefore, if \(0\le\rho\le1/6\), then \(u<1/6\) for all
\(0\le u<\rho\), and hence
\[
        A_{p,q}(u)\ge0
        \qquad(0\le u\le\rho)
\]
up to the irrelevant endpoint convention. Since \(w_\rho\ge0\), the integral
representation gives
\[
        K_{p,q}(\rho)\ge0
        \qquad(0\le\rho\le1/6).
\]

It remains to consider \(\rho\in[1/6,1/2]\). Using
\[
        \lfloor pu\rfloor+\lfloor qu\rfloor-\lfloor Lu\rfloor
        =
        u+\{Lu\}-\{pu\}-\{qu\},
\]
and the notation \(\psi(x)=\{x\}-1/2\), we obtain the exact decomposition
\[
        K_{p,q}(\rho)
        =
        M(\rho)+J_L(\rho)-J_p(\rho)-J_q(\rho),
\]
where
\[
        M(\rho)
        =
        \int_0^\rho
        \left(\one_{u<1/6}+u-\frac12\right)w_\rho(u)\,du .
\]
With \(R=2\pi\rho\), direct integration gives
\[
        M(\rho)
        =
        \frac{R\cos 2R+2R-\frac32\sin 2R-\pi\cos 2R}{\pi}.
\]
Moreover \(M\) is increasing on \([1/6,1/2]\). Indeed,
\[
        \frac{dM}{dR}
        =
        \frac{2}{\pi}
        \left(1-\cos 2R+(\pi-R)\sin 2R\right)
        \geq 0.
\]
Hence $M(\rho)\ge M(1/6)=m_0$, where $m_0=1-\frac{3\sqrt3}{4\pi}$.

By Lemma~\ref{lem:sawtooth-bound},
\[
        |J_n(\rho)|\le \frac{17\sqrt3\pi^2}{27}\cdot\frac{1}{n^2}.
\]
If \(p,q\ge7\), then \(L\ge13\), and therefore
\[
        K_{p,q}(\rho)
        \ge
        m_0-\frac{17\sqrt3\pi^2}{27}\left(\frac2{7^2}+\frac1{13^2}\right)>0.
\]

It remains to consider the cases in which \(\min(p,q)\le6\). By symmetry we
may assume \(p\le q\). If \(2\le p\le6<q\), set $D_p(\rho):=M(\rho)-J_p(\rho)$. The finite certificate in Lemma~\ref{lem:a2-finite-certificates}(i), proved in
Appendix~\ref{app:a2sl2-certificates}, gives
\[
        D_p(\rho)
        -
        \frac{17\sqrt3\pi^2}{27}\left(\frac1{7^2}+\frac1{(p+6)^2}\right)
        \ge 10^{-4}
        \qquad
        (p=2,\ldots,6,\; 1/6\le\rho\le1/2).
\]
Since \(q\ge7\) and \(p+q-1\ge p+6\), this implies
\[
\begin{aligned}
        K_{p,q}(\rho)
        &\ge
        D_p(\rho)
        -
        \frac{17\sqrt3\pi^2}{27}\left(\frac1{q^2}+\frac1{(p+q-1)^2}\right)\\
        &\ge
        D_p(\rho)
        -
        \frac{17\sqrt3\pi^2}{27}\left(\frac1{7^2}+\frac1{(p+6)^2}\right)
        \ge 10^{-4}>0 .
\end{aligned}
\]
Finally, if \(2\le p\le q\le6\), Lemma~\ref{lem:a2-finite-certificates}(ii)
gives
\[
        K_{p,q}(\rho)\ge 10^{-4}
        \qquad(1/6\le\rho\le1/2).
\]
Thus \(K_{p,q}(\rho)\ge0\) for all \(0\le\rho\le1/2\), which is precisely the
desired inequality.
\end{proof}

\subsection{Pure even cycles}
\label{subsec:cycle-sl2}

We first record a trigonometric pair inequality comparing two adjacent
entries from the squared-singular-value list of an even cycle with the
corresponding two entries from the path.

\begin{lemma}\label{lem:cycle-pair-ineq}
Let \(m\ge2\) and \(1\le h\le\lfloor m/2\rfloor\). Then
\[
\cos^2\frac{(h-1)\pi}{m}
+
\cos^2\frac{h\pi}{m}
\ge
\cos^2\frac{(2h-1)\pi}{2m+1}
+
\cos^2\frac{2h\pi}{2m+1}.
\]
\end{lemma}

\begin{proof}
Put $\alpha=\pi/m$, $\beta=\pi/(2m+1)$, $u=(2h-1)\alpha$ and $v=(4h-1)\beta$. Using
\[
        \cos^2 x+\cos^2 y
        =
        1+\cos(x+y)\cos(x-y),
\]
the desired inequality is equivalent to
\begin{equation}\label{eq:cycle-product-ineq}
        \cos u\cos\alpha\ge \cos v\cos\beta .
\end{equation}
We have \(0<\beta<\alpha\), and
\[
        v-u
        =
        \frac{(m+1-2h)\pi}{m(2m+1)}
        >0.
\]
Set $r=v-u$ and $s=\alpha-\beta$. Then \(0<r<s\).

We prove \eqref{eq:cycle-product-ineq} by considering the positions of
\(u\) and \(v\) relative to \(\pi/2\). First suppose that \(u\ge\pi/2\).
Since \(v=u+r\) and \(\alpha=\beta+s\), we have
\[
\begin{aligned}
\cos u\cos\alpha-\cos v\cos\beta
&=
\sin u\cos\beta\sin r
-\cos u\sin\beta\sin s  \\
&\qquad
-\cos u\cos\beta(\cos r-\cos s).
\end{aligned}
\]
Here \(\sin u\ge0\), \(\cos\beta>0\), \(\sin r>0\),
\(-\cos u\ge0\), and \(\cos r-\cos s>0\). Hence the right-hand side is
nonnegative.

Next suppose that \(u<\pi/2\le v\). Then
\(\cos u\cos\alpha>0\), while \(\cos v\cos\beta\le0\), so
\eqref{eq:cycle-product-ineq} is immediate.

It remains to consider the case \(v<\pi/2\). Then all cosines in
\eqref{eq:cycle-product-ineq} are positive. Taking logarithms, the inequality
\eqref{eq:cycle-product-ineq} is equivalent to
\begin{equation}\label{eq:cycle-tan-integral}
        \int_u^v \tan x\,dx
        \ge
        \int_\beta^\alpha \tan x\,dx .
\end{equation}
Let $M=(u+v)/2$. Since \(M>\alpha>\beta\) and \(x\mapsto \tan x/x\) is increasing on \((0,\pi/2)\), we have
\[
        \tan M\ge\frac{M}{\alpha}\tan\alpha,
        \qquad
        \tan\beta\le\frac{\beta}{\alpha}\tan\alpha .
\]
By convexity of \(\tan x\) on \((0,\pi/2)\),
\[
        \int_u^v\tan x\,dx
        \ge r\tan M,
        \qquad
        \int_\beta^\alpha\tan x\,dx
        \le \frac{s}{2}(\tan\alpha+\tan\beta).
\]
Therefore \eqref{eq:cycle-tan-integral} follows once
\begin{equation}\label{eq:cycle-quadratic-reduction}
        rM\ge \frac{s(\alpha+\beta)}2 .
\end{equation}
Since \(2rM=v^2-u^2\) and \(s(\alpha+\beta)=\alpha^2-\beta^2\), inequality
\eqref{eq:cycle-quadratic-reduction} is equivalent to
\[
        \Delta_{m,h}:=
        m^2\bigl((4h-1)^2+1\bigr)
        -(2m+1)^2\bigl((2h-1)^2+1\bigr)
        \ge0 .
\]

Write \(m=2h+q\). A direct expansion gives
\[
\Delta_{m,h}
=
2\left(
8h^2q+2h^2+4hq^2-4hq-6h-3q^2-4q-1
\right).
\]
The assumption \(v<\pi/2\) implies \(m\ge4h-1\), and hence
\(q\ge2h-1\). If \(h\ge2\), the expression in parentheses is increasing in
\(q\) for \(q\ge2h-1\). At \(q=2h-1\), the expression equals $2h(16h^2-21h+3)>0$. Thus \(\Delta_{m,h}\ge0\) for \(h\ge2\).

For \(h=1\), the same expansion reduces to
\[
        \Delta_{m,1}=2(q^2-5),
        \qquad
        q=m-2.
\]
This is nonnegative for \(q\ge3\), equivalently \(m\ge5\). The only remaining cases in the present case \(v<\pi/2\) are \((m,h)=(3,1)\) and \((m,h)=(4,1)\). A direct substitution into~\eqref{eq:cycle-product-ineq} gives a strict inequality in both cases. This proves the lemma.
\end{proof}

The pure even-cycle comparison follows by grouping the squared singular values into pairs and applying weak majorization.

\begin{theorem}\label{thm:cycle-sl2}
For every \(m\ge2\) and every \(t\ge0\),
\[
        \mathsf S_t(C_{2m})\ge \mathsf S_t(P_{2m}).
\]
\end{theorem}

\begin{proof}
The squared singular values are
\[
        4\cos^2\frac{\pi j}{m},
        \qquad
        j=0,\ldots,m-1,
\]
for \(C_{2m}\), and
\[
        4\cos^2\frac{j\pi}{2m+1},
        \qquad
        j=1,\ldots,m,
\]
for \(P_{2m}\).

Let \(c^\downarrow\) and \(p^\downarrow\) denote the decreasing
rearrangements of these two squared-singular-value vectors. We claim that
\(c^\downarrow\) weakly majorizes \(p^\downarrow\). Indeed,
\[
        c_i^\downarrow
        =
        4\cos^2\frac{\pi\lfloor i/2\rfloor}{m},
        \qquad
        p_i^\downarrow
        =
        4\cos^2\frac{i\pi}{2m+1},
        \qquad
        1\le i\le m.
\]
For \(1\le h\le \lfloor m/2\rfloor\), put
\[
        A_h=4\cos^2\frac{(h-1)\pi}{m},
        \qquad
        B_h=4\cos^2\frac{h\pi}{m},
\]
and
\[
        X_h=4\cos^2\frac{(2h-1)\pi}{2m+1},
        \qquad
        Y_h=4\cos^2\frac{2h\pi}{2m+1}.
\]
The angle comparison gives \(A_h\ge X_h\). Moreover,
Lemma~\ref{lem:cycle-pair-ineq}, multiplied by \(4\), gives
\[
        A_h+B_h\ge X_h+Y_h .
\]
Thus each pair \((A_h,B_h)\) weakly majorizes the pair \((X_h,Y_h)\). Summing these pair inequalities gives all even partial sums:
\[
        \sum_{i=1}^{2r} c_i^\downarrow 
        =\sum_{h=1}^{r} (A_h+B_h)
        \ge \sum_{h=1}^{r} (X_h+Y_h)
        =\sum_{i=1}^{2r} p_i^\downarrow,
        \qquad
        1\le r\le \left\lfloor\frac m2\right\rfloor .
\]

It remains to check the odd partial sums. Let \(r\ge0\) and
\(2r+1\le m\). Then
\[
        \frac{r}{m}<\frac{2r+1}{2m+1},
\]
and hence, by the monotonicity of \(x\mapsto\cos^2(\pi x)\) on
\([0,1/2]\),
\[
        c_{2r+1}^\downarrow
        =
        4\cos^2\frac{r\pi}{m}
        \ge
        4\cos^2\frac{(2r+1)\pi}{2m+1}
        =
        p_{2r+1}^\downarrow .
\]
Adding this inequality to the even partial-sum inequality up to \(2r\) gives
the odd partial sums:
\[
        \sum_{i=1}^{2r+1} c_i^\downarrow
        \ge
        \sum_{i=1}^{2r+1} p_i^\downarrow .
\]
Therefore \(c^\downarrow\) weakly majorizes \(p^\downarrow\).

For fixed \(t\ge0\), the function $\phi_t(\lambda)=(\lambda-t)_+^2$ is increasing and convex on \([0,\infty)\). By the standard characterization of weak submajorization by increasing convex functions \cite[Section~A.4]{MarshallOlkinArnold2011}, we have
\[
        \sum_i \phi_t(c_i^\downarrow)
        \ge
        \sum_i \phi_t(p_i^\downarrow).
\]
This is exactly $\mathsf S_t(C_{2m})\ge \mathsf S_t(P_{2m})$, and the proof is complete.
\end{proof}

\section{Deletion theory at the second-order stop-loss level}
\label{sec:sl2-deletion}

We first turn the effect of adding back a deleted vertex into a
one-dimensional interval estimate.  The rank-one trace formula is the bridge
from the spectral stop-loss functional to this interval estimate.  The rest of
the section then converts the interval estimate into concrete deletion
criteria for the minimal-counterexample argument.

We first record the rank-one trace formula for the stop-loss kernel.  This is
a standard finite-dimensional trace derivative formula, but we include the
short proof for completeness.

\begin{lemma}
\label{lem:rank-one-trace-formula}
Let \(M\) be a real symmetric matrix, let \(b\) be a vector of the same dimension, and set $M_\theta=M+\theta bb^{\top}$ for $0\le\theta\le1$. Then, for every \(t\ge0\),
\[
        \tr(M_1-tI)_+^2-\tr(M_0-tI)_+^2
        =
        2\int_0^1 b^{\top}(M_\theta-tI)_+b\,d\theta .
\]
\end{lemma}

\begin{proof}
Put \(\varphi_t(y)=(y-t)_+^2\).  Then
\(\varphi_t\in C^1(\mathbb R)\) and
\(\varphi_t'(y)=2(y-t)_+\).  Choose \(R>0\) such that
\(\spec(M_\theta)\subset[-R,R]\) for every \(0\le\theta\le1\).
By Weierstrass approximation, there are polynomials \(q_j\) such that
\(q_j\to\varphi_t'\) uniformly on \([-R,R]\). Define
\[
        p_j(y)=\varphi_t(0)+\int_0^y q_j(s)\,ds .
\]
Then \(p_j\to\varphi_t\) and \(p_j'\to\varphi_t'\) uniformly on \([-R,R]\).

For a polynomial \(p\), cyclicity of trace gives
\[
        \frac{d}{d\theta}\tr p(M_\theta)
        =
        \tr\bigl(p'(M_\theta)bb^{\top}\bigr)
        =
        b^{\top}p'(M_\theta)b.
\]
Therefore
\[
        \tr p_j(M_1)-\tr p_j(M_0)
        =
        \int_0^1 b^{\top}p_j'(M_\theta)b\,d\theta .
\]
Letting \(j\to\infty\), using the uniform convergence on the common spectral
interval \([-R,R]\), gives the desired identity.
\end{proof}

We next convert the trace formula into a local lower bound for the effect of adding back a deleted vertex. Throughout, \(N_G(v)\) denotes the open neighborhood of \(v\) in \(G\), that is, the set of vertices adjacent to \(v\). For a vector \(x\), we write \(\|x\|_2\) for its Euclidean norm, so that \(\|x\|_2^2=x^\top x\).

\begin{lemma}
\label{lem:sl2-rank-one}
Let \(G\) be a bipartite graph and let \(v\in V(G)\) be a non-isolated vertex. Choose the bipartition so that \(v\) lies on the left side. Let \(B_{G-v}\) be the corresponding biadjacency matrix of \(G-v\). Let \(b\) be the \(0\)-\(1\) indicator vector of \(N_G(v)\) (indexed by the right-side vertices of \(G-v\)), and set
\[
        M=B_{G-v}^\top B_{G-v},
        \qquad
        d=\|b\|_2^2=d_G(v)>0,
        \qquad
        \gamma=b^\top Mb .
\]
For \(0\le\theta\le1\), put \(M_\theta=M+\theta bb^\top\).  Then, for every
\(t\ge0\),
\[
        \mathsf S_t(G)-\mathsf S_t(G-v)
        \ge
        2\int_{\gamma/d}^{\gamma/d+d}(u-t)_+\,du .
\]
\end{lemma}

\begin{proof}
Since \(v\) lies on the left side, adding \(v\) adds the row \(b^\top\) to the biadjacency matrix:
\[
        B_G=
        \begin{pmatrix}
        B_{G-v}\\
        b^\top
        \end{pmatrix}.
\]
Thus the matrix used to compute the squared singular values changes from
\[
        M_0=B_{G-v}^\top B_{G-v}=M
        \qquad\text{to}\qquad
        M_1=B_{G-v}^\top B_{G-v}+bb^\top=M+bb^\top .
\]
Possible additional zero squared singular values do not contribute to
\(\mathsf S_t\). Hence
\[
        \mathsf S_t(G)
        =
        \tr(M_1-t)_+^2,
        \qquad
        \mathsf S_t(G-v)
        =
        \tr(M_0-t)_+^2 .
\]
By Lemma~\ref{lem:rank-one-trace-formula},
\[
        \mathsf S_t(G)-\mathsf S_t(G-v)
        =
        2\int_0^1 b^\top(M_\theta-t)_+b\,d\theta.
\]

Fix \(\theta\in[0,1]\). The matrix \(M_\theta\) is real symmetric and
positive semidefinite.  Hence, by the spectral theorem, there is a real
orthogonal matrix \(U\) and numbers \(\lambda_i\ge0\) such that
\[
        M_\theta
        =
        U\operatorname{diag}(\lambda_1,\ldots,\lambda_m)U^\top .
\]
Put \(c=U^\top b\).  Since \(U\) is orthogonal, we know that \(b=Uc\), \(\sum_i c_i^2=\|b\|_2^2=d\), and
\[
        b^\top(M_\theta-t)_+b
        =
        \sum_{i=1}^m c_i^2(\lambda_i-t)_+ .
\]
The coefficients \(c_i^2/d\) are non-negative and sum to \(1\).  Applying
Jensen's inequality to the convex function \(\lambda\mapsto(\lambda-t)_+\)
gives
\[
        \frac1d b^\top(M_\theta-t)_+b =
        \sum_{i=1}^m \frac{c_i^2}{d}(\lambda_i-t)_+ \ge
        \left(\sum_{i=1}^m \frac{c_i^2}{d}\lambda_i-t\right)_+  =
        \left(\frac{b^\top M_\theta b}{d}-t\right)_+ .
\]
Moreover,
\[
        b^\top M_\theta b
        =
        b^\top Mb+\theta b^\top(bb^\top)b
        =
        \gamma+\theta d^2,
\]
hence
\[
        b^\top(M_\theta-t)_+b
        \ge
        d\left(\frac{\gamma}{d}+\theta d-t\right)_+ .
\]
Substituting this into the trace formula yields
\[
        \mathsf S_t(G)-\mathsf S_t(G-v)
        \ge
        2\int_0^1
        d\left(\frac{\gamma}{d}+\theta d-t\right)_+\,d\theta .
\]
Finally, with \(u=\gamma/d+\theta d\), we have \(du=d\,d\theta\), and therefore
\[
        \mathsf S_t(G)-\mathsf S_t(G-v)
        \ge
        2\int_{\gamma/d}^{\gamma/d+d}(u-t)_+\,du .
\]
This proves the lemma.
\end{proof}

We shall use the following elementary lower bound on \(\gamma\), obtained by counting the edges incident with \(N_G(v)\) after \(v\) is deleted.

\begin{lemma}
\label{lem:gamma-neighbor-count}
Let the notation be as in Lemma~\ref{lem:sl2-rank-one}.  Let \(s\) be the number of neighbors \(x\in N_G(v)\) such that \(d_{G-v}(x)\ge1\).  Then
\[
        \gamma
        \ge
        \sum_{x\in N_G(v)} (d_G(x)-1)
        \ge
        s .
\]
\end{lemma}

\begin{proof}
Choose the bipartition with \(v\) on the left side.  Then \(b\) is the
\(0\)-\(1\) indicator vector of \(N_G(v)\), indexed by the right-side vertices,
and $M=B_{G-v}^\top B_{G-v}$. Thus
\[
        \gamma
        =
        b^\top B_{G-v}^\top B_{G-v}b
        =
        \|B_{G-v}b\|_2^2 .
\]
For a left-side vertex \(y\) of \(G-v\), the \(y\)-coordinate of
\(B_{G-v}b\) is
\[
(B_{G-v}b)_y =\sum_{x}(B_{G-v})_{yx} b_x = |N_{G-v}(y)\cap N_G(v)|.
\]
Hence
\[
        \gamma
        =
        \sum_y |N_{G-v}(y)\cap N_G(v)|^2,
\]
where \(y\) runs over the left-side vertices of \(G-v\). Since
\(|N_{G-v}(y)\cap N_G(v)|\) is a non-negative integer, we know that
\[
        \gamma
        \ge
        \sum_y |N_{G-v}(y)\cap N_G(v)|
        =
        \sum_{x\in N_G(v)} d_{G-v}(x)
        =
        \sum_{x\in N_G(v)}(d_G(x)-1).
\]
This proves the first inequality.

For the second inequality, by the definition of \(s\), exactly \(s\) vertices
\(x\in N_G(v)\) satisfy \(d_{G-v}(x)\ge1\).  Therefore
\[
        \sum_{x\in N_G(v)}(d_G(x)-1)
        =
        \sum_{x\in N_G(v)}d_{G-v}(x)
        \ge s .
\]
This completes the proof.
\end{proof}

For an interval \([a,b]\subset[0,\infty)\), write
\[
        \mathcal I_{[a,b]}(t):=\int_a^b (u-t)_+\,du.
\]
For fixed \(\ell\ge0\) and \(t\ge0\), we shall refer to the trivial fact
that $\mathcal I_{[a,a+\ell]}(t)$ is non-decreasing as a function of \(a\) as the
\emph{translation monotonicity} of \(\mathcal I\); indeed, this is immediate
from the monotonicity of \(u\mapsto(u-t)_+\).

The next elementary packing estimate says that, among all interval sets of a
fixed total length inside \([0,4]\), the largest value of the stop-loss
integral is obtained by placing the whole set at the right endpoint.

\begin{lemma}
\label{lem:spectral-shift-packing}
Let \(E\subset[0,4]\) be a finite disjoint union of intervals, and let
\(\ell=|E|\) be its total length.  Define
\[
        J_E(t):=\int_E (y-t)_+\,dy .
\]
Then, for every \(t\ge0\),
\[
        J_E(t)\le \mathcal I_{[4-\ell,4]}(t).
\]
\end{lemma}

\begin{proof}
Let \(f_t(y)=(y-t)_+\). Then \(f_t\) is non-decreasing. Since \(E\) and \([4-\ell,4]\) have the same measure, the two sets
\[
        A=E\setminus [4-\ell,4],
        \qquad
        B=[4-\ell,4]\setminus E
\]
also have the same measure. Moreover, $A\subset[0,4-\ell]$ and $B\subset[4-\ell,4]$. Therefore
\[
        \int_A f_t(y)\,dy
        \le |A| f_t(4-\ell)
        =
        |B| f_t(4-\ell)
        \le
        \int_B f_t(y)\,dy .
\]
Adding the common contribution from \(E\cap [4-\ell,4]\), we get
\[
        \int_E f_t(y)\,dy
        \le
        \int_{[4-\ell,4]} f_t(y)\,dy=\mathcal I_{[4-\ell,4]}(t).
\]
This is exactly the desired inequality.
\end{proof}

We shall need two path-cost estimates: one for extending a path by one endpoint,
and one for splicing two paths through a new center.

\begin{lemma}
\label{lem:path-costs-sl2}
For every \(n\ge2\) and every \(t\ge0\),
\[
        \mathsf S_t(P_n)-\mathsf S_t(P_{n-1})
        \le
        2\mathcal I_{[3,4]}(t).
\]
Moreover, for every \(a,b\ge1\) and every \(t\ge0\),
\[
        \mathsf S_t(P_{a+b+1})-\mathsf S_t(P_a)-\mathsf S_t(P_b)
        \le
        2\mathcal I_{[2,4]}(t).
\]
\end{lemma}

\begin{proof}
We prove both estimates in the same way.  Let \(G_0\) be the graph before the
update and \(G_1\) the graph after the update.  In the first case,
\(G_0=P_{n-1}\) and \(G_1=P_n\).  In the second case,
\(G_0=P_a\sqcup P_b\) and \(G_1=P_{a+b+1}\).

Choose the bipartition so that the newly added vertex lies on the left side.
Thus the old right-side vertices remain the column set.  Let \(B_0\) and
\(B_1\) be the biadjacency matrices before and after the update, with this
common column set. Now we set $M_0:=B_0^\top B_0$ and $M_1:=B_1^\top B_1$. Since
\[
        B_1=
        \begin{pmatrix}
        B_0\\
        b^\top
        \end{pmatrix},
\]
we have $M_1=M_0+bb^\top$.

Let
\[
        \alpha_1\ge\alpha_2\ge\cdots\ge0,
        \qquad
        \beta_1\ge\beta_2\ge\cdots\ge0
\]
be the eigenvalues of \(M_0\) and \(M_1\), respectively, padded with zeros if
necessary. Since \(M_1=M_0+bb^\top\) is a rank-one positive semidefinite
perturbation of the Hermitian matrix \(M_0\), the rank-one interlacing theorem
for Hermitian matrices gives
\[
        \beta_1\ge \alpha_1\ge \beta_2\ge \alpha_2\ge\cdots\ge0;
\]
see, for example, \cite[Corollary~4.3.9]{HornJohnson2012}.
Therefore the intervals \([\alpha_i,\beta_i]\) are pairwise disjoint except possibly for common endpoints.

Let $E:=\bigcup_i[\alpha_i,\beta_i]$. Since endpoint overlaps have measure zero, we have
\[
        |E|
        =
        \sum_i(\beta_i-\alpha_i)
        =
        \tr(M_1-M_0)
        =
        \tr(bb^\top)
        =
        \|b\|_2^2 .
\]
Since these eigenvalues are precisely the squared singular
values involved in \(\mathsf S_t\), we have
\[
\begin{aligned}
        \mathsf S_t(G_1)-\mathsf S_t(G_0)
        &=
        \sum_i\Bigl((\beta_i-t)_+^2-(\alpha_i-t)_+^2\Bigr)  \\
        &=
        2\sum_i\int_{\alpha_i}^{\beta_i}(y-t)_+\,dy
        =
        2\int_E (y-t)_+\,dy
        =
        2J_E(t).
\end{aligned}
\]
All graphs appearing in the two updates are paths or disjoint unions of paths.
Hence all their squared singular values lie in \([0,4]\), and so
\(E\subset[0,4]\).

For the endpoint update \(P_{n-1}\rightsquigarrow P_n\), the vector \(b\) has
exactly one nonzero entry. Thus \(\|b\|_2^2=1\). By
Lemma~\ref{lem:spectral-shift-packing}, we know that $J_E(t)\le \mathcal I_{[3,4]}(t)$. This proves
\[
        \mathsf S_t(P_n)-\mathsf S_t(P_{n-1})
        \le
        2\mathcal I_{[3,4]}(t).
\]

For the two-arm update $P_a\sqcup P_b\rightsquigarrow P_{a+b+1}$, we choose the two endpoints to be on the right side and insert a new left-side center adjacent to both of them.  Then \(b\) has exactly two nonzero entries, so \(\|b\|_2^2=2\). Again by Lemma~\ref{lem:spectral-shift-packing}, we know that $J_E(t)\le \mathcal I_{[2,4]}(t)$, which proves
\[
        \mathsf S_t(P_{a+b+1})-\mathsf S_t(P_a)-\mathsf S_t(P_b)
        \le
        2\mathcal I_{[2,4]}(t).
\qedhere\]
\end{proof}

Repeated applications of the preceding path-cost estimates give the following
two bounds.

\begin{corollary}
\label{cor:repeated-path-costs-sl2}
For \(r\ge1\), \(m_1,\ldots,m_r\ge1\), and \(t\ge0\),
\[
        \mathsf S_t(P_{m_1+\cdots+m_r})
        -
        \sum_{i=1}^r\mathsf S_t(P_{m_i})
        \le
        2(r-1)\mathcal I_{[3,4]}(t).
\]
If \(r\ge2\), then
\[
        \mathsf S_t(P_{1+m_1+\cdots+m_r})
        -
        \sum_{i=1}^r\mathsf S_t(P_{m_i})
        \le
        2\mathcal I_{[2,4]}(t)
        +
        2(r-2)\mathcal I_{[3,4]}(t).
\]
\end{corollary}

\begin{proof}
We first prove the first estimate by induction on \(r\).  The case \(r=1\) is
trivial.  Put \(M=m_1+\cdots+m_{r-1}\).  For \(r\ge2\), Theorem~\ref{thm:a2-sl2}
gives
\[
        \mathsf S_t(P_{M+m_r})
        -
        \mathsf S_t(P_M)
        -
        \mathsf S_t(P_{m_r})
        \le
        2\mathcal I_{[3,4]}(t).
\]
Adding the induction hypothesis applied to \(m_1,\ldots,m_{r-1}\) gives
\[
        \mathsf S_t(P_{m_1+\cdots+m_r})
        -
        \sum_{i=1}^r\mathsf S_t(P_{m_i})
        \le
        2(r-1)\mathcal I_{[3,4]}(t).
\]

For the second estimate, first insert one new center between \(P_{m_1}\) and \(P_{m_2}\). By Lemma~\ref{lem:path-costs-sl2},
\[
        \mathsf S_t(P_{1+m_1+m_2})
        -
        \mathsf S_t(P_{m_1})
        -
        \mathsf S_t(P_{m_2})
        \le
        2\mathcal I_{[2,4]}(t).
\]
For each \(k=3,\ldots,r\), Theorem~\ref{thm:a2-sl2} gives
\[
\mathsf S_t(P_{1+m_1+\cdots+m_k})
        -
        \mathsf S_t(P_{1+m_1+\cdots+m_{k-1}})
        -
        \mathsf S_t(P_{m_k})
        \le
        2\mathcal I_{[3,4]}(t).
\]
Summing these inequalities for \(k=3,\ldots,r\) together with the first-center estimate proves
\[
        \mathsf S_t(P_{1+m_1+\cdots+m_r})
        -
        \sum_{i=1}^r\mathsf S_t(P_{m_i})
        \le
        2\mathcal I_{[2,4]}(t)
        +
        2(r-2)\mathcal I_{[3,4]}(t).
\qedhere\]
\end{proof}

We shall also need the following elementary interval domination estimates. They are used to compare the rank-one local gain with the path costs above.

\begin{lemma}
\label{lem:interval-domination}
Let \(d,r\) be integers with \(d\ge r+2\) and \(r\ge1\). Then, for every \(t\ge0\),
\[
        \mathcal I_{[(d-r+1)/d,(d-r+1)/d+d]}(t)\ge r\,\mathcal I_{[3,4]}(t).
\]
In the borderline case \(d=r+2\), one has the sharper estimate
\[
        \mathcal I_{[3/d,d+3/d]}(t)
        \ge
        \mathcal I_{[2,4]}(t)+(d-3)\mathcal I_{[3,4]}(t).
\]
\end{lemma}

\begin{proof}
The proof is elementary but slightly lengthy, and is deferred to
Appendix~\ref{app:interval-domination}.
\end{proof}

The next lemma is the deletion criterion used in the minimal-counterexample
argument.  It says that a vertex whose degree is sufficiently large compared
with the number of components created by its deletion can be removed without
losing the desired stop-loss comparison.  The borderline case \(d=q+1\) is
handled by a two-arm path-cost estimate. A \emph{non-isolated component} means a component that is not a single isolated vertex.

\begin{lemma}
\label{lem:sl2-high-redundancy}
Let \(G\) be connected and bipartite, let \(v\in V(G)\), and put $d=d_G(v)$. Let \(G_1,\ldots,G_q\) be the connected components of \(G-v\).  Assume that
\[
        \mathsf S_t(G_i)\ge \mathsf S_t(P_{|V(G_i)|})
        \qquad(i=1,\ldots,q;\ t\ge0).
\]
If \(d\ge q+2\), then
\[
        \mathsf S_t(G)\ge \mathsf S_t(P_{|V(G)|})
        \qquad(t\ge0).
\]
Moreover, if \(d=q+1\), if \(G-v\) has at least two non-isolated components,
and if one component of \(G-v\) contains at least two neighbors of \(v\), then
\[
        \mathsf S_t(G)\ge \mathsf S_t(P_{|V(G)|})
        \qquad(t\ge0).
\]
\end{lemma}

\begin{proof}
Choose the bipartition so that \(v\) lies on the left side, and let
\(M,b,\gamma\) be as in Lemma~\ref{lem:sl2-rank-one}.  Put
\[
        n_i=|V(G_i)|,
        \qquad
        n=|V(G)|=1+\sum_{i=1}^q n_i.
\]

We first prove the assertion under the hypothesis \(d\ge q+2\).  Let \(s\) be the number of neighbors \(x\in N_G(v)\) such that \(d_{G-v}(x)\ge1\).

The remaining \(d-s\) neighbors of \(v\) are isolated vertices in \(G-v\), and
hence account for \(d-s\) isolated components.  Since \(G\) is connected, every
component of \(G-v\) contains at least one neighbor of \(v\).  If all components
of \(G-v\) were isolated, then \(q=d\), contradicting \(d\ge q+2\).  Thus
\(G-v\) has at least one non-isolated component, and so $q\ge (d-s)+1$. Hence $s\ge d-q+1$. By Lemma~\ref{lem:gamma-neighbor-count},
\[
\gamma\ge s\ge d-q+1.
\]
Therefore $\gamma/d \ge (d-q+1)/d$. Lemma~\ref{lem:sl2-rank-one} and translation monotonicity give
\[
        \mathsf S_t(G)-\mathsf S_t(G-v)
        \ge
        2\mathcal I_{[\gamma/d,\gamma/d+d]}(t)  
        \ge
        2\mathcal I_{[(d-q+1)/d,(d-q+1)/d+d]}(t).
\]
Applying Lemma~\ref{lem:interval-domination} with \(r=q\), we get
\[
        \mathsf S_t(G)-\mathsf S_t(G-v)
        \ge
        2q\mathcal I_{[3,4]}(t).
\]
On the other hand, Corollary~\ref{cor:repeated-path-costs-sl2}, applied to the \(q+1\) paths \(P_1,P_{n_1},\ldots,P_{n_q}\), gives
\[
        \mathsf S_t(P_n)-\mathsf S_t(P_1)-\sum_{i=1}^q\mathsf S_t(P_{n_i})
        \le
        2q\mathcal I_{[3,4]}(t).
\]
Since \(\mathsf S_t(P_1)=0\), it follows that
\[
        \mathsf S_t(G)-\mathsf S_t(G-v)
        \ge
        2q\mathcal I_{[3,4]}(t)
        \ge
        \mathsf S_t(P_n)-\sum_{i=1}^q\mathsf S_t(P_{n_i}).
\]
Using additivity over the components of \(G-v\) and the component hypotheses,
\[
        \mathsf S_t(G-v)
        =
        \sum_{i=1}^q\mathsf S_t(G_i)
        \ge
        \sum_{i=1}^q\mathsf S_t(P_{n_i}).
\]
Therefore
\[
\begin{aligned}
        \mathsf S_t(G)
        &=
        \mathsf S_t(G-v)
        +
        \bigl(\mathsf S_t(G)-\mathsf S_t(G-v)\bigr)\\
        &\ge
        \sum_{i=1}^q\mathsf S_t(P_{n_i})
        +
        \left(
        \mathsf S_t(P_n)-\sum_{i=1}^q\mathsf S_t(P_{n_i})
        \right)\\
        &=
        \mathsf S_t(P_n).
\end{aligned}
\]
This proves the first assertion.

We now prove the borderline assertion.  Assume that \(d=q+1\), that \(G-v\)
has at least two non-isolated components, and that one component of \(G-v\)
contains at least two neighbors of \(v\).  Since \(G\) is connected, every
non-isolated component of \(G-v\) contains at least one neighbor of \(v\).
Therefore the number \(s\) of neighbors of \(v\) that remain non-isolated in
\(G-v\) satisfies
\[
        s\ge3.
\]
Lemma~\ref{lem:gamma-neighbor-count} gives \(\gamma\ge3\).  Hence, by
Lemma~\ref{lem:sl2-rank-one} and translation monotonicity,
\[
        \mathsf S_t(G)-\mathsf S_t(G-v)
        \ge
        2\mathcal I_{[\gamma/d,\gamma/d+d]}(t) 
        \ge
        2\mathcal I_{[3/d,d+3/d]}(t).
\]
The sharper part of Lemma~\ref{lem:interval-domination} gives
\[
        \mathsf S_t(G)-\mathsf S_t(G-v)
        \ge
        2\mathcal I_{[2,4]}(t)+2(d-3)\mathcal I_{[3,4]}(t).
\]
Since \(q=d-1\), Corollary~\ref{cor:repeated-path-costs-sl2} gives
\[
\begin{aligned}
        \mathsf S_t(P_n)-\sum_{i=1}^q\mathsf S_t(P_{n_i})
        &\le
        2\mathcal I_{[2,4]}(t)+2(q-2)\mathcal I_{[3,4]}(t)\\
        &=
        2\mathcal I_{[2,4]}(t)+2(d-3)\mathcal I_{[3,4]}(t).
\end{aligned}
\]
Thus the local gain again dominates the path deficit. Combining this with $\sum_{i=1}^q\mathsf S_t(G_i) \ge \sum_{i=1}^q\mathsf S_t(P_{n_i})$ proves $\mathsf S_t(G)\ge \mathsf S_t(P_n)$. This proves the borderline assertion and completes the proof.
\end{proof}

We use the following terminology. In this paper, a \emph{sun graph} means a
graph obtained from an even cycle by attaching any finite number of pendant
leaves to its cycle vertices. A cycle vertex is called \emph{loaded} if at
least one pendant leaf is attached to it. Such a sun graph is called
\emph{sparse} if each cycle vertex supports at most one pendant leaf. In a
sparse sun graph, the loaded cycle vertices are called \emph{\(3\)-separated}
if their pairwise cyclic distances along the base cycle are at least \(3\).

The next proposition reduces a minimal counterexample to the terminal sparse-sun configuration.

\begin{proposition}
\label{prop:sl2-reduction-suns}
Let \(G\) be a connected bipartite graph of minimum order among all connected
bipartite graphs for which there exists \(t\ge0\) such that
\[
        \mathsf S_t(G)<\mathsf S_t(P_{|V(G)|}).
\]
Then \(G\) is a sparse sun graph whose loaded cycle vertices are
\(3\)-separated.
\end{proposition}

\begin{proof}
By minimality, every connected bipartite graph of order strictly smaller than
\(|V(G)|\) satisfies the desired comparison for all \(t\ge0\). In particular,
every connected component arising after deleting a vertex of \(G\) satisfies
the comparison. Therefore, whenever one of the local hypotheses in
Lemma~\ref{lem:sl2-high-redundancy} is met, the corresponding deletion criterion may
be applied.

By Corollary~\ref{cor:tree-stoploss-path}, \(G\) is not a tree. Since \(G\) is connected, it therefore contains a cycle. A cycle is \(2\)-connected, and hence it is contained in a maximal \(2\)-connected subgraph, that is, in a \emph{\(2\)-connected block} of \(G\).

Here and below, \(\kappa(H)\) denotes the number of connected components of a
graph \(H\). Clearly, every vertex satisfies
\[
        d_G(v)-\kappa(G-v)\le1.
\]
Indeed, if some vertex \(v\) satisfied \(d_G(v)\ge\kappa(G-v)+2\), then the
first assertion of Lemma~\ref{lem:sl2-high-redundancy}, together with the
minimality of \(G\), would prove the comparison for \(G\) at all \(t\ge0\),
contradicting the choice of \(G\).

Let \(B\) be a \(2\)-connected block of \(G\).  If \(v\in V(B)\), then
\(B-v\) is connected.  Let \(r_v\) be the number of neighbors of \(v\) outside
\(B\), and let \(c_v\) be the number of components of \(G-v\) which do not meet
\(B-v\).  Each such component contains at least one neighbor of \(v\) outside
\(B\), so \(c_v\le r_v\).  Since \(B-v\) is connected,
\[
        \kappa(G-v)\le 1+c_v,
        \qquad
        d_G(v)=d_B(v)+r_v.
\]
Therefore
\[
        d_G(v)-\kappa(G-v)
        \ge
        d_B(v)+r_v-(1+c_v)
        \ge
        d_B(v)-1.
\]
Together with \(d_G(v)-\kappa(G-v)\le1\), this gives \(d_B(v)\le2\) for every \(v\in V(B)\).  Since \(B\) is \(2\)-connected, every vertex of \(B\) has degree at least \(2\) inside \(B\), and therefore \(B\) must be a cycle. Thus every \(2\)-connected block of \(G\) is a cycle, so \(G\) is a cactus. We now fix a cycle block \(C\) of \(G\). Using the structure of cacti, after deleting all edges of \(C\), every connected component of \(G-E(C)\) contains exactly one vertex of \(C\).

We next show that \(G\) is a sun graph.  Fix a cycle block \(C\) of \(G\).

We claim first that, for every \(v\in V(C)\), the graph \(G-v\) contains at most one non-isolated component.  Suppose, to the contrary, that for some \(v\in V(C)\), the graph \(G-v\) contains at least two non-isolated components. The component containing \(C-v\) contains the two cycle neighbors of \(v\), and therefore one component of \(G-v\) contains at least two neighbors of \(v\). Since \(G\) is connected, every other component of \(G-v\) contains at least one neighbor of \(v\).  Hence
\[
        d_G(v)\ge \kappa(G-v)+1.
\]
Combining this with \(d_G(v)-\kappa(G-v)\le1\), we get
\[
        d_G(v)=\kappa(G-v)+1.
\]
Thus the borderline part of Lemma~\ref{lem:sl2-high-redundancy} applies to \(v\), because \(G-v\) has at least two non-isolated components and one of its components contains at least two neighbors of \(v\).  This would make \(v\) deletable, contradicting the minimality of \(G\).  The claim follows.

Now fix \(v\in V(C)\), and let \(D_v\) be the connected component of
\(G-E(C)\) that contains \(v\).  Since \(G\) is a cactus, every component of
\(G-E(C)\) contains exactly one vertex of \(C\); in particular,
\[
        V(D_v)\cap V(C)=\{v\}.
\]
We show that \(D_v\) is a star centered at \(v\).  No vertex of \(D_v-v\) lies
in the same component of \(G-v\) as \(C-v\).  Indeed, if such a path existed in
\(G-v\), then by taking the first vertex of \(C-v\) reached by the path, we
would obtain in \(G-E(C)\) a path from \(D_v\) to a second vertex of \(C\),
contradicting \(V(D_v)\cap V(C)=\{v\}\).

The component of \(G-v\) containing \(C-v\) is non-isolated.  By the claim,
\(G-v\) has at most one non-isolated component.  Therefore every component of
\(D_v-v\) is isolated.  Since \(D_v\) is connected, every vertex of \(D_v-v\)
is adjacent only to \(v\).  Hence \(D_v\) is a star centered at \(v\), with
possibly no leaves.

This holds for every \(v\in V(C)\).  Since every component of \(G-E(C)\) contains exactly one vertex of \(C\), the graph \(G\) is obtained from the cycle \(C\) by attaching pendant leaves to cycle vertices. The cycle \(C\) is even because \(G\) is bipartite. Hence \(G\) is a sun graph.

We now strengthen this to sparsity.  Suppose that some cycle vertex \(v\)
supports at least two pendant leaves.  Delete one such leaf \(\ell\), and put
\(H=G-\ell\).  Then \(H\) is connected.  When \(\ell\) is added back to \(H\),
the added vertex has degree \(1\).  Choose the bipartition so that \(\ell\)
lies on the left side.  In the notation of Lemma~\ref{lem:sl2-rank-one}, we
have \(b=\one_{\{v\}}\), and Lemma~\ref{lem:gamma-neighbor-count} gives
\[
        \gamma
        \ge
        \sum_{w\in N_G(\ell)} (d_G(w)-1)
        =
        d_G(v)-1
        \ge
        3,
\]
because \(v\) is incident with its two cycle edges and with at least two pendant leaves.  Hence Lemma~\ref{lem:sl2-rank-one} and translation monotonicity give
\[
        \mathsf S_t(G)-\mathsf S_t(H)
        \ge
        2\mathcal I_{[\gamma,\gamma+1]}(t)
        \ge
        2\mathcal I_{[3,4]}(t).
\]
By minimality, \(H\) satisfies the comparison with \(P_{n-1}\).  On the other
hand, Lemma~\ref{lem:path-costs-sl2} gives
\[
        \mathsf S_t(P_n)-\mathsf S_t(P_{n-1})
        \le
        2\mathcal I_{[3,4]}(t).
\]
Thus adding \(\ell\) back gives at least as much gain as the endpoint cost of
passing from \(P_{n-1}\) to \(P_n\).  Hence \(G\) would satisfy the desired
comparison, contradicting the minimality of \(G\).  Therefore \(G\) is sparse.

It remains to prove that the loaded cycle vertices are \(3\)-separated.  Since
\(G\) is now known to be a sparse sun graph, it suffices to exclude the cyclic
patterns \(11\) and \(101\), where \(1\) denotes a loaded cycle vertex and
\(0\) denotes an unloaded cycle vertex.

First suppose that two adjacent cycle vertices \(x\) and \(y\) are loaded.
Since \(G\) is sparse, \(x\) supports exactly one pendant leaf; call it
\(\ell_x\).  Delete \(x\).  Then
\[
        G-x=K\sqcup\{\ell_x\},
\]
where \(K\) is connected.  Choose the bipartition so that \(x\) lies on the
left side, and let \(z\) be the other cycle neighbor of \(x\).  By
Lemma~\ref{lem:gamma-neighbor-count},
\[
        \gamma
        \ge
        \sum_{w\in N_G(x)} (d_G(w)-1).
\]
The pendant leaf \(\ell_x\) contributes \(d_G(\ell_x)-1=0\).  Since \(y\) is
loaded and \(G\) is sparse, \(d_G(y)=3\), while \(d_G(z)\ge2\).  Hence
\[
        \gamma
        \ge
        (d_G(y)-1)+(d_G(z)-1)
        \ge
        2+1=3.
\]
Also \(d_G(x)=3\).  Therefore Lemma~\ref{lem:sl2-rank-one} and translation
monotonicity give
\[
        \mathsf S_t(G)-\mathsf S_t(G-x)
        \ge
        2\mathcal I_{[\gamma/3,\gamma/3+3]}(t)
        \ge
        2\mathcal I_{[1,4]}(t).
\]
By minimality, \(K\) satisfies the path comparison, and
\(\mathsf S_t(P_1)=0\).  Moreover, Lemma~\ref{lem:path-costs-sl2} gives
\[
        \mathsf S_t(P_n)-\mathsf S_t(P_{|V(K)|})-\mathsf S_t(P_1)
        \le
        2\mathcal I_{[2,4]}(t)
        \le
        2\mathcal I_{[1,4]}(t).
\]
Thus the local gain from adding \(x\) back dominates the corresponding path
deficit.  This would make \(x\) deletable, contradicting the minimality of
\(G\).  Hence the pattern \(11\) cannot occur.

Next suppose that the pattern \(101\) occurs.  Thus there are three consecutive
cycle vertices \(x,u,z\), where \(x\) and \(z\) are loaded and \(u\) is
unloaded.  Delete \(u\), and put \(H=G-u\).  Then \(H\) is connected.  Choose
the bipartition so that \(u\) lies on the left side.  By
Lemma~\ref{lem:gamma-neighbor-count},
\[
        \gamma
        \ge
        \sum_{w\in N_G(u)} (d_G(w)-1)
        =
        (d_G(x)-1)+(d_G(z)-1)
        =
        2+2=4,
\]
because \(x\) and \(z\) are loaded in a sparse sun graph.  Since \(u\) is
unloaded, \(d_G(u)=2\).  Hence Lemma~\ref{lem:sl2-rank-one} and translation
monotonicity give
\[
        \mathsf S_t(G)-\mathsf S_t(H)
        \ge
        2\mathcal I_{[\gamma/2,\gamma/2+2]}(t)
        \ge
        2\mathcal I_{[2,4]}(t).
\]
By minimality, \(H\) satisfies the path comparison with \(P_{n-1}\).  Meanwhile
Lemma~\ref{lem:path-costs-sl2} gives
\[
        \mathsf S_t(P_n)-\mathsf S_t(P_{n-1})
        \le
        2\mathcal I_{[3,4]}(t)
        \le
        2\mathcal I_{[2,4]}(t).
\]
Thus \(u\) would be deletable, contradicting the minimality of \(G\).  Hence
the pattern \(101\) cannot occur.

Therefore no two loaded cycle vertices are at cyclic distance \(1\) or \(2\).
Equivalently, the loaded cycle vertices are \(3\)-separated.
\end{proof}

\section{The terminal \texorpdfstring{$3$}{3}-separated sparse-sun case}
\label{sec:support-neighbor-deletion-clean}

This section treats the terminal case left by Proposition~\ref{prop:sl2-reduction-suns}. Thus \(G\) is a \(3\)-separated sparse sun graph with at least one pendant leaf.  We shall delete an unloaded cycle neighbor of a loaded cycle vertex, rather than deleting the pendant leaf itself. The goal is to show that the rank-one spectral shift produced by adding this vertex back dominates the endpoint increment \(P_{n-1}\rightsquigarrow P_n\).

The proof is reduced to an envelope comparison on \(0\le t\le4\); the finite one-variable inequalities used in that comparison are verified in Appendix~\ref{app:verification}. For \(t\ge4\), the path-endpoint contribution vanishes by spectral support.

\subsection{Spectral-shift notation}

Recall the spectral-shift interval notation from the proof of Lemma~\ref{lem:path-costs-sl2}.  Let \(M\succeq0\) and \(b\ne0\), and consider the
positive rank-one update
\[
        M\rightsquigarrow M+bb^\top .
\]
Let
\[
        \alpha_1\ge\alpha_2\ge\cdots\ge0,
        \qquad
        \beta_1\ge\beta_2\ge\cdots\ge0
\]
be the eigenvalues of \(M\) and \(M+bb^\top\), respectively, with zero padding
if necessary. By the rank-one interlacing theorem, we know that
\[
        \beta_1\ge \alpha_1\ge \beta_2\ge \alpha_2\ge\cdots\ge0 .
\]
We define the associated spectral-shift interval set by
\[
        E(M,b):=\bigcup_i[\alpha_i,\beta_i],
\]
where zero-length intervals and endpoint overlaps are harmless.  Thus
\(E(M,b)\) records the interlacing gaps created by the positive rank-one
update.

For a finite interval set \(E\subset[0,\infty)\), recall that
\[
        J_E(t)=\int_E (y-t)_+\,dy,
        \qquad t\ge0.
\]
When \(E=E(M,b)\), this quantity is half of the increase of the second-order stop-loss functional:
\[
        \tr(M+bb^\top-tI)_+^2-\tr(M-tI)_+^2
        =
        2J_{E(M,b)}(t).
\]
If \(M_\theta=M+\theta bb^\top\), then combining this identity with Lemma~\ref{lem:rank-one-trace-formula} gives
\begin{equation}\label{eq:J-integral-clean}
        J_{E(M,b)}(t)
        =
        \int_0^1 b^\top(M_\theta-t)_+b\,d\theta .
\end{equation}

For the path endpoint increment \(P_{n-1}\rightsquigarrow P_n\), label the
vertices as \(1,\ldots,n\), so that \(P_{n-1}\) is the path
\(1-\cdots-(n-1)\) and the new edge is \((n-1)n\).  Choose the bipartition so
that the new endpoint \(n\) lies on the left side; then the old endpoint
\(n-1\) lies on the right side.  Let \(B_{P_{n-1}}\) be the biadjacency matrix
of \(P_{n-1}\) with this bipartition, and put
\[
        M_0:=B_{P_{n-1}}^\top B_{P_{n-1}}.
\]
Let \(e\) be the standard basis vector corresponding to the right-side vertex
\(n-1\).  Adding the new endpoint appends the row \(e^\top\) to the
biadjacency matrix, and therefore
\[
        B_{P_n}^\top B_{P_n}=M_0+ee^\top.
\]
We write 
\[
        E_n^P:=E(M_0,e)
\]
for the corresponding spectral-shift interval set.

\subsection{The support-neighbor configuration}

Let \(G\) be a \(3\)-separated sparse sun graph with at least one pendant leaf.
Choose a loaded cycle vertex \(a\).  Since the loaded cycle vertices are
\(3\)-separated, both cycle neighbors of \(a\) are unloaded.  Choose one of
them and call it \(u\).  Let \(c\) be the other cycle neighbor of \(u\), and put
\[
        H=G-u,
        \qquad n=|V(G)|.
\]
Deleting \(u\) breaks the base cycle, while all pendant leaves remain attached
to the resulting path.  Hence \(H\) is a tree.

Choose the bipartition so that \(u\) lies on the left side; then \(a\) and
\(c\) lie on the right side.  Let \(B_H\) be the biadjacency matrix of \(H\),
with columns indexed by the right side.  Let \(e_a\) and \(e_c\) denote the
standard basis vectors corresponding to the right-side vertices \(a\) and
\(c\), and set
\[
        M=B_H^\top B_H,
        \qquad
        b=e_a+e_c,
        \qquad
        M_\theta=M+\theta bb^\top \quad(0\le\theta\le1).
\]
Adding \(u\) back to \(H\) appends to the biadjacency matrix one new row corresponding to \(u\). This row has nonzero entries precisely in the columns \(a\) and \(c\). Therefore
\[
        B_G^\top B_G=M+bb^\top .
\]
Thus the rank-one update \(M\rightsquigarrow M+bb^\top\) is exactly the
squared-singular-value update from \(H\) to \(G\).  Define
\[
        E_{G,u}:=E(M,b).
\]

The next result is the final deletion step for the terminal sparse-sun case:
it shows that adding back the unloaded neighbor \(u\) gives at least as much
second-order stop-loss gain as adding an endpoint to the path.

\begin{theorem}\label{thm:SN}
With the notation above,
\[
        J_{E_{G,u}}(t)\ge J_{E_n^P}(t)
        \qquad\text{for every }t\ge0.
\]
\end{theorem}

\subsection{Local data for the support-neighbor update}

The next lemma records the local first moment, a fifth-moment lower bound, and
the spectral support needed for the envelope comparison.

\begin{lemma}\label{lem:local-data}
Assume that \(G\) is a \(3\)-separated sparse sun graph whose base cycle has length at least \(6\). Then, for every
\(0\le\theta\le1\),
\[
        b^\top M_\theta b=3+4\theta,
\]
\[
        b^\top M_\theta^5 b\ge q_5(\theta):=
        64\theta^5+240\theta^4+440\theta^3+507\theta^2+387\theta+174.
\]
Moreover, 
\[
\spec(M_\theta)\subset[0,5].
\]
\end{lemma}

\begin{proof}
We first compute \(b^\top M_\theta b\).  Since \(b=e_a+e_c\), we have \(\|b\|^2=2\), and hence
\[
        b^\top M_\theta b=b^\top Mb+\theta\|b\|^4=b^\top Mb+4\theta .
\]
It remains to compute \(b^\top Mb=\|B_Hb\|_2^2\).  In \(H\), the vertex \(a\) has two remaining left neighbors, namely its remaining cycle neighbor and its pendant leaf.  The vertex \(c\) is unloaded by \(3\)-separation, and after deleting \(u\) it has only one remaining left neighbor.  Since the cycle has length at least \(6\), \(a\) and \(c\) have no common left neighbor in \(H\). Thus
\[
        b^\top Mb=\|B_Hb\|_2^2=2+1=3,
\]
and the claimed identity follows.

Let \(H_0\) be the subgraph of \(H\) induced by the unique \(a\)-\(c\) path in \(H\), together with the pendant leaf attached to \(a\), and with all other pendant leaves removed.
Let \(B_{H_0}\) be its biadjacency matrix, with rows indexed by the left side
and columns indexed by the right side, in the same convention as for \(B_H\).
Let \(N\) be the number of right-side vertices of \(H_0\), ordered so that
\(a\) is the first vertex and \(c\) is the last vertex.  The Gram matrix on
the right side is therefore
\[
        T_N:=B_{H_0}^\top B_{H_0}.
\]
If \(r_1=a, r_2, \ldots,r_N=c\) are the right-side vertices in this order, then \((T_N)_{ij}=|N_{H_0}(r_i)\cap N_{H_0}(r_j)|\). Hence the diagonal entries record the right-side degrees, so they are \(2,2,\ldots,2,1\): the vertex \(a\) has one path neighbor and its pendant leaf, each internal right-side path vertex has two path neighbors, and \(c\) has only one remaining path neighbor after \(u\) is deleted. Moreover, two consecutive right-side vertices along \(H_0\) have exactly one common left neighbor, while non-consecutive right-side vertices have no common left neighbor. Thus
\[
        T_N=
        \begin{pmatrix}
        2&1&0&\cdots&0\\
        1&2&1&\ddots&\vdots\\
        0&1&2&\ddots&0\\
        \vdots&\ddots&\ddots&\ddots&1\\
        0&\cdots&0&1&1
        \end{pmatrix}.
\]
Since \(a=r_1\) and \(c=r_N\), the vector corresponding to \(e_a+e_c\) in this reduced model is $b_N=e_1+e_N$. Put the rank-one interpolation as
\[
        M_{N,\theta}:=T_N+\theta b_Nb_N^\top .
\]

We first show that the reduced local model has its smallest fifth moment at
the stable value \(N=7\):
\begin{equation}\label{eq:smallest-fifth-moment-n=7}
        b_N^\top M_{N,\theta}^5 b_N
        \ge
        b_7^\top M_{7,\theta}^5 b_7
        \qquad(N\ge3,\;0\le\theta\le1).
\end{equation}
Put
\[
        P_N:=b_Nb_N^\top,
        \qquad
        s_N(r):=b_N^\top T_N^r b_N
        \qquad(r\ge0).
\]
We first prove that
\begin{equation}\label{eq:sN-s7}
        s_N(r)\ge s_7(r)
        \qquad(N\ge3,\;0\le r\le5).
\end{equation}

View \(T_N\) as the adjacency matrix of a looped colored path on
\(\{1,\ldots,N\}\): there is one edge between \(i\) and \(i+1\), each of the
vertices \(1,\ldots,N-1\) carries two loop-colors, and the endpoint \(N\)
carries one loop-color.  With this convention, \((T_N^r)_{ij}\) counts
colored walks of length \(r\) from \(i\) to \(j\).  Since
\(b_N=e_1+e_N\), the quantity \(s_N(r)\) counts colored walks of length \(r\)
whose initial and terminal vertices lie in the endpoint set \(\{1,N\}\).

We construct an injection from the colored walks counted by \(s_7(r)\) to
those counted by \(s_N(r)\), for every \(N\ge3\) and \(0\le r\le5\).  In the
\(N=7\) model, the two endpoints are at distance \(6\).  Hence a walk of
length at most \(5\) cannot visit both endpoints.  Every endpoint walk counted
by \(s_7(r)\) is therefore either a return walk at \(1\) or a return walk at
\(7\).

Assume first that \(N\ge4\). A return walk of length \(r\le5\) at an endpoint can reach distance at most \(2\) from that endpoint. Thus a return walk at \(1\) in the \(N=7\) model only uses the vertices \(1,2,3\). These vertices have the same loop-colors and edge structure in the \(N\)-model, so
we map such a walk identically to a return walk at \(1\).  Similarly, a return
walk at \(7\) only uses the vertices \(7,6,5\), and we translate these vertices
to \(N,N-1,N-2\).  This preserves all edge steps and loop-colors.  The image
of a left-end return walk starts at \(1\), while the image of a right-end
return walk starts at \(N\); since \(N\ge4\), these image families are
disjoint.  Hence \(s_N(r)\ge s_7(r)\) for \(N\ge4\).

It remains to handle \(N=3\).  The right-end return walks in the \(N=7\)
model translate exactly to right-end return walks in the \(N=3\) model by
sending \(7,6,5\) to \(3,2,1\).  For the left-end return walks, the identity
map works except for one missing colored walk.  This can occur only when
\(r=5\): the walk reaches vertex \(3\), uses the loop-color at vertex \(3\)
which is absent in the \(N=3\) model, and then returns:
\[
        1\to2\to3\to3\to2\to1 .
\]
We map this single missing walk instead to the endpoint-crossing walk
\[
        1\to2\to3\to3\to3\to3 .
\]
This image is not used by any left-end return walk, since it ends at \(3\), and
it is not used by any right-end return walk, since it starts at \(1\).  Thus
the injection also exists for \(N=3\), proving \eqref{eq:sN-s7}.

We now expand the noncommutative product \((T_N+\theta P_N)^5\).  A word with
exactly \(m\) copies of \(P_N\) and \(5-m\) copies of \(T_N\) can be written
uniquely in the form
\[
        T_N^{r_0}P_NT_N^{r_1}P_N\cdots P_NT_N^{r_m},
        \qquad
        r_0+\cdots+r_m=5-m,
\]
where the exponents \(r_i\) are allowed to be zero.  For such a word, using
\(P_N=b_Nb_N^\top\), its contribution to the quadratic form with \(b_N\) is
\[
b_N^\top T_N^{r_0}P_NT_N^{r_1}P_N\cdots P_NT_N^{r_m}b_N = \prod_{\ell=0}^m b_N^\top T_N^{r_\ell}b_N = \prod_{\ell=0}^m s_N(r_\ell).
\]
Therefore
\[
        b_N^\top M_{N,\theta}^5b_N
        =
        \sum_{m=0}^5
        \theta^m
        \sum_{\substack{r_0,\ldots,r_m\ge0\\
                        r_0+\cdots+r_m=5-m}}
        \prod_{\ell=0}^m s_N(r_\ell).
\]
By~\eqref{eq:sN-s7}, each coefficient of this polynomial in \(\theta\) is at
least the corresponding coefficient for \(N=7\).  Since \(0\le\theta\le1\),
this coefficientwise comparison gives~\eqref{eq:smallest-fifth-moment-n=7}.

It remains to compute the stable \(N=7\) model.  A direct exact multiplication in this \(7\times7\) model gives
\[
        b_7^\top M_{7,\theta}^5b_7
        =
        64\theta^5+240\theta^4+440\theta^3
        +507\theta^2+387\theta+174
        =
        q_5(\theta).
\]
The exact rational verification of this \(7\times7\) computation is recorded in Section~\ref{app:local-moment-check}. Hence, by~\eqref{eq:smallest-fifth-moment-n=7}, the reduced local model satisfies
\[
        b_N^\top M_{N,\theta}^5b_N\ge q_5(\theta).
\]

Restoring the pendant leaves that were removed in the definition of \(H_0\) can only increase the fifth moment seen from \(b\).  Let \(R_0\) be the set of right-side vertices of \(H_0\), ordered as above, and let \(R_1\) be the set of remaining right-side vertices of \(H\).  Order the columns of \(B_H\) as \(R_0\) followed by \(R_1\), and write $B_H=(C_0\ C_1)$. Then
\[
        B_H^\top B_H=
        \begin{pmatrix}
        C_0^\top C_0 & C_0^\top C_1\\
        C_1^\top C_0 & C_1^\top C_1
        \end{pmatrix}.
\]
The rows of \(B_H\) coming from the subgraph \(H_0\) contribute exactly \(T_N\) to the upper-left block. All remaining rows contribute only entrywise non-negative terms. Hence $C_0^\top C_0=T_N+\Delta$ for some entrywise non-negative matrix \(\Delta\). 

Moreover, \(b=e_a+e_c\) is supported on \(R_0\), and in these block coordinates $b=\binom{b_N}{0}$. Since \(M_{N,\theta}=T_N+\theta b_Nb_N^\top\), we obtain
\[
\begin{aligned}
        M_\theta
        &=
        B_H^\top B_H+\theta bb^\top  \\
        &=
        \begin{pmatrix}
        T_N+\Delta & C_0^\top C_1\\
        C_1^\top C_0 & C_1^\top C_1
        \end{pmatrix}
        +
        \begin{pmatrix}
        \theta b_Nb_N^\top & 0\\
        0&0
        \end{pmatrix}  \\
        &=
        \begin{pmatrix}
        M_{N,\theta}+\Delta & X\\
        X^\top & Y
        \end{pmatrix},
\end{aligned}
\]
where \(X=C_0^\top C_1\), \(Y=C_1^\top C_1\), and \(X,Y\) are entrywise non-negative.

Now define
\[
        \widetilde M_{N,\theta}:=
        \begin{pmatrix}
        M_{N,\theta}&0\\
        0&0
        \end{pmatrix}.
\]
The displayed block decomposition shows that $M_\theta\ge \widetilde M_{N,\theta}$ entrywise. Since both matrices are entrywise non-negative, entrywise domination is preserved under multiplication, and hence under taking positive integer powers. Therefore $M_\theta^5\ge \widetilde M_{N,\theta}^{\,5}$ entrywise and
\[
        b^\top M_\theta^5b
        \ge
        b^\top \widetilde M_{N,\theta}^{\,5}b
        =
        b_N^\top M_{N,\theta}^5b_N
        \ge
        q_5(\theta).
\]

It remains to prove the spectral support bound.  We shall show that every row
sum of \(M_\theta\) is at most \(5\).  Since \(M_\theta\) is entrywise
non-negative, the row-sum bound for nonnegative matrices
\cite[Theorem~8.1.22]{HornJohnson2012} gives
\[
        \rho(M_\theta)
        \le
        \max_y\sum_z (M_\theta)_{yz}.
\]
Thus the desired spectral upper bound follows once all row sums are shown to be
at most \(5\).

We first estimate the row sums of \(M=B_H^\top B_H\).  For a right-side vertex
\(y\), with the sum over \(z\) running over right-side vertices, we have
\[
\begin{aligned}
        \sum_z M_{yz}
        &=
        \sum_z\sum_{\ell\in L}
        (B_H)_{\ell y}(B_H)_{\ell z}  \\
        &=
        \sum_{\ell\in N_H(y)\cap L}\sum_z (B_H)_{\ell z}
        =
        \sum_{\ell\in N_H(y)\cap L}d_H(\ell),
\end{aligned}
\]
where \(L\) is the left side of the bipartition.  Thus the row sum at \(y\) is
the sum of the degrees, in \(H\), of the left neighbors of \(y\).

We now check these sums case by case.  If \(y\ne a\) is a loaded cycle vertex on the right side, then its two cycle neighbors are unloaded by \(3\)-separation, and its pendant leaf has degree \(1\).  Hence the row sum is at most
\[
        2+2+1=5.
\]
For \(y=a\), the neighbor \(u\) has been deleted.  Thus the only remaining left neighbors of \(a\) are its other cycle neighbor, of degree at most \(2\), and its pendant leaf, of degree \(1\).  Hence the row sum is at most
\[
        2+1=3.
\]
If \(y\ne c\) is an unloaded cycle vertex on the right side, then at most one of its two left cycle neighbors is loaded. Therefore the row sum is at most
\[
        3+2=5.
\]
For \(y=c\), the deleted vertex \(u\) removes one of the two cycle-neighbor
contributions, so \(c\) has only one left neighbor in \(H\), of degree at most
\(3\).  Hence the row sum is at most \(3\).  Finally, if \(y\) is a right-side
pendant leaf, then \(y\) has a unique left neighbor, and that neighbor has
degree at most \(3\).  Hence the row sum is again at most \(3\).

The rank-one term \(\theta bb^\top\), with \(b=e_a+e_c\), affects only the
rows \(a\) and \(c\).  In each of those two rows it adds \(2\theta\le2\) to the
row sum.  Since the row sums of \(M\) at \(a\) and \(c\) are at most \(3\), and
all other row sums of \(M\) are at most \(5\), every row sum of \(M_\theta\) is
at most \(5\).  Hence \(\rho(M_\theta)\le5\).  Since \(M_\theta\succeq0\), all
eigenvalues of \(M_\theta\) are non-negative, and therefore
\[
        \spec(M_\theta)\subset[0,5].
\qedhere\]
\end{proof}

\subsection{Moment data for the path endpoint increment}

We shall need the following moment bound for the spectral-shift interval set associated with adding one endpoint to a path.

\begin{lemma}\label{lem:path-moments}
For every \(n\ge2\), let \(E_n^P\) be the path endpoint spectral-shift interval
set for \(P_{n-1}\rightsquigarrow P_n\). Then, for every integer \(m\ge1\),
\[
        \int_{E_n^P} y^{m-1}\,dy
        \le
        \frac1{2m}\binom{2m}{m}.
\]
\end{lemma}

\begin{proof}
Use the notation from Lemma~\ref{lem:path-costs-sl2}.  Let
\[
        \beta_1\ge \alpha_1\ge \beta_2\ge \alpha_2\ge\cdots\ge0
\]
be the interlacing order for the squared singular values in the endpoint update \(P_{n-1}\rightsquigarrow P_n\), after zero padding if necessary. Thus \(\{\beta_i\}_i\) and \(\{\alpha_i\}_i\) are the multisets \(\{\mu_i(P_n)\}_i\) and \(\{\mu_i(P_{n-1})\}_i\), respectively, and $E_n^P=\bigcup_i[\alpha_i,\beta_i]$.

Let \(A_k\) be the adjacency matrix of \(P_k\).  Since \(P_k\) is bipartite, its nonzero adjacency eigenvalues are the pairs \(\pm\sqrt{\mu_i(P_k)}\). Hence $\operatorname{tr} A_k^{2m}=2\sum_i\mu_i(P_k)^m$. Therefore
\[
        \int_{E_n^P}y^{m-1}\,dy
        =
        \sum_i\int_{\alpha_i}^{\beta_i}y^{m-1}\,dy  
        =
        \frac1m\sum_i(\beta_i^m-\alpha_i^m) 
        =
        \frac{\operatorname{tr}A_n^{2m}
        -\operatorname{tr}A_{n-1}^{2m}}{2m}.
\]
It remains to prove
\[
        \operatorname{tr}A_n^{2m}
        -
        \operatorname{tr}A_{n-1}^{2m}
        \le
        \binom{2m}{m}.
\]

The trace \(\operatorname{tr}A_k^{2m}\) counts closed walks of length \(2m\) in \(P_k\). Therefore the difference $\operatorname{tr}A_n^{2m}-\operatorname{tr}A_{n-1}^{2m}$ counts precisely those closed walks of length \(2m\) in \(P_n\) that visit the new endpoint \(n\).

We inject these walks into the set of balanced \(\{\pm1\}\)-sequences of length \(2m\), where balanced means that the sequence contains exactly \(m\) entries equal to \(+1\) and exactly \(m\) entries equal to \(-1\). Let
\[
        W=(v_0,v_1,\ldots,v_{2m}=v_0)
\]
be such a closed walk in \(P_n\), and suppose that \(W\) visits \(n\). Define
\[
        \varepsilon_i:=v_i-v_{i-1}\in\{-1,+1\},
        \qquad i=1,\ldots,2m.
\]
Since \(W\) is closed, \(\sum_{i=1}^{2m}\varepsilon_i=v_{2m}-v_0=0\). Thus the step sequence is balanced.

We claim that the map \(W\mapsto(\varepsilon_1,\ldots,\varepsilon_{2m})\) is injective on the walks counted above. Put
\[
        S_0=0,\qquad
        S_j=\sum_{i=1}^j\varepsilon_i .
\]
Then \(v_j=v_0+S_j\). Since the walk stays in \(P_n\), we have \(S_j\le n-v_0\) for every \(j\). Since the walk visits \(n\), equality holds for at least one \(j\). Hence
\[
        n-v_0=\max_{0\le j\le 2m}S_j.
\]
Thus the initial vertex is determined uniquely by the step sequence:
\[
        v_0=n-\max_{0\le j\le2m}S_j.
\]
Once \(v_0\) and the step sequence are known, the whole walk is determined. This proves injectivity.

There are exactly \(\binom{2m}{m}\) balanced \(\{\pm1\}\)-sequences of length
\(2m\).  Consequently,
\[
        \operatorname{tr}A_n^{2m}
        -
        \operatorname{tr}A_{n-1}^{2m}
        \le
        \binom{2m}{m},
\]
and the displayed moment bound follows.
\end{proof}

\subsection{Envelope comparison on \texorpdfstring{\(0\le t\le4\)}{0 <= t <= 4}}

We first obtain a low-threshold lower bound from the rank-one Jensen estimate.

\begin{lemma}\label{lem:SN-jensen-envelope}
Assume that \(G\) is a \(3\)-separated sparse sun graph whose base cycle has length at least \(6\). For \(0\le t\le3\),
\[
        J_{E_{G,u}}(t)\ge L_{\rm av}(t),
\]
where
\[
        L_{\rm av}(t):=
        2\int_0^1\left(\frac32+2\theta-t\right)_+\,d\theta
        =
        \begin{cases}
        5-2t, &0\le t\le3/2,\\[2mm]
        \dfrac12\left(\dfrac72-t\right)^2,
        &3/2\le t\le3.
        \end{cases}
\]
\end{lemma}

\begin{proof}
Apply Lemma~\ref{lem:sl2-rank-one} to the update \(H\rightsquigarrow G\), where the
added vertex is \(u\).  In the notation of that lemma,
\[
        d=\|b\|^2=d_G(u)=2.
\]
Moreover, Lemma~\ref{lem:local-data} gives
\[
        \gamma=b^\top Mb=3.
\]
Therefore
\[
        \mathsf S_t(G)-\mathsf S_t(H)
        \ge
        2\int_{\gamma/d}^{\gamma/d+d}(y-t)_+\,dy
        =
        2\int_{3/2}^{7/2}(y-t)_+\,dy.
\]
On the other hand, by the definition of the spectral-shift interval set, if \(\alpha_i\) and \(\beta_i\) are the eigenvalues of \(M\) and \(M+bb^\top\), respectively, paired by rank-one interlacing, then
\[
        E_{G,u}=E(M,b)=\bigcup_i[\alpha_i,\beta_i].
\]
Therefore
\[
\begin{aligned}
        \mathsf S_t(G)-\mathsf S_t(H)
        &=
        \sum_i\left((\beta_i-t)_+^2-(\alpha_i-t)_+^2\right) \\
        &=
        2\sum_i\int_{\alpha_i}^{\beta_i}(y-t)_+\,dy 
        =
        2J_{E_{G,u}}(t).
\end{aligned}
\]
Hence
\[
        J_{E_{G,u}}(t)
        \ge
        \int_{3/2}^{7/2}(y-t)_+\,dy
        =
        2\int_0^1\left(\frac32+2\theta-t\right)_+\,d\theta
        =
        L_{\rm av}(t).
\]
The displayed piecewise formula follows by evaluating this elementary integral.
\end{proof}

We next derive the high-threshold lower envelope from the fifth-moment information in Lemma~\ref{lem:local-data}.

\begin{lemma}\label{lem:SN-fifth-envelope}
Assume that \(G\) is a \(3\)-separated sparse sun graph whose base cycle has length at least \(6\). For \(3\le t\le4\),
\[
        J_{E_{G,u}}(t)\ge L_5(t):=
        \frac1{K(t)}
        \left(
        \frac{A(t)^2}{2B(t)}
        +1000\frac{A(t)^3}{B(t)^3}
        \right),
\]
where
\[
        A(t):=1812-7t^4,
        \qquad
        B(t):=4001-4t^4,
        \qquad
        K(t):=5(t^3+5t^2+25t+125).
\]
\end{lemma}

\begin{proof}
Fix \(3\le t\le4\).  We first prove the scalar inequality
\begin{equation}\label{eq:scalar-fifth-clean}
        (y-t)_+
        \ge
        \frac{y^5-t^4y}{K(t)}
        \qquad(0\le y\le5).
\end{equation}
If \(0\le y\le t\), then \(y^5-t^4y\le0\), so the right-hand side is non-positive.  If \(t<y\le5\), then
\[
        \frac{y^5-t^4y}{y-t}
        =
        y(y+t)(y^2+t^2).
\]
The last expression is increasing in \(y\ge0\), and its maximum on \([t,5]\) is $5(5+t)(25+t^2)=K(t)$. This proves~\eqref{eq:scalar-fifth-clean}.

By Lemma~\ref{lem:local-data}, we know that all eigenvalues of
\(M_\theta\) lie in \([0,5]\).  Hence, by functional calculus,
\[
        (M_\theta-t)_+
        \succeq
        \frac{M_\theta^5-t^4M_\theta}{K(t)}.
\]
Pairing with \(b\) gives
\[
        b^\top(M_\theta-t)_+b
        \ge
        \frac{b^\top M_\theta^5b-t^4b^\top M_\theta b}{K(t)}.
\]
The left-hand side is non-negative, \(K(t)>0\), and \(x\mapsto x_+\) is increasing. Using Lemma~\ref{lem:local-data}, we obtain
\[
        b^\top(M_\theta-t)_+b
        \ge
        \frac{[q_5(\theta)-t^4(3+4\theta)]_+}{K(t)}.
\]
Put
\[
        P_t(\theta):=q_5(\theta)-t^4(3+4\theta).
\]
By~\eqref{eq:J-integral-clean},
\[
        J_{E_{G,u}}(t)
        =
        \int_0^1 b^\top(M_\theta-t)_+b\,d\theta
        \ge
        \frac1{K(t)}\int_0^1 [P_t(\theta)]_+\,d\theta.
\]

It remains to lower-bound the last integral.  Set \(z=1-\theta\).  A direct
expansion gives
\begin{equation}\label{eq:Pt-expansion-clean}
        P_t(1-z)
        =
        A-Bz+3907z^2-2040z^3+560z^4-64z^5,
\end{equation}
where \(A=A(t)\) and \(B=B(t)\). On \(3\le t\le4\),
\[
        A(t)\ge A(4)=20>0,
        \qquad
        B(t)\ge B(4)=2977>0.
\]
Let $h(t):=A(t)/B(t)$. Then \(0<h(t)<1\), and
\[
        h'(t)=\frac{t^3(16A(t)-28B(t))}{B(t)^2}
        =
        -\frac{83036t^3}{B(t)^2}<0.
\]
Thus
\[
        h(t)\le h(3)=\frac{1245}{3677}.
\]
The expansion~\eqref{eq:Pt-expansion-clean} and the derivative computation above are verified exactly in Section~\ref{app:local-moment-check}.

For \(0\le z\le h(t)\), we have
\[
        -2040z^3\ge -2040h(t)z^2,
        \qquad
        -64z^5\ge -64h(t)^3z^2,
        \qquad
        560z^4\ge0.
\]
Therefore
\[
        P_t(1-z)
        \ge
        A-Bz+D(t)z^2,
\]
where
\[
        D(t):=3907-2040h(t)-64h(t)^3.
\]
Since \(h(t)\le1245/3677\), the exact rational check in Section~\ref{app:local-moment-check} gives $D(t)>3000$. Moreover, \(A-Bz\ge0\) for \(0\le z\le h(t)=A/B\).  Hence
\[
\begin{aligned}
        \int_0^1[P_t(\theta)]_+\,d\theta
        &=
        \int_0^1[P_t(1-z)]_+\,dz  \\
        &\ge
        \int_0^{A/B}(A-Bz+3000z^2)\,dz \\
        &=
        \frac{A^2}{2B}+1000\frac{A^3}{B^3}.
\end{aligned}
\]
Combining this with the previous lower bound for \(J_{E_{G,u}}(t)\) gives
\(J_{E_{G,u}}(t)\ge L_5(t)\).
\end{proof}

We shall use the following elementary moment envelope for the path endpoint interval set.

\begin{lemma}\label{lem:path-moment-envelope}
Let \(E\subset[0,4]\) be a finite union of intervals. For every \(r>1\) and every \(0<t<4\),
\[
        J_E(t)
        \le
        \Gamma_r(t)\int_{E}y^r\,dy,
\]
where
\[
        \Gamma_r(t):=
        \max_{t< y\le4}\frac{y-t}{y^r}
        =
        \begin{cases}
        \dfrac{(r-1)^{r-1}}{r^r t^{r-1}},
        &0<t\le \dfrac{4(r-1)}r,\\[3mm]
        \dfrac{4-t}{4^r},
        &\dfrac{4(r-1)}r\le t<4.
        \end{cases}
\]
For \(r=1\), one has the simpler bound
\[
        J_E(t)\le \int_E y\,dy
        \qquad(t\ge0).
\]
\end{lemma}

\begin{proof}
By the definition of \(\Gamma_r(t)\), we have
\[
        (y-t)_+\le \Gamma_r(t)y^r
        \qquad(0\le y\le4).
\]
Integrating over \(E\) gives
\[
        J_{E}(t)
        =
        \int_{E}(y-t)_+\,dy
        \le
        \Gamma_r(t)\int_{E}y^r\,dy.
\]

It remains only to compute \(\Gamma_r(t)\).  For \(t<y\le4\), set $\phi(y)=(y-t)/y^r$. Then
\[
        \phi'(y)=\frac{rt-(r-1)y}{y^{r+1}}.
\]
Thus the unconstrained maximizer is $y_*=rt/(r-1)$. If \(y_*\le4\), equivalently \(t\le4(r-1)/r\), then
\[
        \Gamma_r(t)=\phi(y_*)
        =
        \frac{(r-1)^{r-1}}{r^r t^{r-1}}.
\]
If \(y_*>4\), equivalently \(t\ge4(r-1)/r\), then the maximum occurs at \(y=4\), and
\[
        \Gamma_r(t)=\phi(4)=\frac{4-t}{4^r}.
\]
The case \(r=1\) follows immediately from \((y-t)_+\le y\).
\end{proof}

The following scalar certificate completes the comparison between the lower
envelopes and the path-side upper envelopes.

\begin{lemma}\label{lem:SN-scalar-certificate}
Let $C_{19}:=\frac1{20}\binom{39}{19}$. Define
\[
        U_1(t):=
        \begin{cases}
        \dfrac32,
        &0\le t\le3/2,\\[2mm]
        \dfrac5{6t},
        &3/2\le t\le2,\\[2mm]
        \dfrac{1701}{640t^3},
        &2\le t\le3,
        \end{cases}
\qquad
\text{and}
\qquad
        U_2(t):=
        \begin{cases}
        C_{19}\dfrac{18^{18}}{19^{19}t^{18}},
        &3\le t\le72/19,\\[2mm]
        C_{19}\dfrac{4-t}{4^{19}},
        &72/19\le t\le4.
        \end{cases}
\]
With \(L_{\rm av}\) and \(L_5\) as in Lemmas~\ref{lem:SN-jensen-envelope} and~\ref{lem:SN-fifth-envelope}, one has
\[
        L_{\rm av}(t)\ge U_1(t)
        \qquad(0\le t\le3),
\]
and
\[
        L_5(t)\ge U_2(t)
        \qquad(3\le t\le4).
\]
\end{lemma}

\begin{proof}
By the piecewise definitions of \(U_1\) and \(U_2\), the assertion is exactly
the following five scalar inequalities:
\[
        \begin{cases}
        L_{\rm av}(t)\ge\dfrac32,
        &0\le t\le3/2,\\[2mm]
        L_{\rm av}(t)\ge\dfrac5{6t},
        &3/2\le t\le2,\\[2mm]
        L_{\rm av}(t)\ge\dfrac{1701}{640t^3},
        &2\le t\le3,\\[2mm]
        L_5(t)\ge C_{19}\dfrac{18^{18}}{19^{19}t^{18}},
        &3\le t\le72/19,\\[2mm]
        L_5(t)\ge C_{19}\dfrac{4-t}{4^{19}},
        &72/19\le t\le4.
        \end{cases}
\]

Using the explicit formula for \(L_{\rm av}\), and clearing positive
denominators, the first three inequalities are equivalent to
\[
        G_0(t):=\frac72-2t\ge0
        \qquad(0\le t\le3/2),
\]
\[
        G_1(t):=3t\left(\frac72-t\right)^2-5\ge0
        \qquad(3/2\le t\le2),
\]
and
\[
        G_2(t):=320t^3\left(\frac72-t\right)^2-1701\ge0
        \qquad(2\le t\le3).
\]

We now clear denominators in the two inequalities involving \(L_5\).  Put
\[
        \widehat L(t):=\frac12A(t)^2B(t)^2+1000A(t)^3.
\]
Then
\[
        L_5(t)=\frac{\widehat L(t)}{B(t)^3K(t)}.
\]
Since \(B(t)>0\), \(K(t)>0\), and \(t>0\) on the relevant intervals, the fourth
inequality is equivalent to
\[
        F_1(t):=
        t^{18}\widehat L(t)
        -
        C_{19}\frac{18^{18}}{19^{19}}B(t)^3K(t)\ge0
        \qquad(3\le t\le72/19).
\]
Similarly, the fifth inequality is equivalent to
\[
        F_2(t):=
        \widehat L(t)
        -
        C_{19}\frac{4-t}{4^{19}}B(t)^3K(t)\ge0
        \qquad(72/19\le t\le4).
\]
Here \(\deg F_1=34\) and \(\deg F_2=16\).

The remaining verification is a finite exact polynomial check.  By the exact
Bernstein verification in Section~\ref{app:scalar-envelope-check}, all Bernstein
coefficients of \(G_0,G_1,G_2\) on their stated intervals are positive; all
Bernstein coefficients of \(F_1\) on \([3,72/19]\) are positive; and all
Bernstein coefficients of \(F_2\) are positive on the four subintervals
\[
        [72/19,74/19],\quad
        [74/19,75/19],\quad
        [75/19,151/38],\quad
        [151/38,4].
\]
Since Bernstein basis functions are non-negative on the underlying interval,
the Bernstein positivity criterion in Section~\ref{app:bernstein-criterion} proves the
non-negativity of these five polynomials on the stated intervals.  Hence all five
scalar inequalities hold.
\end{proof}

We now combine the lower envelopes for the support-neighbor update with the upper envelopes for the path endpoint update on the compact interval \(0\le t\le4\).

\begin{lemma}
\label{lem:SN-compact}
Assume that \(G\) is a \(3\)-separated sparse sun graph whose base cycle has length at least \(6\). Then
\[
        J_{E_{G,u}}(t)\ge J_{E_n^P}(t)
        \qquad(0\le t\le4).
\]
\end{lemma}

\begin{proof}
The lower envelopes from Lemmas~\ref{lem:SN-jensen-envelope} and~\ref{lem:SN-fifth-envelope} give
\[
        J_{E_{G,u}}(t)
        \ge
        \begin{cases}
        L_{\rm av}(t), &0\le t\le3,\\[2mm]
        L_5(t), &3\le t\le4.
        \end{cases}
\]
On the path side, Lemma~\ref{lem:path-moments} gives
\[
        \int_{E_n^P}y\dd y\le\frac32,
        \qquad
        \int_{E_n^P}y^2\dd y\le\frac{10}{3},
        \qquad
        \int_{E_n^P}y^4\dd y\le\frac{126}{5},
        \qquad
        \int_{E_n^P}y^{19}\dd y\le C_{19}=\frac1{20}\binom{39}{19}.
\]
Applying Lemma~\ref{lem:path-moment-envelope} yields
\[
        J_{E_n^P}(t)
        \le
        \begin{cases}
        U_1(t), &0\le t\le3,\\[2mm]
        U_2(t), &3\le t\le4,
        \end{cases}
\]
where \(U_1\) and \(U_2\) are the upper envelopes defined in Lemma~\ref{lem:SN-scalar-certificate}. Therefore, by Lemma~\ref{lem:SN-scalar-certificate}, for \(0\le t\le3\),
\[
        J_{E_{G,u}}(t)
        \ge L_{\rm av}(t)
        \ge U_1(t)
        \ge J_{E_n^P}(t),
\]
while for \(3\le t\le4\),
\[
        J_{E_{G,u}}(t)
        \ge L_5(t)
        \ge U_2(t)
        \ge J_{E_n^P}(t).
\]
This proves the desired comparison on the whole interval \(0\le t\le4\).
\end{proof}

\subsection{The exceptional \texorpdfstring{$C_4$}{C4} case}

It remains to handle the unique terminal case not covered by the compact
comparison above, namely the four-cycle with one attached leaf.

\begin{lemma}
\label{lem:C4}
Let \(G\) be obtained from \(C_4\) by attaching one pendant leaf to a cycle vertex. If \(u\) is a cycle neighbor of this loaded vertex, then Theorem~\ref{thm:SN} holds.
\end{lemma}

\begin{proof}
In this case \(H=G-u\) is isomorphic to \(P_4\).  The nonzero squared singular
values are
\[
        \mu(G)=\left\{\frac{5+\sqrt{17}}2,\frac{5-\sqrt{17}}2\right\},
        \qquad
        \mu(H)=\mu(P_4)=\left\{\frac{3+\sqrt5}2,\frac{3-\sqrt5}2\right\}.
\]
Moreover, $\mu(P_5)=\{3,1\}$. By the definition of the spectral-shift interval set, the update
\(H\rightsquigarrow G\) satisfies
\[
        \mathsf S_t(G)-\mathsf S_t(H)
        =
        2J_{E_{G,u}}(t).
\]
Similarly, for the path endpoint update \(P_4\rightsquigarrow P_5\), we have
\[
        \mathsf S_t(P_5)-\mathsf S_t(P_4)
        =
        2J_{E_5^P}(t).
\]
Therefore
\[
        2\bigl(J_{E_{G,u}}(t)-J_{E_5^P}(t)\bigr)
        =
        \mathsf S_t(G)-\mathsf S_t(H)
        -\mathsf S_t(P_5)+\mathsf S_t(P_4).
\]
Since \(H\cong P_4\), we have
\(\mathsf S_t(H)=\mathsf S_t(P_4)\). Hence 
\[
        2\bigl(J_{E_{G,u}}(t)-J_{E_5^P}(t)\bigr)
        =
        \mathsf S_t(G)-\mathsf S_t(P_5).
\]
Using the squared singular values displayed above, this becomes
\[
        J_{E_{G,u}}(t)-J_{E_5^P}(t)
        =
        \frac12\left[
        \left(\frac{5+\sqrt{17}}2-t\right)_+^2
        +
        \left(\frac{5-\sqrt{17}}2-t\right)_+^2
        -(3-t)_+^2-(1-t)_+^2
        \right].
\]
Set $a:=(5-\sqrt{17})/2$ and $b:=(5+\sqrt{17})/2$. It remains to prove
\[
        \Phi(t):=
        \frac12\left[(b-t)_+^2+(a-t)_+^2-(3-t)_+^2-(1-t)_+^2\right]
        \ge0.
\]
The only breakpoints of this positive-part expression are \(a,1,3,b\), and they satisfy \(0<a<1<3<b\).

If \(0\le t\le a\), then $\Phi(t)=11/2-t>0$. If \(a\le t\le1\), then
\[
        \Phi(t)
        =
        -\frac{t^2}{2}
        +\frac{3-\sqrt{17}}{2}t
        +\frac{1+5\sqrt{17}}{4}.
\]
This quadratic is decreasing on \([a,1]\), and $\Phi(1)=(5+3\sqrt{17})/4>0$. If \(1\le t\le3\), then
\[
        \Phi(t)
        =
        \frac12\left[(b-t)^2-(3-t)^2\right]
        =
        \frac12(b-3)(b+3-2t)\ge0,
\]
because \(b>3\) and \(b+3-2t\ge b-3\) on this interval. If \(3\le t\le b\), then \(\Phi(t)=\frac12(b-t)^2\ge0\), and if
\(t\ge b\), then \(\Phi(t)=0\).  Therefore \(\Phi(t)\ge0\) for every
\(t\ge0\).
\end{proof}

We are now ready to present the following.

\begin{proof}[Proof of Theorem~\ref{thm:SN}]
If the cycle length is \(4\), the result is Lemma~\ref{lem:C4}.  Assume that the cycle length is at least \(6\). By Lemma~\ref{lem:SN-compact},
\[
        J_{E_{G,u}}(t)\ge J_{E_n^P}(t)
        \qquad(0\le t\le4).
\]
For \(t\ge4\), the path endpoint interval set is supported in \([0,4]\), so \(J_{E_n^P}(t)=0\), while \(J_{E_{G,u}}(t)\ge0\).  Hence the inequality holds for every \(t\ge0\).
\end{proof}

\section{Completion of the proof}\label{sec:completion}

We first complete the proof of the bipartite stop-loss comparison.

\begin{proof}[Proof of Theorem~\ref{thm:bipartite-stoploss}]
Suppose, for contradiction, that the second-order stop-loss comparison fails. Choose a connected bipartite counterexample \(G\) of minimum order, and fix a threshold \(t_0\ge0\) witnessing the failure:
\[
        \mathsf S_{t_0}(G)<\mathsf S_{t_0}(P_n),
        \qquad n=|V(G)|.
\]
By the minimality of \(G\), every connected bipartite graph of order strictly smaller than \(n\) satisfies the comparison against the path of the same order for every threshold.  In particular, every connected component produced by any vertex deletion from \(G\) satisfies the comparison for all thresholds.

The tree case is excluded by Corollary~\ref{cor:tree-stoploss-path}.  Hence \(G\) is not a tree, and Proposition~\ref{prop:sl2-reduction-suns} applies: \(G\) is a sparse sun graph whose loaded cycle vertices are \(3\)-separated.

If this sparse sun has no pendant leaves, then it is just an even cycle, which is excluded by Theorem~\ref{thm:cycle-sl2}. Thus \(G\) has at least one pendant leaf. Choose a loaded cycle vertex and an unloaded cycle neighbor \(u\), and put \(H=G-u\).  Deleting \(u\) breaks the base cycle into a path, while all pendant leaves remain attached to this path.  Hence \(H\) is a tree on \(n-1\) vertices. By Corollary~\ref{cor:tree-stoploss-path},
\[
        \mathsf S_{t_0}(H)\ge \mathsf S_{t_0}(P_{n-1}).
\]

By Theorem~\ref{thm:SN},
\[
        J_{E_{G,u}}(s)\ge J_{E_n^P}(s)
        \qquad(s\ge0).
\]
Evaluating at \(s=t_0\) and using the defining relation between spectral-shift intervals and second-order stop-loss increments, we obtain
\[
        \mathsf S_{t_0}(G)-\mathsf S_{t_0}(H)
        =
        2J_{E_{G,u}}(t_0)
        \ge
        2J_{E_n^P}(t_0)
        =
        \mathsf S_{t_0}(P_n)-\mathsf S_{t_0}(P_{n-1}).
\]
Adding this inequality to
\(\mathsf S_{t_0}(H)\ge \mathsf S_{t_0}(P_{n-1})\) gives
\[
        \mathsf S_{t_0}(G)\ge \mathsf S_{t_0}(P_n),
\]
contradicting the choice of \(t_0\).  Therefore the second-order stop-loss
comparison holds for every connected bipartite graph and every threshold
\(t\ge0\).
\end{proof}

We now prove the main \(p\)-energy theorem.

\begin{proof}[Proof of Theorem~\ref{thm:pgeq2}]
If \(n=1\), the assertion is trivial. Hence assume \(n\ge2\). 

The case \(p=2\) follows from
\[
        \mathcal E_2(G)=\operatorname{tr}A(G)^2=2|E(G)|
        \ge 2(n-1)=\mathcal E_2(P_n),
\]
because \(G\) is connected.

We next record the equality case for \(p=4\).  For every simple graph,
\[
        \mathcal E_4(G)=\operatorname{tr}A(G)^4
        =
        2m+4\sum_{v\in V(G)}\binom{d_v}{2}+8C_4(G),
\]
where \(m=|E(G)|\), \(d_v=d_G(v)\), and \(C_4(G)\) is the number of
\(4\)-cycles in \(G\).  For the path,
\[
        \mathcal E_4(P_n)
        =
        2(n-1)+4(n-2)
        =
        6n-10.
\]
If \(m\ge n\), then, since \(G\) is connected and \(n\ge2\), every vertex has
positive degree and
\[
        \sum_v\binom{d_v}{2}\ge \sum_v(d_v-1)=2m-n.
\]
Hence
\[
        \mathcal E_4(G)
        \ge
        2m+4(2m-n)
        =
        10m-4n
        \ge
        6n
        >
        6n-10
        =
        \mathcal E_4(P_n).
\]
If \(m=n-1\), then \(G\) is a tree.  In this case
\[
        \sum_v\binom{d_v}{2}\ge \sum_v(d_v-1)=n-2,
\]
with equality if and only if every vertex has degree at most \(2\).  A
connected tree with maximum degree at most \(2\) is a path.  Thus
\[
        \mathcal E_4(G)=\mathcal E_4(P_n)
        \quad\Longleftrightarrow\quad
        G\cong P_n.
\]
This proves the theorem, including the equality case, when \(p=4\).

Now assume \(p>2\).  We first prove the strict equality statement in the
connected bipartite case.

Let \(G\) be connected and bipartite.  If \(2<p<4\), then by
Theorem~\ref{thm:r1-main-comparison}, we know that $R_1(G;x)\ge R_1(P_n;x)$ for every $x>0$. Theorem~\ref{thm:R-implies-p} gives
\[
        \mathcal E_p(G)\ge \mathcal E_p(P_n).
\]
Suppose equality holds.  Put \(\alpha=p/2\in(1,2)\).  By the Mellin
representation used in Theorem~\ref{thm:R-implies-p},
\[
        0
        =
        \sum_i\mu_i(G)^\alpha-\sum_i\mu_i(P_n)^\alpha
        =
        c_\alpha\int_0^\infty
        \bigl(R_1(G;x)-R_1(P_n;x)\bigr)x^{-\alpha-1}\,dx .
\]
The integrand is non-negative and continuous in \(x>0\), while the weight is
strictly positive.  Therefore
\[
        R_1(G;x)=R_1(P_n;x)\qquad(x>0).
\]
Since
\[
        r_1(y)=y-\log(1+y)=\frac{y^2}{2}+O(y^3)
        \qquad(y\to0),
\]
we have
\[
        R_1(G;x)
        =
        \frac{x^2}{2}\sum_i\mu_i(G)^2+O(x^3),
        \qquad
        R_1(P_n;x)
        =
        \frac{x^2}{2}\sum_i\mu_i(P_n)^2+O(x^3).
\]
Hence $\sum_i\mu_i(G)^2=\sum_i\mu_i(P_n)^2$. Because both graphs are bipartite,
\[
        \mathcal E_4(G)=2\sum_i\mu_i(G)^2,
        \qquad
        \mathcal E_4(P_n)=2\sum_i\mu_i(P_n)^2.
\]
Thus \(\mathcal E_4(G)=\mathcal E_4(P_n)\), and the \(p=4\) equality case
proved above gives \(G\cong P_n\).

If \(p>4\), put \(\alpha=p/2>2\).  By Theorem~\ref{thm:bipartite-stoploss},
\[
        \mathsf S_t(G)\ge \mathsf S_t(P_n)
        \qquad(t\ge0).
\]
Using the integral representation
\[
        u^\alpha
        =
        \frac{\alpha(\alpha-1)(\alpha-2)}2
        \int_0^\infty (u-t)_+^2 t^{\alpha-3}\,dt
        \qquad(u\ge0),
\]
we obtain
\[
\begin{aligned}
\sum_i\mu_i(G)^\alpha-\sum_i\mu_i(P_n)^\alpha
&=
\frac{\alpha(\alpha-1)(\alpha-2)}2
\int_0^\infty
\bigl(\mathsf S_t(G)-\mathsf S_t(P_n)\bigr)t^{\alpha-3}\,dt
\ge0.
\end{aligned}
\]
Therefore
\[
        \mathcal E_p(G)\ge\mathcal E_p(P_n).
\]
If \(G\not\cong P_n\), then the \(p=4\) equality case gives $\mathcal E_4(G)>\mathcal E_4(P_n)$.

Equivalently,
\[
        \mathsf S_0(G)-\mathsf S_0(P_n)
        =
        \sum_i\mu_i(G)^2-\sum_i\mu_i(P_n)^2
        =
        \frac12\bigl(\mathcal E_4(G)-\mathcal E_4(P_n)\bigr)
        >0.
\]
By continuity of \(t\mapsto \mathsf S_t(G)-\mathsf S_t(P_n)\), the difference
is positive on some interval \((0,\varepsilon)\).  Since
\(t^{\alpha-3}>0\) on \((0,\varepsilon)\), the above integral is strictly
positive.  Hence
\[
        \mathcal E_p(G)>\mathcal E_p(P_n)
        \qquad(p>4)
\]
unless \(G\cong P_n\).

Combining the cases \(2<p<4\), \(p=4\), and \(p>4\), we have proved that, for
every connected bipartite graph \(G\) on \(n\) vertices and every \(p>2\),
\[
        \mathcal E_p(G)\ge\mathcal E_p(P_n),
\]
with equality if and only if \(G\cong P_n\).

Finally let \(G\) be an arbitrary connected graph on \(n\) vertices, and let \(p>2\).  By Lemma~\ref{lem:bipartite-reduction}, there exists a connected bipartite spanning subgraph \(H\) of \(G\) such that $\mathcal E_p(H)\le \mathcal E_p(G)$. The connected bipartite case applied to \(H\) gives $\mathcal E_p(H)\ge \mathcal E_p(P_n)$, and therefore
\[
        \mathcal E_p(G)\ge \mathcal E_p(P_n).
\]

It remains to identify equality. Suppose $\mathcal E_p(G)=\mathcal E_p(P_n)$. Then both inequalities above must be equalities.  In particular,
\[
        \mathcal E_p(H)=\mathcal E_p(P_n),
\]
so the bipartite equality case gives \(H\cong P_n\).

We now use equality in the bipartite reduction.  In the proof of
Lemma~\ref{lem:bipartite-reduction}, one has
\[
        A(H)=\frac{A(G)+(-D A(G)D)}2.
\]

By the strict convexity of finite-dimensional Schatten \(p\)-spaces for \(1<p<\infty\)~\cite[Corollary~1]{AziznejadUnser2021}, the Schatten norm has strictly convex unit ball. Consequently, for \(M\ne N\),
\[
        \left\|\frac{M+N}{2}\right\|_{S_p}^p
        <
        \frac{\|M\|_{S_p}^p+\|N\|_{S_p}^p}{2}.
\]
Indeed, if \(\|M\|_{S_p}=\|N\|_{S_p}\), this follows directly from strict
convexity of the unit ball; if \(\|M\|_{S_p}\ne\|N\|_{S_p}\), it follows from
the triangle inequality and the strict convexity of \(x\mapsto x^p\). Hence equality in
\[
        \mathcal E_p(H)
        =
        \left\|
        \frac{A(G)+(-D A(G)D)}2
        \right\|_{S_p}^p
        \le
        \frac{\|A(G)\|_{S_p}^p+\|-D A(G)D\|_{S_p}^p}{2}
        =
        \mathcal E_p(G)
\]
forces $A(G)=-D A(G)D$. Writing \(A(G)\) with respect to the bipartition
\(V(G)=X\sqcup Y\) used in the construction of \(H\), we have
\[
        A(G)=
        \begin{pmatrix}
        A_X&B\\
        B^\top&A_Y
        \end{pmatrix}.
\]
The identity \(A(G)=-DA(G)D\) forces \(A_X=A_Y=0\).  Hence \(G\) has no edges inside \(X\) or inside \(Y\). Therefore the bipartite spanning subgraph \(H\) constructed in Lemma~\ref{lem:bipartite-reduction} contains all edges of \(G\), and so \(G=H\). Since \(H\cong P_n\), we get \(G\cong P_n\).

The reverse implication is immediate: if \(G\cong P_n\), then equality holds for every \(p\).  This completes the proof.
\end{proof}

\appendix

\section{Finite certificates for the path-splicing low range}
\label{app:a2sl2-certificates}

This appendix gives the finite verification used in the low-range part of Theorem~\ref{thm:a2-sl2}.  The verification is not a sampling argument.  It is
a rigorous interval computation with rational subdivision endpoints and outward-rounded ball arithmetic, implemented in Arb~\cite{JohanssonArb}.
The computations are carried out in SageMath~\cite{SageMath10.8}, using its
built-in \texttt{RealBallField} class as the Arb interface to Arb~\cite{JohanssonArb}.

We use the notation from the proof of Theorem~\ref{thm:a2-sl2}. For
\(1/6\le\rho\le1/2\),
\[
        K_{p,q}(\rho)=M(\rho)+J_{p+q-1}(\rho)-J_p(\rho)-J_q(\rho),
\]
and
\[
        D_p(\rho)=M(\rho)-J_p(\rho).
\]

\begin{lemma}\label{lem:a2-finite-certificates}
The following inequalities hold.
\begin{enumerate}[label=(\roman*)]
\item For each \(p=2,\ldots,6\) and every \(1/6\le\rho\le1/2\),
\[
        D_p(\rho)-
        \frac{17\sqrt3\pi^2}{27}\left(\frac1{7^2}+\frac1{(p+6)^2}\right)\ge 10^{-4}.
\]

\item For each \(2\le p\le q\le6\) and every \(1/6\le\rho\le1/2\),
\[
        K_{p,q}(\rho)\ge 10^{-4}.
\]
\end{enumerate}
\end{lemma}

\begin{proof}
The proof is a finite rigorous interval computation.  We describe the
computation precisely.

For the strip inequalities in (i), we use
\[
        D_p(\rho)=
        \int_0^\rho
        \bigl(\one_{u<1/6}+u-\{pu\}\bigr)w_\rho(u)\,du,
\]
where $w_\rho(u)=16\pi\sin(2\pi u)\bigl(\cos(2\pi u)-\cos(2\pi\rho)\bigr)$. On each interval cut out by the rational breakpoints
\[
        0,\quad \frac16,\quad \frac12,\quad \frac{k}{p}
        \quad(0\le k\le p),
\]
the coefficient $\one_{u<1/6}+u-\{pu\}$ is affine in \(u\).  Therefore \(D_p(\rho)\) is evaluated on each such cell using the elementary antiderivatives of
\[
        (\alpha+\beta u)\sin(2\pi u)\cos(2\pi u),
        \qquad
        (\alpha+\beta u)\sin(2\pi u),
\]
with \(\rho\) treated as an interval variable.

For the small-box inequalities in (ii), the active set is constant between consecutive rational breakpoints
\[
        \frac16,\quad \frac12,\quad \frac{i}{p},\quad \frac{j}{q},
        \quad \frac{k}{p+q-1}.
\]
More explicitly, the signed atoms are
\[
        (0,+1),\qquad (1/6,-1),\qquad
        (i/p,+1),\qquad (j/q,+1),\qquad
        (k/(p+q-1),-1),
\]
restricted to those atoms with \(\alpha\le \rho\).  On any subinterval not containing a breakpoint, this active set is fixed, and one has the exact formula
\[
        K_{p,q}(\rho)
        =
        4\sum_{\alpha\in\mathcal A} \omega_\alpha
        \bigl(\cos(2\pi\alpha)-\cos(2\pi\rho)\bigr)^2,
\]
where \(\mathcal A\) is the fixed active set on that subinterval and \(\omega_\alpha\in\{1,-1\}\) is the corresponding sign.  Endpoint conventions are harmless, since a newly active atom contributes zero at the breakpoint.

The interval certification is implemented in the SageMath script
\[
        \texttt{core/verify\_path\_splicing.sage}
\]
available in the public GitHub repository~\cite{Git}.  The script uses exact
rational subdivision endpoints and Sage's \texttt{RealBallField(256)}, which
is an interface to Arb ball arithmetic.  On each box the program computes a
ball enclosure for the residual
\[
        \text{left-hand side}-\text{right-hand side}-10^{-4}.
\]
The box is accepted only if the lower endpoint of this residual ball is
strictly positive.  Otherwise the box is bisected.  Thus every accepted box
certifies the required inequality on the whole box, not merely at sample
points.  The corresponding output log is
\[
        \texttt{core/verify\_path\_splicing\_output.txt}.
\]
The script was run with SageMath~10.8.
\end{proof}

\section{Elementary interval domination inequalities}
\label{app:interval-domination}

For convenience, we restate Lemma~\ref{lem:interval-domination} here.

\medskip

\noindent\textbf{Restatement of Lemma~\ref{lem:interval-domination}.}
Let \(d,r\) be integers with \(d\ge r+2\) and \(r\ge1\).  Put
\[
        a=\frac{d-r+1}{d},\qquad B=a+d.
\]
Then for every \(t\ge0\),
\[
        \mathcal I_{[a,B]}(t)\ge r\,\mathcal I_{[3,4]}(t).
\]
Moreover, if \(d=r+2\), equivalently \(a=3/d\) and \(B=d+3/d\), then for every
\(t\ge0\),
\[
        \mathcal I_{[a,B]}(t)
        \ge
        \mathcal I_{[2,4]}(t)+(d-3)\mathcal I_{[3,4]}(t).
\]

\begin{proof}[Proof of Lemma~\ref{lem:interval-domination}]

We shall repeatedly use the elementary formula
\[
        \mathcal I_{[u,v]}(t)=
        \begin{cases}
        0, & t\ge v,\\[1mm]
        \dfrac12(v-t)^2, & u\le t\le v,\\[2mm]
        (v-u)\left(\dfrac{u+v}{2}-t\right), & t\le u.
        \end{cases}
\]

We first prove the main estimate and then the moreover clause.

For the main estimate, set \(m=d-r\ge2\). Then \(a=(m+1)/d\) and
\(B=d+a\ge d+\frac{3}{d}\ge4\). Let
\[
        H(t)=\mathcal I_{[a,B]}(t)-r\mathcal I_{[3,4]}(t).
\]
For \(t\ge4\), $\mathcal I_{[3,4]}(t) = 0$ and \(H(t)\ge0\).  For
\(3\le t\le4\),
\[
        H(t)=\frac12\{(B-t)^2-r(4-t)^2\}.
\]
The ratio \((B-t)/(4-t)\) is increasing with respect to $t$ on \([3,4)\). Hence it is enough to check $t=3$ and we need to prove \(B-3 \geq \sqrt r\).  Since \(m\ge2\),
\[
        B-3=r+m-3+\frac{m+1}{r+m}
        \ge r-1+\frac{3}{r+2},
\]
and
\[
        \left(r-1+\frac{3}{r+2}\right)^2-r
        =\frac{(r-1)(r^3+2r^2+r-1)}{(r+2)^2}\ge0 .
\]
Thus \(H\ge0\) on \([3,4]\).

For \(a\le t\le3\),
\[
        H(t)=\frac12(B-t)^2-r\left(\frac72-t\right).
\]
This is a convex quadratic with unique minimizer at \(t=a+m\). If \(m\ge3\), then \(a+m>3\), so the minimum on \([a,3]\) is attained at \(t=3\), already handled. If \(m=2\), then the critical point lies in the interval and
\[
        H(a+m)=r\left(\frac r2-\frac32+\frac{3}{r+2}\right)\ge0,
\]
with equality only at \(r=1\).  Hence \(H\ge0\) on \([a,3]\).

For \(t\le a\), the function \(H\) is affine with slope \(r-d<0\). Therefore its minimum on \(( -\infty,a]\) is attained at \(t=a\), which was covered in the previous paragraph. This proves the main estimate.

For the moreover clause, assume in addition that \(d=r+2\). Then \(a=3/d\leq 1\) and \(B=d+a\), so we may define
\[
        H_d(t)=\mathcal I_{[a,B]}(t)-\mathcal I_{[2,4]}(t)
        -(d-3)\mathcal I_{[3,4]}(t).
\]
For \(t\ge4\), $\mathcal I_{[2,4]}(t)
        +(d-3)\mathcal I_{[3,4]}(t) = 0$, so \(H_d(t)\ge0\).  For
\(3\le t\le4\),
\[
        H_d(t)=\frac12\{(B-t)^2-(d-2)(4-t)^2\}.
\]
As above, \((B-t)/(4-t)\) is increasing on \([3,4)\), so it suffices to check the endpoint \(t=3\), namely \(B-3\ge\sqrt{d-2}\). This follows from
\[
        \left(d-3+\frac3d\right)^2-(d-2)
        =\frac{(d-3)(d^3-4d^2+5d-3)}{d^2}\ge0.
\]
For \(2\le t\le3\), $H_d(t)$ is affine and direct differentiation gives
\[
        H_d'(t)=1-a\ge0.
\]
Thus the minimum on this interval is at \(t=2\), where
\[
        H_d(2)=\frac12(B-2)^2-2-\frac32(d-3)
        =\frac{(d-3)(d^3-4d^2+3d-3)}{2d^2}\ge0.
\]
For \(a\le t\le2\), we have
$
        H_d'(t)=t-a-1.
$
Since \(a=3/d\le1\), the point \(t=a+1\) lies in \([a,2]\) and \(H_d\) attains its minimum on \([a,2]\) at \(t=a+1\). At this point,
\[
        H_d(a+1)=\frac{(d-3)(d^2-4d+2)}{2d}\ge0,
\]
because the expression is \(0\) when \(d=3\), and \(d^2-4d+2>0\) for
\(d\ge4\). Hence \(H_d(t)\ge0\) for all \(t\in[a,2]\).

Finally, for \(t\le a\), \(H_d\) is affine with negative slope, so its minimum
on \((-\infty,a]\) is attained at \(t=a\). Since \(t=a\) belongs to the
previous interval, \(H_d(a)\ge0\). This proves the claim.
\end{proof}

\section{Exact verification of finite auxiliary inequalities}
\label{app:verification}

This appendix records the finite algebraic checks used above.  All computations are over the rational field \(\QQ\).  No floating-point approximation is used.

\subsection{Bernstein positivity criterion}
\label{app:bernstein-criterion}

We use the following elementary Bernstein-basis certificate for polynomial
non-negativity on a rational interval.  Let \(P\in\QQ[t]\), let
\(\alpha<\beta\) be rational, and let \(m\ge\deg P\).  The degree-\(m\)
Bernstein coefficients of \(P\) on \([\alpha,\beta]\) are the rational numbers
\(b_0,\ldots,b_m\) defined by
\[
        P((1-x)\alpha+x\beta)
        =
        \sum_{k=0}^m
        b_k\binom{m}{k}x^k(1-x)^{m-k}.
\]
If \(b_k\ge0\) for every \(k\), then \(P(t)\ge0\) for every
\(t\in[\alpha,\beta]\).  Indeed, for such \(t\) we may write
\(t=(1-x)\alpha+x\beta\) with \(x=(t-\alpha)/(\beta-\alpha)\in[0,1]\), and
each Bernstein basis function
\[
        \binom{m}{k}x^k(1-x)^{m-k}
\]
is non-negative on \([0,1]\).

For reproducibility, we compute the coefficients by the standard conversion formula from the power basis to the Bernstein basis; see, for example, \cite[Eq.~(2)]{CargoShisha1966}. If $P((1-x)\alpha+x\beta)=\sum_{j=0}^m c_jx^j$, where the coefficients are padded by \(c_j=0\) when necessary, then the degree-\(m\) Bernstein coefficients are
\[
        b_k
        =
        \sum_{j=0}^k
        c_j\frac{\binom{k}{j}}{\binom{m}{j}},
        \qquad 0\le k\le m .
\]
This is the formula used in the exact rational checks below.

\subsection{Scalar envelope certificates}
\label{app:scalar-envelope-check}

The polynomials used in Lemma~\ref{lem:SN-scalar-certificate} are
\[
        G_0(t)=\frac72-2t,
\qquad
        G_1(t)=3t\left(\frac72-t\right)^2-5,
\qquad
        G_2(t)=320t^3\left(\frac72-t\right)^2-1701,
\]
\[
        F_1(t)=t^{18}\widehat L(t)
        -C_{19}\frac{18^{18}}{19^{19}}B(t)^3K(t),
\qquad
\text{and}
\qquad
        F_2(t)=\widehat L(t)
        -C_{19}\frac{4-t}{4^{19}}B(t)^3K(t),
\]
where
\[
        C_{19}:=\frac1{20}\binom{39}{19},
\qquad
        A(t)=1812-7t^4,
\qquad
        B(t)=4001-4t^4,
\]
\[
        K(t)=5(t^3+5t^2+25t+125),
\qquad
\text{and}
\qquad
        \widehat L(t)=\frac12A(t)^2B(t)^2+1000A(t)^3.
\]
The following table records exact lower bounds for the Bernstein coefficients on the indicated intervals.

\[
\begin{array}{c|c|c|c}
\text{polynomial} & \text{interval} & \text{degree} &
\text{lower bound for every Bernstein coefficient}\\ \hline
G_0 & [0,3/2] & 1 & 1/2\\
G_1 & [3/2,2] & 3 & 17/2\\
G_2 & [2,3] & 5 & 459\\
F_1 & [3,72/19] & 34 & 10^{20}\\
F_2 & [72/19,74/19] & 16 & 10^{11}\\
F_2 & [74/19,75/19] & 16 & 10^{10}\\
F_2 & [75/19,151/38] & 16 & 10^9\\
F_2 & [151/38,4] & 16 & 10^8
\end{array}
\]
They are verified by exact rational arithmetic in the SageMath script
\[
\texttt{core/verify\_scalar\_envelope.sage}
\]
in the public GitHub repository~\cite{Git}.  The corresponding output log is
\[
\texttt{core/verify\_scalar\_envelope\_output.txt}.
\]
The script computes the Bernstein coefficients over \(\QQ\) using the
conversion formula above and checks that the minimum coefficient in each row
of the table is at least the stated lower bound.

\subsection{Local fifth-moment and envelope algebra checks}
\label{app:local-moment-check}

The exact polynomial and rational identities used in Lemmas~\ref{lem:local-data} and~\ref{lem:SN-fifth-envelope} are verified as follows.  Besides
the stable \(N=7\) moment polynomial \(q_5\), we also check the expansion of
\(P_t(1-z)\), the numerator identity underlying the derivative formula for
\(h(t)=A(t)/B(t)\), the exact value \(h(3)=1245/3677\), and the rational
lower bound used for \(D(t)\).

\medskip
\noindent
\textbf{Exact SageMath algebra check.}
The exact rational verification is implemented in the SageMath script
\[
\texttt{core/verify\_local\_moment.sage}
\]
in the public GitHub repository~\cite{Git}.  The corresponding output log is
\[
\texttt{core/verify\_local\_moment\_output.txt}.
\]
The script checks the stable \(N=7\) moment polynomial \(q_5\), the
coefficientwise non-negativity of the differences for \(3\le N\le6\), the
expansion of \(P_t(1-z)\), the numerator identity for the derivative of
\(h(t)=A(t)/B(t)\), the value \(h(3)=1245/3677\), and the rational lower
bound used for \(D(t)\).  All computations are performed over \(\QQ\), with
no floating-point arithmetic.

\section*{Acknowledgments and AI disclosure}

The authors thank Clive Elphick for helpful comments and suggestions. Q.~Tang was partially supported by the National Key Research and Development Program of China under grant 2023YFA1010203. Y.~Liu was partially supported by Beijing Natural Science Foundation under grant QY26408.

The authors used AI tools during the preparation of this manuscript. In particular, an AI image-generation tool generated Figures~\ref{fig:tree-transform} and~\ref{fig:tree-shift-interpolation}. AI assistance partly drafted the code included in \cite{Git}, and the authors then checked, edited, and verified it. The mathematical ideas, proof strategy, and final arguments are due to the authors. The authors used AI tools only as an aid for drafting, coding, and presentation. The authors are fully responsible for all statements, proofs, computations, and figures in the paper.

\end{document}